%% file: master-thesis.tex
\documentclass[10pt,a4paper,twoside]{memoir}

\title{Generating functions in symplectic and contact geometry}

\usepackage{amsmath, amsfonts, mathtools, amsthm}
\usepackage{textcomp}
\usepackage[framemethod=tikz]{mdframed}
\usepackage{pifont}
\usepackage[pdfborder={0 0 0}, hypertexnames=false]{hyperref}
\usepackage{graphicx}
\usepackage{float}
\usepackage{units}
\usepackage{xcolor}
\usepackage{braket}
\usepackage[backend=bibtex,url=false,style=alphabetic]{biblatex}
\usepackage{booktabs}
\usepackage{relsize}
\usepackage{titlesec}
\usepackage{multicol}
\usepackage[all]{xy}
\usepackage{enumitem}
\setenumerate[0]{label=(\roman*)}

\renewbibmacro{in:}{}

\newsubfloat{figure}

\renewcommand{\thetable}{\Roman{table}}
\captiondelim{. }

\swapnumbers
\theoremstyle{definition}\newmdtheoremenv[
  linewidth=2,
  hidealllines=true,
  leftline=true,
  innerleftmargin=7pt,
  innerrightmargin=0,
  innertopmargin=-7pt,
  innerbottommargin=-1pt,
  splittopskip=0,
]{definition}{Definition}[chapter]

\newtheorem{remark}[definition]{}
\newmdtheoremenv[
  linewidth=2,
  hidealllines=true,
  leftline=true,
  innerleftmargin=7pt,
  innerrightmargin=0,
  innertopmargin=-7pt,
  innerbottommargin=-1pt,
  splittopskip=0,
]{proposition}[definition]{Proposition}
\newmdtheoremenv[
  linewidth=2,
  hidealllines=true,
  leftline=true,
  innerleftmargin=7pt,
  innerrightmargin=0,
  innertopmargin=-7pt,
  innerbottommargin=-1pt,
  splittopskip=0,
]{corollary}[definition]{Corollary}
\newmdtheoremenv[
  linewidth=2,
  hidealllines=true,
  leftline=true,
  innerleftmargin=7pt,
  innerrightmargin=0,
  innertopmargin=-7pt,
  innerbottommargin=-1pt,
  splittopskip=0,
]{theorem}[definition]{Theorem}
\newmdtheoremenv[
  linewidth=2,
  hidealllines=true,
  leftline=true,
  innerleftmargin=7pt,
  innerrightmargin=0,
  innertopmargin=-7pt,
  innerbottommargin=-1pt,
  splittopskip=0,
]{lemma}[definition]{Lemma}
\newmdtheoremenv[
  linewidth=2,
  hidealllines=true,
  leftline=true,
  innerleftmargin=7pt,
  innerrightmargin=0,
  innertopmargin=-7pt,
  innerbottommargin=-1pt,
  splittopskip=0,
]{definiprop}[definition]{Def. \& Prop}
\newmdtheoremenv[
  linewidth=2,
  hidealllines=true,
  leftline=true,
  innerleftmargin=7pt,
  innerrightmargin=0,
  innertopmargin=-7pt,
  innerbottommargin=-1pt,
  splittopskip=0,
]{conjecture}[definition]{Conjecture}

\theoremstyle{definition}

\input{macros.tex}

\setulmarginsandblock{3.6cm}{3.7cm}{*}
\setlrmarginsandblock{3.9cm}{2.7cm}{*}
\checkandfixthelayout

\chapterstyle{dash}
\titleformat{\section}{\normalfont\sffamily\large\bfseries}{\thesection}{1.0em}{}
\titleformat{\subsection}{\normalfont\sffamily}{\thesubsection}{1.0em}{}
\titleformat{\subsubsection}{\sffamily\itshape}{\thesubsubsection}{1em}{}

\captionnamefont{\footnotesize}
\captiontitlefont{\footnotesize}

\setsecnumdepth{subsection}
\counterwithout{figure}{chapter}
\counterwithout{table}{chapter}
\counterwithout{equation}{chapter}
\counterwithout{footnote}{chapter}

\makeatletter
\counterwithout{figure}{chapter}
\counterwithout{table}{chapter}
\renewcommand\@memfront@floats{}
\renewcommand\@memmain@floats{}
\renewcommand\@memback@floats{}
\makeatletter

\allowdisplaybreaks

\begin{document}

\pagestyle{empty}
\begin{adjustwidth}{-0.5cm}{0.7cm}
\begin{center}
	\includegraphics[height=1.7cm]{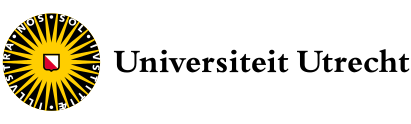}\\[35mm]
  {\sffamily
  {\fontsize{18pt}{18pt}\selectfont Generating functions in\\[3mm] symplectic and contact geometry \medskip}
	\\[20mm]

	\textit{Master's thesis by}\\[4mm]
	{\large Aaron Gootjes-Dreesbach}\\[4mm]
	\vfill

	\setlength{\tabcolsep}{0.01\linewidth}
	\begin{tabular}{@{} >{\raggedleft}p{0.47\linewidth} p{0.48\linewidth} @{}}
	{ \textbf{submitted to}} & Department of Mathematics \\
	& Universiteit Utrecht\\[3mm]
	{\textbf{supervised by}} & Dr. Fabian Ziltener \\[3mm]
	{\textbf{second reader}} & Prof. Dr. Marius Crainic \\
	\end{tabular}\\[7mm]
{{September 2020}}}
  \setlength{\tabcolsep}{5pt}
\end{center}
\end{adjustwidth}
\newpage
\ 
\vfill
  \textit{\small{ Updated version from \today.}}

\frontmatter
\begin{center}
	\bfseries\normalsize \textsf{Abstract}
\end{center}

\begin{abstract}
\footnotesize
A \textit{translated point} of a contactomorphism $\phi$ on a contact manifold with contact form $\alpha$ is a point $p$ where $\alpha$ is preserved under $\phi$ and whose image under $\phi$ lies in the same Reeb trajectory. They were introduced as a contact analogon for fixed points of Hamiltonian diffeomorphisms by Sheila Sandon in~\cite{San12} and can be understood as a special case of leafwise fixed points. A contact version of the non-degenerate Arnol'd conjecture on spheres was established in~\cite{San13} using a generating function approach. It turns out  that Sandon's proof only works under the assumption that there exists a generating function whose sublevel set at zero has nontrivial homology. This thesis proves the result under this additional assumption and fills gaps in other parts of Sandon's argument.
\end{abstract}
\ \\

\begin{center}
	\bfseries\normalsize \textsf{Contents}
\end{center}
\renewcommand{\contentsname}{}
\vspace{-3.5cm}
\tableofcontents*
\pagestyle{empty}
\mainmatter
\renewcommand{\thetable}{\Roman{table}}
\pagestyle{plain}

\include{introduction}
\include{background}

\include{techniques}

\include{sandon}

\appendix

\include{appendix}

\backmatter
\pagestyle{plain}
\chapter{References}
\printbibliography[heading=none]
\newpage

\end{document}

%% file: macros.tex
\newcommand*{\I}{\mathrm{i}}
\newcommand{\R}{\mathbb{R}}
\newcommand{\N}{\mathbb{N}}
\newcommand{\Z}{\mathbb{Z}}
\newcommand{\C}{\mathbb{C}}
\newcommand{\ind}{\operatorname{ind}}

\newcommand{\Id}{\operatorname{Id}}

\newcommand{\startsubpart}[1]{\noindent\textbf{#1.}}

\newcommand{\symp}{\mathcal{S}}
\newcommand{\lift}{\tilde{\mathcal{S}}}

%% file: introduction.tex
\chapter{Introduction and Main Result}

A critical development in 19th century theoretical physics was the reformulation of classical mechanics as \textit{Hamiltonian mechanics}. This more abstract approach allows for a deeper structural understanding of physical systems and paved the way for most of modern physics. Here, a system is encoded in terms of two objects, a \textit{configuration space} and an \textit{energy function.} The configuration space $N$ contains all possible configurations the system can be in at any given instant in time. It gives rise to the larger \textit{phase space} $M$ of all possible states: A point in phase space is a pair of a point in configuration space (\textit{generalized position}) and the infinitesimal change of the position that is in progress at that time (\textit{generalized momentum})\footnote{We only need to look at first-order derivatives here because Newton's equations of motion are differential equations of second order: All higher derivatives are redundant since they are determined by the configuration itself and first derivatives.}.
As an example, consider the pendulum depicted in Fig.~\ref{fig:pendulum}. An energy function $H:M\times\R\to\R$ yields for every state $x$ at a given time $t$ the total energy $H_t(x)$. A crucial physical insight is that $H$ suffices to completely determine the dynamics of the system.\medskip

\begin{figure}[t]
\centering
\includegraphics[width=0.5\textwidth]{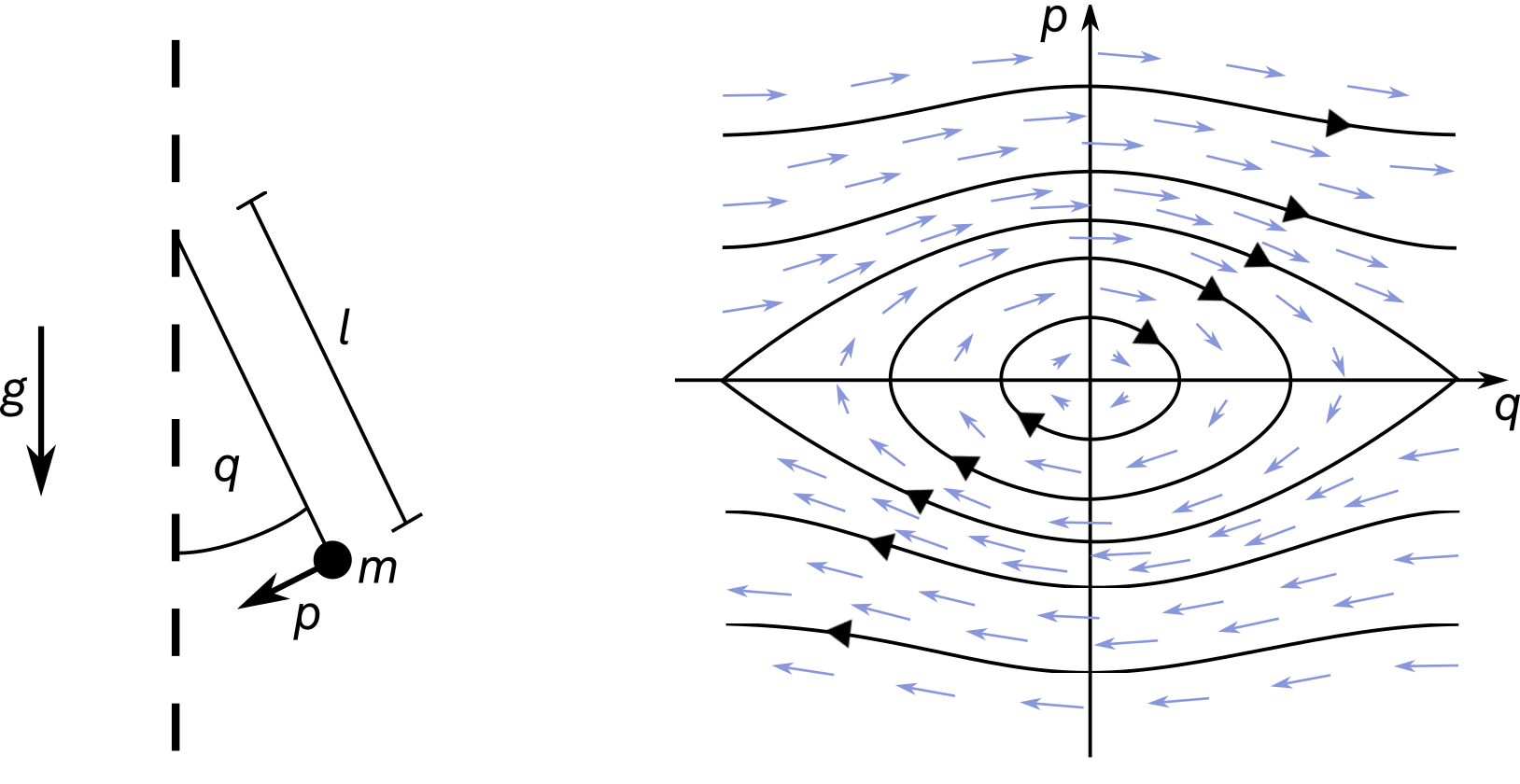}
	\caption{Consider a pendulum of length $l$ and mass $m$ in a gravitational field $g$. Its configuration space is given by the set $N=S^1$ of all angles $q$ to the vertical. The phase space is $M=S^1\times\R\simeq T^*N$, and its elements $(q,p)$ completely determine the state by also specifying the momentum $p$ of the pendulum. The total energy $H(p,q)=p^2/2m-lg\sin q$ generates the Hamiltonian vector field depicted on the right. Following its flow gives exactly the trajectories of the system, which are indicated for some initial conditions.}\label{fig:pendulum}
\end{figure}

To understand how this works mathematically, we need the framework of \textit{symplectic geometry.} Phase space is represented as a differentiable manifold $M$ that carries as additional structure a closed and non-degenerate 2-form $\omega$. This \textit{symplectic form} allows a function $H:M\times\R\to\R$ to uniquely generate a time-dependent vector field $X_t$ via the \textit{Hamilton equations}
$$\omega(X_t, \cdot) = -dH_t.$$
This vector field can be integrated to the \textit{Hamiltonian flow} $\phi^H_t$ on the phase space $M$. In physics, this flow describes precisely how a system with the energy function $H$ will evolve over time. After starting in the state $x\in M$, it will be found in the state $\phi^H_t(x)$ once the time $t$ has passed.

If we start from a configuration space described as a manifold $N$, then phase space is just the cotangent bundle $M=T^*N$ together with the exterior derivative of the \textit{Liouville form} $\lambda_{p}:= p\circ d_p\pi$ as the canonical choice for the symplectic form, where $\pi:T^*N\to N$ is the canonical projection.

Consider now an energy function $H:T^*N\to\R$ that does not depend on time. It turns out that this symmetry under translation in time implies that energy is a conserved quantity over time. So if $E$ is a regular value of $H$, then $H^{-1}(E)$ is a hypersurface containing all states of a fixed energy $E$ that is preserved by the Hamiltonian flow. \textit{Contact geometry} can be seen as the study of the naturally induced structure on such constant energy surfaces. In general, a \textit{contact form} on a $2n-1$ dimensional manifold is a 1-form $\alpha$ such that $\alpha\wedge(d\alpha)^{n-1}\neq0.$ This means geometrically that $d\alpha$ has a unique null direction that is transverse to the kernel of $\alpha$. Appropriately normalized, we call this direction the \textit{Reeb vector field} of $\alpha$, generating the \textit{Reeb flow}. In the case of the constant energy surfaces, restricting the Liouville form~$\lambda$ to them indeed yields such a contact form\footnote{Technically, $H^{-1}(E)$ needs to also be transverse to the \textit{Liouville vector field} $X_\text{can}$, which in induced coordinates $(q_i,p_i)$ on $T^*N$ is $\sum_i p_i\frac{\partial}{\partial p_i}$. This is however a physically reasonable assumption: Otherwise, scaling the momentum of a state would not lead to a change in energy.}. The Reeb flow is then nothing but the Hamiltonian flow restricted to the hypersurface.\medskip

In physics, one often considers systems that are described using a time-independent energy function together with a short time-dependent perturbation, say on the interval $[0,1]$. For some states $x\in M$, we might observe what looks like a huge coincident: After the perturbation, the system might be in a state that is different from $x$ only by a shift in time in its original orbit instead of being fundamentally changed. Mathematically, this is an example of a \textit{leafwise fixed point} of the Hamiltonian flow. A similar situation can occur when we limit our attention to a constant energy surface and perturb its dynamics for a short while. A $\textit{translated point}$ is a state that the perturbation affects only by shifting it in time.

Formally, a \textit{contactomorphism} on a manifold $M$ with contact form $\alpha$ is a diffeomorphism $\phi:M\to M$ that preserves the contact form up to scaling with a positive function. A \textit{translated point} $x\in M$ of $\phi$ is a point where $\alpha$ is preserved and whose image lies in the same orbit of the Reeb flow as $x$. Translated points can be seen as special cases either of leafwise fixed points or Reeb chords between Lagrangian submanifolds.
\medskip

Translated points were introduced by Sheila Sandon in~\cite{San12} as contact analogues to fixed points of Hamiltonian diffeomorphisms. She gave the following existence theorem in~\cite{San13}:

\begin{theorem}\label{thm:sandon}
\cite{San13}
	\begin{enumerate}
		\item Every contactomorphism of $S^{2n-1}$ which is contact isotopic to the identity and generic in the sense that all its translated points are non-degenerate has at least two translated points.
		\item Every contactomorphism of $\R P^{2n-1}$ which is contact isotopic to the identity has at least $2n$ translated points.
	\end{enumerate}
\end{theorem}

For a precise definition of non-degeneracy of translated points, see Section~2.4. This theorem can be seen as a contact version of the Arnol'd conjecture for the sphere and real projective space. In~\cite{San13}, Sandon also gives proofs for the $C^0$ and $C^1$-small cases of this conjecture, and in~\cite{GKPS} for contactomorphisms on lens spaces. The proof of this theorem essentially follows a path laid out by earlier work of Th\'eret~\cite{The96, The98} and Givental~\cite{Giv90}: It proceeds by explicitly lifting the contactomorphism to a symplectic setting in $\R^{2n}$, constructing \textit{generating functions} and then using Morse theoretic arguments. In the simplest case, a symplectomorphism on a manifold $M$ is said to be generated by a smooth function $f:M\to\R$ if there is an identification of the cotangent bundle $T^*M$ with $M\times M$ such that the 0-section coincides with the diagonal and the graph of the 1-form $df$ coincides with the graph of the symplectomorphism. The proof requires a more general notion due to H\"ormander where the generating function is defined on a fiber bundle over $M$~\cite{Hoe71}.\medskip

This thesis is meant to recap the proof of part~(i) of Theorem~\ref{thm:sandon} and to provide additional detail beyond the original publication. In this process, I uncovered a gap in the proof that Sandon and I have not been able to close so far (for details consider Remark~\ref{remark-missing-assumption}). We will consider the following statement instead:

\begin{theorem}\label{main-result-proven-intro}
Let $\phi$ be a contactomorphism on $S^{2n-1}$ without degenerate translated points. Assume that $\phi$ has a generating function $F:S^{2n+k-1}\to\R$ such that the sublevel set $\{F\#0\leq0\}$ is either empty or an embedded submanifold with non-trivial homology.

Then $\phi$ has at least two translated points.
\end{theorem}

Here we have replaced the assumption that $\phi$ is contact isotopic to the identity with the existence of a generating function with a particular associated sublevel set. While an exact understanding of this new assumption requires the definitions of Chapter~3, note that the critical difference here is the homological condition: It will turn out that if this was not required, existence of a contact isotopy would suffice to satisfy the assumptions of Theorem~\ref{main-result-proven-intro}.\medskip

This thesis adds to Sandon's proof by providing more details throughout: It closes gaps and fixes minor errors in the parametric Morse theory arguments, homological considerations, the composition formula and existence of simple generating functions of lifts. More care is taken to state and prove many results more generally than done in~\cite{San13} in order to make them applicable as stated to lifts of contactomorphisms. This is strictly speaking not the case in the original proof due to irregularity of the lift at the origin.
We also modify the definition of non-degenerate translated points to be more natural. While coinciding with Sandon's definition in the context of her theorem, our Definition~\ref{def:translated} is more restrictive in general and motivated by the connection to leafwise fixed-points (compare Remark~\ref{comparison-sandon-def}).
Finally, our Definition~\ref{def-gen-on-sphere} of generating functions \textit{on the sphere} allows for a cleaner presentation of the proof.\medskip

The remainder of this thesis is divided into three chapters: The second chapter provides background and context to the main result. In particular, we introduce basic symplectic and contact geometry as well as translated points and the Arnol'd conjecture. The third chapter explains the method of generating functions of contactomorphisms on the sphere in several successive steps. Finally, we provide  a form of parametric Morse theory and investigate the homology of sublevel sets of generating functions in order to prove the main result in the fourth chapter.
The appendix lists all differences between this thesis and Sandon's original paper and discusses how contactomorphisms on real projective space can be lifted to Euclidean space. The latter is necessary for a similar generating function approach of real projective space, e.g. as in part~(ii) of Theorem~\ref{thm:sandon}.\medskip

I would like to express deep gratitude to my advisor Fabian for his patience and countless helpful discussions. I am also very grateful to Sheila Sandon and Alexander Givental for kindly answering my questions on \cite{San12} and \cite{Giv90}. The support of my parents, my sister and friends has been invaluable while writing this thesis.

%% file: background.tex
\chapter{Background and Context} \label{chap:background}

In this chapter, we will recap essential concepts necessary to understand the statement of the main theorem and place it in a broader context. For the sake of brevity we will omit some proofs of standard results, instead referring to more thorough expositions in the literature.\medskip

The first two sections start by introducing basic notions of symplectic and contact geometry. In particular, we define symplectic reduction and the standard contact forms on spheres and real projective space that Sandon's result assumes. In the third section, we define symplectization as a way to associate a symplectic manifold to a given contact structure.
This allows us to introduce translated points and discuss their interpretations and known existence results in the fourth section. Finally, we discuss the main theorem from the perspective of a contact version of the Arnol'd conjecture.

\section{Symplectic Geometry}

Symplectic geometry concerns itself with \textit{symplectic manifolds:}

\begin{definition}
	A \textbf{symplectic manifold} is a manifold $M$ together with a closed and non-degenerate 2-form $\omega$. A \textbf{symplectomorphism} between symplectic manifolds $(M,\omega)$ and $(M',\omega')$ is a diffeomorphism $\phi: M\to M'$ such that $\phi^*\omega' = \omega$.
\end{definition}

Purely by linear algebra, the existence of a non-degenerate 2-form already implies that $M$ is even-dimensional. In fact, all symplectic manifolds locally look alike:

\begin{remark}[Standard symplectic structure on $\R^{2n}$ as a local model]\label{standard-symp-struct}
	Using global coordinates $z_j=x_j+\I y_j$ on $\R^{2n}\simeq \C^n$, we can define a symplectic 2-form $$\omega_{\text{std}}=\sum\limits_{j=1}^n dx_j\wedge dy_j = \frac{1}{2\I}\sum\limits_{j=0}^n d\bar{z}_j\wedge dz_j.$$
	This is not just an example, but the unique local model: By a celebrated theorem due to Darboux (see e.g. Theorem~3.2.2 in~\cite{MS17}), around every point of any given symplectic manifold there exists a neighbourhood that is symplectomorphic to an open subset of $(\R^{2n},\omega_{\text{std}})$.
	Note also that this standard symplectic structure generalizes to the cotangent bundles of arbitrary manifolds, which we will come back to in Remark~\ref{cotangent-bundle}.
\end{remark}

A central phenomenon in symplectic geometry is that a time-dependent function on $M$ (e.g., the total energy of states of a mechanical system) induces a flow:

\begin{definiprop}
	Consider a symplectic manifold $(M,\omega)$. A \textbf{Hamiltonian function} is a smooth map $H:M\times\R\to\R, (x,t)\mapsto H_t(x)$. It generates a unique time-dependent vector field $X_t$ via the relation $\omega(X_t,\cdot)=-dH_t$. The flow $\phi^H_t$ of $X_t$ starting at the identity is called \textbf{Hamiltonian flow} and preserves $\omega$ at every time. If it is additionally always surjective, it is called a \textbf{Hamiltonian isotopy.} A symplectomorphism $\phi:M\to M$ is a \textbf{Hamiltonian diffeomorphism} if it is the time-1 map of a Hamiltonian isotopy, i.e. $\phi=\phi^H_1$ for some $H.$
\end{definiprop}

\begin{proof}
	The unique existence of $X_t$ is an immediate consequence of non-degeneracy of $\omega$ while the flow preserves it due to Cartan's magic formula and closedness of $\omega.$ For more details see e.g. Section 3.1 of~\cite{MS17}.
\end{proof}

The symplectic structure on a manifold allows us to distinguish between different kinds of submanifolds:

\begin{definition}
	Given a symplectic form $\omega$ on a manifold $M$ and a subspace $W$ of the tangent space in $p\in M,$ we can define the \textbf{$\omega$-orthogonal} of $W$ as $$W^\omega := \{v\in T_pM \;|\; \forall w\in W: \omega(v,w)=0 \}$$
	and say that the subspace $W$ is
	\begin{itemize}
	\item \textbf{symplectic} if $\omega_p|_W$ is non-degenerate,
	\item \textbf{isotropic} if $W\subseteq W^\omega$,
	\item \textbf{coisotropic} if $W^\omega\subseteq W$,
	\item \textbf{Lagrangian} if $W^\omega = W$.
	\end{itemize}

	Applying the definition of $\omega$-orthogonal pointwise yields an operation on vector subbundles of $TM$. We similarly say that a submanifold of $M$ is symplectic, (co-)isotropic or Lagrangian if this holds for its tangent spaces in every point, respectively.
\end{definition}

These definitions are clearly invariant under symplectomorphisms. The following properties and alternative characterizations are also often convenient:

\begin{lemma}\label{lemma:submfd-types}
For a symplectic form $\omega$ on a $2n$-dimensional manifold $M$, point $p\in M$ and vector subspace $W\subseteq T_pM$, the following hold:
	\begin{enumerate}
		\item $W= (W^\omega)^\omega$
		\item $\dim W + \dim W^\omega = 2n$
		\item $W$ symplectic $\iff W^\omega$ symplectic $\iff W\cap W^\omega=0 \implies \dim W$ even
		\item $W$ isotropic $\iff W^\omega$ coisotropic $\iff \omega|_W=0 \implies \dim W \leq n$
		\item $W$ Lagrangian $\iff W$ isotropic and coisotropic $\iff \dim W=n$ and $W$ either isotropic or coisotropic
	\end{enumerate}
\end{lemma}

\begin{proof}
Follows immediately by linear algebra, for details see e.g. Section 2.1 of~\cite{MS17}.
\end{proof}

Weinstein's approach~\cite{Wei81} to symplectic geometry focuses on the central role of Lagrangian submanifolds in particular. We will see one example of this in Proposition~\ref{symplecto-to-lagrangian}, which tells us that we can decide whether a diffeomorphism is a symplectomorphism by checking whether its graph is Lagrangian.

Weinstein established a crucial normal form theorem for Lagrangian submanifolds. Our main argument does not require it, but we will refer to it in contextual remarks:

\begin{theorem}[Weinstein Lagrangian neighbourhood theorem]\label{thm:weinstein-lagrangian}
Let $M$ be a manifold with symplectic form $\omega$ and a closed Lagrangian submanifold $L$. Then there exists a neighbourhood of $L$ that is symplectomorphic to a neighbourhood of the 0-section in $T^*L$ via a symplectomorphism that extends the map $L\ni x\mapsto(x,0)\in T^*L$.
\end{theorem}

\begin{proof}
See e.g. Theorem~3.4.13 in~\cite{MS17}.
\end{proof}

Finally, we will require the notion of \textit{symplectic reduction} of coisotropic submanifolds to define generating functions. Some of the directions tangent to the submanifold might need to be paired with transverse vectors in order for the symplectic form to not evaluate to zero. In the sense of Lemma~\ref{lemma:submfd-types}~(iv), these directions can be considered isotropic. To make this notion precise, consider the following:

\begin{proposition}\label{iso-leaves}
	For a symplectic manifold $(M,\omega)$ and a coisotropic submanifold $N$, $N^\omega:=(TN)^\omega$ is an integrable distribution on $N$. By Frobenius' theorem, it determines a foliation $\mathcal{F}_{N,\omega}$ of $N$.\\

	\startsubpart{Definition} We call $\mathcal{F}_{N,\omega}$ the \textbf{characteristic foliation} with the \textbf{isotropic leaves}.
\end{proposition}

\begin{proof}
	Following a standard calculation (e.g. \cite[Lemma~5.4.1]{MS17}) we check integrability by showing that $(TN)^\omega$ is closed under the Lie bracket. This means that for all vector fields $X,Y$ with values in $(TN)^\omega$, $p\in N,$ and $Z_p\in T_pN$ it must follow that $$\omega([X_p,Y_p],Z_p)=0.$$ To see this, continue $Z_p$ to a vector field $Z$ on $N$ and compute
	\begin{align*}
		0 &= d\omega(X,Y,Z) \\ &= \mathcal{L}_X(\omega(Z,Y)) + \mathcal{L}_Y(\omega(X,Z)) + \mathcal{L}_Z(\omega(Y,X))  \\ &\qquad+ \omega([Y,Z],X) + \omega([Y,X],Z)  + \omega([X,Y],Z)  \\
		&= \omega([X,Y], Z)
	\end{align*}

\end{proof}

Symplectic reduction is, philosophically speaking, a way to quotient out these directions that are unpaired within the submanifold $N$. Under suitable assumptions, this reduces $N$ to a manifold $N_\omega$ with a canonical symplectic structure:

\begin{definition}\label{def:regularity}
	Let $\iota:N\hookrightarrow M$ be a coisotropic submanifold of a symplectic manifold $(M,\omega)$. The \textbf{symplectic reduction} $N_\omega$ of $N$ in $M$ is the topological space $N/\sim$, where the quotient identifies points in $N$ that lie in the same leaf of the foliation $\mathcal{F}_{N,\omega}$.\\
	
	$N$ is called \textbf{regular} if the following holds:
	\begin{itemize}
		\item For all $p\in N$, there is a submanifold $S\subseteq N$ containing $p$ that intersects every isotropic leaf at most once and satisfies $T_qN=T_qS\oplus T_q N^\omega$ for all $q\in S.$
		\item $N_\omega$ is a Hausdorff space.
	\end{itemize}

\startsubpart{Proposition}
	For a regular coisotropic submanifold $N$, the symplectic reduction carries...
	\begin{itemize}
		\item a unique smooth manifold structure such that the projection $\pi: N\to N_\omega$ is a submersion, and
		\item a unique symplectic form $\bar{\omega}$ such that $\pi^*\bar{\omega} = \iota^*\omega$.
	\end{itemize}
\end{definition}

\begin{proof}
	See e.g. Proposition~5.4.5 of~\cite{MS17}.
\end{proof}

We again want to consider how Lagrangian submanifolds behave under this construction:

\begin{lemma}\label{lem:lag-symp-red}
Let $(M,\omega)$ be a symplectic manifold with a regular coisotropic submanifold $N$ and Lagrangian submanifold $L$. If $L$ is transverse to $N$, then $\pi(L\cap N)$ is an immersed Lagrangian along the quotient map $\pi$ of the symplectic reduction $N_\omega$.
\end{lemma}

\begin{proof}
See e.g. Proposition 5.4.7 of \cite{MS17}.
\end{proof}

\section{Contact Geometry}

\textit{Contact structures} are defined as maximally non-integrable hyperplanes and are the natural odd-dimensional analogues to the even-dimensional symplectic structures. We will require the slightly stronger notion of a \textit{contact form}:

\begin{definition}\label{def:contact-basics}
	A \textbf{contact form} on a manifold $M$ of dimension $2n-1$ is a 1-form $\alpha$ such that $\alpha\wedge(d\alpha)^{n-1}\neq 0$. A \textbf{contactomorphism} between manifolds with contact forms $(M_i,\alpha_i)$ is a diffeomorphism $\phi: M_1\to M_2$ such that $\phi^*\alpha_2 = e^g\alpha_1$ for some $g:M_1\to\R$. If $g$ vanishes, $\phi$ is called \textbf{strict}. A \textbf{contact isotopy} on $(M,\alpha)$ is an isotopy $\phi$ such that every $\phi_t$ is a contactomorphism.
\end{definition}

There are various ways to associate contact forms to some symplectic structures and vice versa, see for example the discussion in the introduction for constant energy surfaces or Section 3.5 of~\cite{MS17} for \textit{prequantization}. We will consider \textit{symplectization} in detail in the next section.\medskip

The contact condition on $\alpha$ can geometrically be understood as a 'maximal' amount of twisting of the hyperplanes $\ker\alpha$ while moving along them, preventing the existence of integral submanifolds of dimension larger than $n$ (see e.g. Section 5.1 of~\cite{Bla10}). In particular it implies that $d\alpha$ has a unique null direction transverse to $\ker\alpha$, which gives rise to a canonical flow associated with a contact form:

\begin{definiprop}
	The exterior derivative $d\alpha$ of a contact form has a unique null direction that is transverse to $\ker \alpha$. The \textbf{Reeb vector field} of $\alpha$ is the normalized vector field along this direction, i.e. the unique $X^\alpha\in\mathfrak{X}(M)$ such that $\iota_{X^\alpha}d\alpha=0$ and $\alpha(X^\alpha)=1$.
	The \textbf{Reeb flow} $R^\alpha$ generated by $X^\alpha$ is a contact isotopy.
\end{definiprop}
\begin{proof}
	See e.g. 1.1.9 and 2.3.2 of \cite{Gei08}.
\end{proof}

Similarly to the symplectic case, time-dependent functions on a contact manifold generate flows:

\begin{definiprop}\label{contact-hamiltonian}
	For every smooth time-dependent function $H_t: M\to \mathbb{R}$, there exists a unique contact isotopy starting at the identity and generated by a vector field $X_t$ that satisfies $\alpha(X_t)=H_t$. The function $H_t$ is then referred to as the \textbf{(contact) Hamiltonian} that generates the isotopy, and every contact isotopy starting at the identity arises this way. We can uniquely characterize the vector field $X_t$ by $\iota_{X_t}d\alpha=dH_t(R_a)\alpha - dH_t$ and $\alpha(X_t)=H_t$.
\end{definiprop}

\begin{proof}
	See e.g. Section 2.3 of~\cite{Gei08}.
\end{proof}

We again have a standard contact form on Euclidean space as a local model:

\begin{remark}[Standard contact structure on $\R^{2n+1}$ as local model]\label{darboux}
	Using global coordinates $x_j, y_j$ and $z$ on $\R^{2n+1}$, we can define a contact form $$\alpha_{\text{std}}= dz- \sum_{j=1}^n y_j\wedge dx_j.$$ It is straightforward to check that the unit vector field $\partial/\partial z$ is the Reeb vector field and translation in that direction is the Reeb flow. By a contact version of Darboux' Theorem~(e.g.~2.5.1 in \cite{Gei08}), every point of any given contact manifold has a neighbourhood that is strictly contactomorphic to a subset of $(\R^{2n+1}, \alpha_{\text{std}})$.
\end{remark}

This standard symplectic form furthermore induces contact forms on odd-dimensional spheres and real projective spaces that are the topic of Theorem~\ref{thm:sandon}:

\begin{remark}[Contact form on $S^{2n-1}$]\label{standard-struct-sphere}
	Note that the standard symplectic structure on $\R^{2n}\simeq\C^n$ from \ref{standard-symp-struct} can be written as $d\lambda$, where $$\lambda = \sum_{j=1}^n (x_j dy_j - y_j dx_j) = \sum_{j=1}^n\frac{\bar{z}_jdz_j-z_jd\bar{z}_j}{2\I}$$ is the Liouville form expressed in complex coordinates $z_j=x_j+\I y_j$. Restricting $\lambda$ to the unit $(2n-1)$-sphere yields a contact form $\alpha_\text{std}$ on it. It is straightforward to verify that the Reeb vector field is given by $(-y_1, ..., -y_n, x_1, ..., x_n)$ and its flow is $z\mapsto \exp(\I t) z$, which in particular is $2\pi$-periodic.
\end{remark}

\begin{remark}[Contact form on $\R P^{2n-1}$]\label{standard-struct-proj}
	Note that $\R P^{2n-1}$ is the result of quotienting out the antipodal discrete $\Z_2$ action on the sphere, i.e. there is a double cover $$\pi: S^{2n-1}\to S^{2n-1}/\Z_2 \simeq \R P^{2n-1}.$$ The standard contact form $\tilde{\alpha}$ on $\R P^{2n-1}$ arises by pushing the standard contact form of $S^{2n-1}$ along this map, as the latter is invariant under the $\Z_2$ action. The Reeb vector field and flow are correspondingly also given by the pushforward along $\pi$ and the induced map on the quotient, respectively.
\end{remark}

\section{Symplectization of Manifolds with a Contact Form}

\textit{Symplectization} is a way to relate a symplectic manifold to any given manifold with contact form. This will allow us to use techniques of symplectic geometry in the search for translated points. The statements in this section are based on \cite{MS17} and \cite{San14}, but reformulated here as a functor.

\begin{definition}[Symplectization functor]\label{def:symplectization}
	Let $(M,\alpha)$ and $(N,\beta)$ be manifolds with contact forms, and $\phi:(M,\alpha)\to (N,\beta)$ a contactomorphism such that $\phi^*\alpha=e^g\alpha$ for some $g:M\to\R$.

	\begin{itemize}
		\item The \textbf{symplectization $\symp (M,\alpha)$ of $(M,\alpha)$} is the manifold $M\times\R$ equipped with the 2-form $\omega_{\alpha}:=d(e^\theta \alpha),$ where $\theta$ is the coordinate on $\R.$
		\item The \textbf{symplectization of $\phi$} is the map $\symp \phi: M\times\R\to N\times\R$ defined by $$(p,\theta)\mapsto (\phi(p), \theta-g(p)).$$
	\end{itemize}

\startsubpart{Proposition}
	$\symp$ forms a functor from the category of manifolds with contact forms and contactomorphisms to the category of symplectic manifolds with symplectomorphisms. This means that $\symp (M,\alpha)$ is a symplectic manifold, $\symp\phi$ is a symplectomorphism, symplectization of contactomorphisms and composition commute and $\symp \Id_{(M,\alpha)} = \Id_{\symp (M,\alpha)}$.
\end{definition}

\begin{proof}

	We first show that $\symp (M,\alpha)$ is a symplectic manifold. Since $\omega_\alpha$ is clearly exact, we only need to show non-degeneracy. To this end, let $2n-1=\dim M$ and check that $$(d(e^\theta \alpha))^n = (e^\theta d\theta\wedge\alpha + e^\theta d\alpha)^n = ne^{n\theta} d\theta\wedge\alpha\wedge(d\alpha)^{n-1}$$ does not vanish, which follows immediately by the contact condition.

	To see that $\symp \phi$ is a symplectomorphism, we calculate $$(\symp\phi)^* \omega = d \left((\symp\phi)^*e^\theta \alpha\right) = d(e^{\theta-g}\phi^*\alpha) = d(e^{\theta-g}e^g\alpha) = \omega.$$

	Symplectization and composition commute immediately by construction and the identity on $M$ is clearly mapped to the identity on $\symp M$.
\end{proof}

Note that one can already define a symplectization for the weaker contact structure, independently of a contact form. This is often referred to as \textit{intrinsic symplectization}, see e.g. \cite{MS17}.\medskip

We can lift a contact Hamiltonian to obtain a symplectic Hamiltonian:

\begin{definition}\label{def:symplectization-ham}
	Let $(M,\alpha)$ be a manifold with contact form and $H_t\in\mathcal{C}^\infty(M)$ a time-dependent function. We define its \textbf{lift} to the symplectization $\symp (M,\alpha)$ to be the function
	\begin{align*}
		\symp H_t : M\times\R&\to\R\\
		(p,\theta) &\mapsto e^\theta H_t(p).
	\end{align*}

\startsubpart{Proposition} Let $\phi_t$ be a contact isotopy generated by a contact Hamiltonian $H_t$. Then the symplectization $\symp \phi_t$ is generated by the symplectic Hamiltonian $\symp H_t$.

\end{definition}

\begin{proof}

	The fact that $H_t$ generates $\phi_t$ means that there is a unique vector field $X_t$ on $M$ generating $\phi_t$ such that $\alpha(X_t)=H_t.$ If $g_t$ are chosen such that $\phi_t^* \alpha = e^{g_t}\alpha$, then the symplectizations of $\phi_t$ are given by $$\symp \phi_t (p,\theta) = (\phi_t(p), \theta-g_t(p)).$$ Differentiating yields the vector field $$\tilde{X}_t := \left(X_t, -\frac{d}{dt}g_t\right).$$ We are done if we show that this is the vector field generated by $\symp H_t$, i.e. $\omega_\alpha(\tilde{X}_t,\cdot) = -d(\symp H_t).$ Using the definitions of $\omega_\alpha$ and $\symp H_t$ and dividing by $e^\theta$, this is equivalent to
	$$ (d\theta\wedge\alpha + d\alpha) (\tilde{X}_t,\cdot) = -H_t d\theta - dH_t . $$ Evaluating the left hand vector field insertion and using $\alpha(X_t)=H_t$, we see that this is equivalent to
	\begin{equation} d\alpha(X_t,\cdot) - \frac{d}{dt}g_t \;\alpha = - dH_t .\label{eq:ziel}\end{equation}
To verify that this equation holds, we first compute:
	\begin{align*}
		\left(\frac{d}{dt}g_t\right) \phi_t^*\alpha &= \frac{d}{dt} (e^{g_t}\alpha)
		= \frac{d}{dt} \phi_t^* \alpha \\
		&= \phi^*_t \left( \mathcal{L}_{X_t} \alpha \right) \\
		&= \phi^*_t \big( d(\alpha(X_t)) +d\alpha(X_t,\cdot) \big) \\
		&= \phi^*_t \big( dH_t +d\alpha(X_t,\cdot) \big)
	\end{align*}
	For the first two equalities, we use $\phi^*_t\alpha = e^{g_t}\alpha$. The third equality follows from the following general expression for the time derivative of a pullback of a family of forms $\beta_t$ along a flow $\phi_t$:
	\begin{equation}\label{eq:timederivflow}
	\frac{d}{dt} \phi_t^* \beta_t = \phi_t^*\left( \mathcal{L}_{X_t}\beta_t + \frac{d}{dt}\beta_t  \right).
	\end{equation}
	Finally, we use Cartan's magic formula and $\alpha(X_t)=H_t$ for the last two equalities.
	
	Pushing forward and canceling the pull-back now yields Eq.~\eqref{eq:ziel}.

\end{proof}

Note that this analogy between symplectic and contact isotopies is not perfect: All contact isotopies are generated by Hamiltonians, while this is not true for all symplectic isotopies (see e.g.~\cite{Gei08}).

\section{Translated Points of Contactomorphisms}\label{sec:translated}

Sandon introduced a notion of \textit{translated points} in~\cite{San12} as a contact analogue for fixed points in the symplectic setting:

\begin{definition}[Discriminant and translated points]\label{def:translated}
	Consider a contactomorphism $\phi$ on a manifold $M$ with contact form $\alpha$ and a function $g$ such that $\phi^*\alpha = e^g \alpha$.\medskip
	
	\begin{itemize}
	\item A \textbf{discriminant point} of $\phi$ is a fixed point $p=\phi(p)$ for which $g(p)=0$ holds. A discriminant point $p$ is \textbf{non-degenerate} if $\not\exists X\in T_pM: d\phi (X) = X$ and $X(g)=0$.
	
	\item A \textbf{translated point} is a point $p$ such that there exists \textit{at least one} $t\in\R$ such that $p$ is a discriminant point of $R^\alpha_{t}\phi$, where $R^\alpha$ is the Reeb flow on $M$. A translated point $p$ is \textbf{non-degenerate} if, for \textit{any} $t$ such that $p$ is a discriminant point of $R^\alpha_{t}\phi$, $p$ is a non-degenerate discriminant point.

	\end{itemize}
\end{definition}

\begin{remark}[Comparison to Sandon's definition]\label{comparison-sandon-def}
Our definitions match those of Sandon except in one point: For non-degeneracy of a translated point, Sandon only requires non-degeneracy as a discriminant point of $R^\alpha_{-t}\phi$ for a single $t$, namely the smallest one such that $p$ is a discriminant point of $R^\alpha_{-t}\phi$. Our definition is stronger and assumes non-degeneracy for all such $t$. We consider our definition more natural and consistent with the notion of non-degeneracy of leafwise fixed points in~\cite{Zil10} (compare Proposition~\ref{char-pts-via-symp}).

For the purposes of our main result, this makes no difference: If $t_1$ and $t_2$ are two times where the definition for a translated point $p$ applies, then the maps under consideration differ by $\Phi:=R^\alpha_{(t_2-t_1)}$ and we must have $\Phi(p)=p.$ It is clear that the differential of the Reeb flow on the sphere $z\mapsto\exp(it)z$ is the identity at all its fixed points $p$. For this reason, the degeneracy conditions at any suitable $t$ are equivalent. The same holds for real projective space, the other subject of~\cite{San13}.

\end{remark}

Sandon also points out that we can view translated points as special cases of \textit{leafwise fixed points} or interpret them in terms of \textit{Reeb chords}:

\begin{definition}[Leafwise fixed points]
Let $N$ be a coisotropic submanifold of a symplectic manifold $(M,\omega).$ For $x\in N$ write $N_x\subseteq N$ for the isotropic leaf through $x$. A \textbf{leafwise fixed point} of a symplectomorphism $\phi:M\to M$ is a point $x\in N$ such that $\phi(x)\in N_x.$
\end{definition}

\begin{proposition}[Characterization of discriminant and translated points by symplectization]\label{char-pts-via-symp}
	Let $\phi$ be a contactomorphism on a manifold $M$ with contact form $\alpha$.

	\begin{enumerate}
	\item $p\in M$ is a discriminant point of $\phi$ if and only if $(p,\theta)$ is a fixed point of $\symp \phi$ for any (equivalently all) $\theta\in\R.$ Non-degeneracy of the former corresponds to non-degeneracy of the latter\footnotemark\ in the directions tangent to $M$.

	\item $p\in M$ is a translated point of $\phi$ if and only if $(p,\theta)$ is a leafwise fixed point of $\symp\phi$ with regard to the coisotropic submanifold $M\times\{\theta\}$ of $\symp(M,\alpha)$ for any (equivalently all) $\theta\in\R$. Non-degeneracy of the former corresponds to non-degeneracy of the latter\footnotemark\ in the directions tangent to $M$.
	\end{enumerate}
\end{proposition}

\setcounter{footnote}{\value{footnote}-1}
\footnotetext{Note that by a non-degenerate fixed point of $\psi$ we mean a fixed point $x=\psi(x)$ such that the graph of $\psi$ intersects the diagonal transversally in $(x,\psi(x))$, or equivalently such that one is not an eigenvalue of $d_x\psi$.}

\setcounter{footnote}{\value{footnote}+1}
\footnotetext{Refer to Section 2.4 of~\cite{Zil10} for the definition of non-degenerate leafwise fixed points based on the linear holonomy of a foliation. Note that part (ii) of this proposition has no bearing on the proof of our main result.}

\begin{proof}
Recall that $\symp\phi:M\times\R\to M\times\R$ is defined by $$
\symp\phi(p,\theta) = (\phi(p), \theta-g(p)),
$$
where $g:M\to\R$ is the map such that $\phi^*\alpha=e^g\alpha.$
Compute the derivative for $X\in T_pM, Y\in T_\theta\R$:
$$
d_{(p,\theta)}(\symp\phi)(X,Y) =
(d_p\phi X, Y-d_p g X).
$$

Regarding (i): By the definition of $\symp\phi,$ $(p,\theta)$ is a fixed point if and only if $\phi(p)=p$ and $g(p)=0,$ irrespective of $\theta.$ This is exactly the condition for $p$ to be a discriminant point. This is non-degenerate by definition if there exists no $X\in T_pM\setminus\{0\}$ such that $(d_p\phi,d_pg)(X)=(X,0).$ By our computation of the derivative of $\symp\phi,$ this is equivalent to $$
d_{(p,\theta)}(\symp\phi)(X,0) = (X,0)
$$ or in other words that there is no eigenvector $(X,0)$ of $d_{(p,\theta)}(\symp\phi)$ with eigenvalue one. This is nondegeneracy of $(p,\theta)$ along $M$. Note that every fixed point of $\symp\phi$ is automatically degenerate in the $\R-$direction.\medskip

Regarding (ii):
For this part of the proof we will be working in the context of Section~2 of~\cite{Zil10}. Pick any $\theta\in\R.$\medskip

Note first that if $X$ is the Reeb vector field on $(M,\alpha),$ then the Reeb vector field on $(M\times\{\theta\}, e^\theta\alpha)$ is $(e^{-\theta}X,0).$ Pairing this with any vector $(Y,0)$ tangent to $M\times\{\theta\}$ pairs to zero under the symplectic form on $M\times\R,$ i.e. $d(e^\theta\alpha)((X,0),(Y,0)) = 0.$ It follows that $(X,0)$ lies in the symplectic complement of the tangent space of the coisotropic submanifold. In other words, the characteristic foliation of the coisotropic submanifold $M\times\{\theta\}$ coincides with the Reeb foliation on the contact manifold $(M\times\{\theta\}, e^\theta\alpha).$\medskip

Since $\symp R^\alpha_t(p,\theta)=(R^\alpha_t(p),\theta)$ for the Reeb flow $R^\alpha_t$ on $M$, the last paragraph implies that two points $(p_i,\theta)\in M\times\{\theta\}$ lie in the same leaf of the characteristic foliation exactly when there exists a $t$ such that $R^\alpha_t(p_1)=p_2$.\medskip

By definition and (i), $p\in M$ is a translated point of $\phi$ on $M$ if and only if there is a $t$ such that $(p,\theta)$ is a fixed point of $\symp R^\alpha_{t} \symp\phi$. By the previous paragraph, this is the case if $\symp\phi(p,\theta)$ lies in the same isotropic leaf as $(p,\theta),$ i.e. if $(p,\theta)$ is a leafwise fixed point.\medskip

The only thing left to show is that the notions of non-degeneracy match. Note first that the spaces $N_x\mathcal{F}$ from~\cite{Zil10} can here be identified with the symplectic complement $(X,0)_{(p,\theta)}^\omega$ of the vector field $(X,0)$ at $x=(p,\theta)\in N$ with respect to the symplectic form $\omega:=d(e^\theta\alpha)$. For any path within an isotropic leaf connecting two points $(p,\theta)$ and $(R^\alpha_{t}(p),\theta),$ the linear holonomy map $\operatorname{hol}_x^\mathcal{F}$ from~\cite{Zil10} is in this context just the map on $(X,0)_{(p,\theta)}^\omega$ induced by $d_{(p,\theta)}(\symp R^\alpha_{t}).$ This induced map is well-defined since the Reeb flow preserves the Reeb vector field and $\omega.$

A translated point $p\in M$ is non-degenerate by definition and (i) if and only if the map $d_{(p,\theta)}\symp(R^\alpha_{t}\phi)$ has no eigenvalue one. $d_{(p,\theta)}\symp\phi$ preserves the Reeb vector field since $g(p)=0,$ and the Reeb flow always preserves it. It follows that the above condition need only be checked on the subspace $(X,0)_{(p,\theta)}^\omega.$ This is now precisely equivalent to Eq.~(2.7) of~\cite{Zil10}.
\end{proof}

\begin{remark}[Translated points as Reeb chords]
A \textbf{Reeb chord} on a manifold equipped with a contact form is a section of an integral curve of the Reeb vector field that starts and ends on Legendrian submanifolds.

	Now fix a contactomorphism $\phi$ on a manifold $M$ equipped with contact form $\alpha$ such that $\phi^*\alpha = e^g\alpha$. Following Sandon~\cite{San12}, consider the contact product, i.e. the manifold $M\times M\times\R$ with contact form $A=e^\theta \alpha_1 - \alpha_2$, where $\theta$ is the coordinate on $\R$ and $\alpha_i$ are $\alpha$ pulled back along the first and second projection, respectively. A translated point now corresponds to a Reeb chord between the diagonal $\Delta := \{ (q,q,0) \;|\; q\in M \}$ and the graph $\operatorname{gr}_\phi := \{ (q,\phi(q),g(g)) \;|\; q\in M \}$. To see this, note that a translated point corresponds precisely to a point $(q,R^\alpha_t,0)\in \operatorname{gr}_\phi$ for some time $t$, and the Reeb flow on $M\times M\times\R$ can be expressed by the flow $R^\alpha$ on $M$ as $R^A=(0,-R^\alpha,0)$.
\end{remark}

We want to give a short overview of related existence results in the literature:

\begin{remark}[Existence results of translated points]
We first consider a number of results known before the introduction of translated points that could be used to prove their existence due to their interpretation as either Reeb chords and leafwise fixed points:

While many theorems give existence of \textbf{Reeb chords} connecting a Legendrian submanifold with itself, for translated points we require a chord between a Legendrian submanifold and a contact deformation of it. One such theorem for the 0-section in the 1-jet bundle is due to Chekanov~\cite{Ch96}, on which some of Sandon's results in~\cite{San12} build.

The problem of finding \textbf{leafwise fixed points} was introduced by Moser~\cite{Mos78}. The special case of translated points fixes the codimension of the coisotropic submanifold to one, so we will not consider the numerous results for minimal and maximal codimension here.

Moser's original result was generalized by Banyaga~\cite{Mos78,Ban80} to show existence of leafwise fixed points for closed coisotropic submanifolds provided that the symplectomorphism is $C^1$-small. A decade later, a fruitful line of research was opened with Hofer's introduction of his metric on the group of Hamiltonian symplectomorphisms~\cite{Hof90}. Assuming smallness of the symplectomorphism with regard to this metric instead, he originally showed existence for hypersurfaces of restricted contact type in $\R^{2n}$. This result was subsequently extended to hold under weaker assumptions on the manifolds~\cite{Gin07, Dra08, Ker08, Gue10, Zil10, AF10, AM10, Kan12}, often giving a homological lower bound on the number of leafwise fixed points. Instead of $C^1$ or Hofer-smallness, \cite{Zil17} considers Hamiltonian symplectomorphisms that are $C^0$-close to the identity for general closed coisotropic submanifolds. Some theorems consider very concrete manifolds~\cite{Mer11, AF12, AF10b, AM11} or use a symmetry of the Hamiltonian isotopy~\cite{EH89, AF12b}.

Since they were introduced, a number of results have appeared that are specific to \textbf{translated points}. In her original paper~\cite{San12}, Sandon uses~\cite{Ch96} to establish existence of translated points for compactly supported contactomorphisms that are contact isotopic to the identity and defined on either $\R^{2n+1}$ or $\R^{2n}\times S^1$, and then shows a contact analogue of a result due to Viterbo~\cite{Vit92} for \textit{iterated} translated points.
Like the main theorem of this thesis, most existence results of translated points can be understood as contact versions of the Arnol'd conjecture in special cases (compare the next section and in particular Remark~\ref{progress-contact-arn}). \cite{San13} also gives proofs of the contact Arnol'd conjecture for the $C^0$ and $C^1$ cases. A very similar approach to this paper is taken in \cite{GKPS}, yielding the conjecture by defining an analogue to the non-linear Maslov index for lens spaces.
In \cite{AM13,AFM15,Ter18,MN18,MU19}, bounds on the number of translated points are derived by making assumptions on the contact manifold. \cite{She17} introduces a contact version of the Hofer metric and proves that contactomorphisms that are small with respect to this metric implies bounds on the number of translated points, given that some assumptions on the manifold are met.
\end{remark}

\section{The Symplectic and Contact Arnol'd Conjectures}\label{sec:carnold}

For time-independent Hamiltonian functions, critical points correspond to fixed points of the generated Hamiltonian diffeomorphism. Arnol'd~\cite{Arn65} conjectured that this link between fixed points and functions on the underlying manifold holds more generally:

\begin{conjecture}[Arnol'd]
	Any Hamiltonian symplectomorphism $\phi:M\to M$ on a closed symplectic manifold $M$ has at least as many fixed points as the minimal number of critical points of smooth functions on it: $$\#\text{Fix}(\phi)\geq \text{Crit}(M) := \min \{ \#\text{Crit}(f)\;|\; f\in C^\infty(M,\R)\}$$
\end{conjecture}

Since its inception, this conjecture has been central to many developments in symplectic topology such as Floer homology. While significant progress has been made on it, in full generality it remains open.

\begin{remark}[Versions of the Arnol'd conjecture]
	Note that using Lyusternik-Schnirelman theory, one can establish that $\text{Crit}(M)$ is larger than the \textit{cup length} $\text{cup}(M)$, i.e. the maximal number of cohomology classes of 1-forms on $M$ such that their cup product does not vanish (see e.g.~\cite{CLOT03}). So the Arnol'd conjecture in particular implies the \textit{weak Arnol'd conjecture} that
	$$\#\text{Fix}(\phi)\geq \text{cup}(M).$$

	Since the link between fixed points and critical points in the time-independent case preserves degeneracy, it is natural to also consider the \textit{non-degenerate Arnol'd conjecture}. It states that if $\phi$ is generic in the sense that it has only non-degenerate fixed points, there exist at least as many of them as a Morse function on the manifold must have critical points:
	$$\#\text{Fix}(\phi)\geq \widetilde{\text{Crit}}(M) := \min \{ \#\text{Crit}(f)\;|\; f\in C^\infty(M,\R), \; f \text{ is Morse} \}$$

	Note that a Morse function is just a smooth map $M\to\R$ such that every critical point is non-degenerate. By the Morse inequality (see e.g.~\cite{milnor-morse}), this would in particular imply a version of the \textit{weak non-degenerate Arnol'd conjecture}. This posits that, if $\phi$ has only non-degenerate fixed points, we have
	$$\#\text{Fix}(\phi)\geq \sum\limits_{k=0}^{\dim M} \beta_k(M,\mathbb{F}),$$
	where the sum on the right side is taken over the $k$-th Betti numbers of the manifold with respect to some field $\mathbb{F}$.

	Note that if we consider a non-compact manifold in the (strong) Arnol'd conjecture instead, it holds trivially: Every non-compact smooth manifold admits a smooth real-valued function without critical points, see e.g. Theorem~4.8 of~\cite{Hir61}.
\end{remark}

\begin{remark}[Progress on the Arnol'd conjectures]
A generating function approach implies the Arnol'd conjecture in the case that $\phi$ is sufficiently close to the identity in the $C^1-$topology (see e.g. Chapter~9 of~\cite{MS17}). After being proven by Eliashberg for Riemannian surfaces and by Conley and Zehnder for tori, the weak non-degenerate Arnol'd conjecture was established using \textit{Floer homology}.

For overviews and references on these and more results, consider Chapter~11 of~\cite{MS17}, page~153 of~\cite{AD14} and \cite{Sal99}.

\end{remark}

Based on Section~\ref{sec:translated}, translated points can be seen as an analogue to fixed points of Hamiltonian diffeomorphisms when moving from the symplectic to the contact setting. This prompted Sheila Sandon to formulate the following contact version of the Arnol'd conjecture~\cite{San12}:

\begin{conjecture}[Contact Arnol'd]\label{carnold}
The number of translated points of a contactomorphism on a compact manifold that is contact isotopic to the identity is at least the minimal number of critical points of functions on $M$.
\end{conjecture}

\begin{remark}[Relation to the main theorem]
	We will now argue that $\widetilde{\text{Crit}}(S^n)=2$ and $\text{Crit}(\R P^{n}) = n+1.$ It follows that Theorem~\ref{thm:sandon} is just the non-degenerate contact Arnol'd conjecture for the sphere and the full contact Arnol'd conjecture for real projective space.

	For the $n$-sphere with $n>0$, we have $\text{Crit}(S^{n})=\widetilde{\text{Crit}}(S^n)=2$: Any function on it must have at least a maximum and minimum by compactness, and the projection on any axis is a function with no more critical points than that.

	For $\R P^{n},$ we need to put in slightly more work that we only outline here. The \textit{Lyusternik-Schnirelman category} of real projective space is $$LS(\R P^n) = n,$$ see e.g. Example 1.8(2) of~\cite{CLOT03}.
	Combining the Lyusternik-Schnirelman Theorem and Takens' Theorem (Theorem~1.15 and Proposition~7.26 of~\cite{CLOT03}, respectively) yields $$1+LS(M)\leq \text{Crit}(M)\leq 1+ \dim (M)$$
	for a smooth connected manifold $M$.
	For $M=\R P^n$, this indeed implies $\widetilde{\text{Crit}}(\R P^n)\geq \text{Crit}(\R P^n) = n+1$.
	For the other direction in the non-degenerate case or to avoid Takens' theorem altogether, we could also follow Example 3.8 in~\cite{Mat02}. It shows that for $a_i\in\R$ pairwise distinct but otherwise arbitrary, $$(x_0:...:x_n)\mapsto \frac{\sum a_i x_i^2}{|x|^2}$$ is a Morse function on $\R P^n$ with exactly $n+1$ critical points (the equivalence classes of the standard unit vectors).

\end{remark}

\begin{remark}[Progress on the contact Arnol'd conjecture]\label{progress-contact-arn}
	Completely analogously to the symplectic case, Sandon also proves the $C^0$ and $C^1$-small version of this conjecture in~\cite{San13} using simple generating functions. \cite{GKPS} builds on the methods of~\cite{San13} to show the conjecture for lens spaces. A Rabinowitz-Floer homology approach to translated points is taken in \cite{AM13}. Building on this, \cite{AFM15} shows a non-degenerate weak contact Arnol'd conjectures under the assumption that the contact manifold has no contractible closed Reeb orbits. This is generalized to all \textit{hypertight} contact manifolds in~\cite{MN18}. \cite{She17} establishes the weak Arnol'd conjecture given smallness under a Hofer-norm for contactomorphisms. \cite{Ter18} shows cup-length estimates without non-degeneracy assumptions for certain prequantization spaces.

\end{remark}

%% file: techniques.tex
\chapter{Generating Functions in Symplectic and Contact Geometry} \label{chap:techniques}

This chapter introduces the theory of generating functions of contactomorphisms on the sphere. This notion is based on a generalization of classical generating functions of symplectic geometry to fiber bundles due to H\"ormander~\cite{Hoe71} which was subsequently applied to the contact setting among others by Sheila Sandon. In the generating function approach, \textit{some} differentiable functions $F:S^{2n+k-1}$ determine maps $\phi:S^{2n-1}\to S^{2n-1}.$ Philosophically, the idea is that the graph of the differential $dF$ is a set that, after some identification and reduction steps, equals the graph of the map $\phi.$
\medskip

The four central features of this theory that we need for our proof can be summarized as follows:

\begin{enumerate}
	\item The Reeb flow on the sphere has a particularly nice family of generating functions that are restrictions of quadratic forms to the unit sphere.
	\item For functions $(F_{j}:S^{2n+k_j-1}\to\mathbb{R})_{j\in\{1,2\}}$ that generate $\phi_{j},$ there is a composition operation '$\#$' such that $F_1\# F_2:S^{2n+(4n+k_1+k_2)-1}\to\mathbb{R}$ generates $\phi_1\circ\phi_2$.
	\item If $\phi$ is contact isotopic to the identity, then it has a generating function.
	\item Critical points $(\zeta,\nu)\in S^{2n+k-1}\subseteq\mathbb{R}^{2n+k}$ of $F$ with value 0 correspond one-to-one to discriminant points $\zeta/|\zeta|\in S^{2n-1}$ of $\phi$.
\end{enumerate}

We build up to a definition of generating functions of contactomorphisms in several steps throughout this chapter. We start by discussing \textit{exact} symplectic structures due to their close connection to \textit{simple} generating functions. In the second section, we introduce generating functions of subsets of cotangent bundles. An identification of $T^*\R^{2n}\simeq\R^{2n}\times\R^{2n}$ then allows us to extend this notion to maps on $\R^{2n}$ in the following section. Section~4 describes a method of lifting contactomorphisms on the sphere to symplectomorphisms in Euclidean space. We conclude by defining generating functions of contactomorphisms and proving the four key results.

\section{Exact Symplectic Structures}

We now introduce a stronger type of symplectic structure that assumes the existence of a primitive of the symplectic form, generalizing the Liouville 1-form on the cotangent bundle. Exact symplectic structures are of particular interest in the context of \textit{simple} generating functions, which will be introduced in the next section. In the larger context of our proof of the main result, we will crucially need Lemma~\ref{prop:hamiso} to construct generating functions for small contact isotopies. This short exposition is based on~\cite{San14} and~\cite{MS17}.

\begin{definition}
	An \textbf{exact symplectic manifold} is a manifold $M$ with a \textbf{Liouville form} $\lambda$ such that $\omega = -d\lambda$ is a symplectic form on $M$. A symplectomorphism $\phi$ between exact symplectic manifolds $(M_1,\lambda_1)$ and $(M_2,\lambda_2)$ is \textbf{exact} if $(\phi^*\lambda_2-\lambda_1)$ is exact.
\end{definition}

Note that for a symplectomorphism $\phi$ between exact symplectic manifolds, $(\phi^*\lambda_2-\lambda_1)$ is already automatically closed.

On an exact symplectic manifold, symplectic isotopies are Hamiltonian precisely if the change in the pullback of the Liouville form is given by the exterior derivative of an action integral:

\begin{lemma}[Characterization of Hamiltonian isotopies on exact symplectic manifolds]\label{prop:hamiso}
	Let $(M,\omega=-d\lambda)$ be an exact symplectic manifold. A symplectic isotopy $(\phi_t)_{t\in[0,1]}$ starting at the identity is a Hamiltonian isotopy if and only if there exists a smooth family of functions $S_t\in C^\infty(M)$ such that $\phi_t^*\lambda-\lambda=dS_t.$ In this case, the $S_t$ are (up to a constant) given by $$S_t= \int\limits_0^t (\lambda(X_s)+H_s)\circ \phi_s ds,$$ where $X_t$ and $H_t$ are the vector field and the Hamiltonian function that generate $\phi_t$, respectively.
\end{lemma}

\begin{proof}
	We follow \cite[Lemma 2.5]{San14} and \cite[Proposition 9.3.1]{MS17}. Assume first that $\phi_t$ is a Hamiltonian isotopy generated by $X_t$, which is in turn generated by $H_t$ via $\omega(X_t,\cdot) = -dH_t.$ We calculate the change in $\phi_t^*\lambda-\lambda:$

	\begin{align*}
		\frac{d}{dt} \left(\phi_t^*\lambda - \lambda \right) &= \frac{d}{dt}\phi_t^* \lambda = \phi_t^*(\mathcal{L}_{X_t}\lambda) \\
		&= \phi_t^*\Big( d(\lambda(X_t)) + \iota_{X_t}d\lambda  \Big) \\
		&= \phi_t^* \Big( d(\lambda(X_t)) - \omega(X_t,\cdot) \Big) \\
		&= \phi_t^* \Big( d(\lambda(X_t)) + d H_t \Big) = \frac{d}{dt} dS_t.
	\end{align*}

	Here we have used the general expression for change in pullbacks along isotopies again (compare Eq.~\eqref{eq:timederivflow}), as well as Cartan's magic formula, the definition of $\omega$, the fact that $H_t$ generates $X_t$ and the definition of $S_t$ from the statement of the lemma. Since $\phi_t^*\lambda-\lambda$ and $d S_t$ both vanish for $t=0$, they must be equal at all times.

	Now assume conversely that $\phi_t^*\lambda-\lambda = d S_t$ for a symplectic isotopy $\phi_t$ generated by the vector field $X_t$ and any smooth $S_t$. Define the map $$H_t :=-\lambda(X_t) +  \Big( \frac{d}{dt} S_t \Big)\circ \phi_t^{-1}.$$ We claim that this generates $X_t$ and thereby makes $\phi_t$ a Hamiltonian isotopy. Indeed

	\begin{equation*}
		-dH_t = d\Big(\lambda(X_t)\Big) - (\phi_t^{-1})^* \frac{d}{dt} dS_t = d\Big(\lambda(X_t)\Big) - (\phi_t^{-1})^* \frac{d}{dt} \Big( \phi_t^*\lambda -\lambda \Big) =  -\iota_{X_t} d\lambda  =  \omega(X_t,\cdot),
	\end{equation*}

	where we just computed the exterior derivative, used the assumption, and then the formula for pullbacks along isotopies and Cartan's magic formulas. We can then apply the argument of the first half of the proof to show that $S_t$ must be given by the action integral, up to a constant.

\end{proof}

This means in particular:

\begin{corollary}\label{exactcoroll}
Every Hamiltonian symplectomorphism on an exact symplectic manifold is also an exact symplectomorphism.
\end{corollary}

The notion of exactness can also be extended to Lagrangian submanifolds:

\begin{definition}\label{exact-lagrangian}
	A Lagrangian submanifold $\iota:L\hookrightarrow M$ of an exact symplectic manifold $(M,\omega=-d\lambda)$ is \textbf{exact} if $\iota^*\lambda$ is exact (it is always closed).\ \\

\startsubpart{Proposition} Exactness of a Lagrangian submanifold is preserved under exact symplectomorphisms.

\end{definition}

\begin{proof}
	Assume $\phi:(M_1,\lambda_1)\to(M_2,\lambda_2)$ is an exact symplectomorphism such that $\phi^*\lambda_2-\lambda_1 = d\eta,$ and $\iota:L\hookrightarrow M_1$ is an exact Lagrangian submanifold. Its image is then embedded via $\phi\iota:L\hookrightarrow M_2$. Now $$ (\phi\iota)^*\lambda_2 = \iota^* \Big( \phi^*\lambda_2 \Big) = \iota^* \Big(  \lambda_1+d\eta  \Big) = \iota^*\lambda_1 + d\eta,$$
	so $\iota:L\hookrightarrow M_1$ is exact if and only if $\phi\iota:L\hookrightarrow M_2$ is.
\end{proof}

The canonical example of exact symplectic structures is the cotangent bundle with the Liouville form:

\begin{remark}[Exact symplectic structure on the cotangent bundle]\label{cotangent-bundle}
	Given a smooth manifold $M$, define on the cotangent bundle $\pi:T^*M\to M$ the Liouville form $$\lambda_{can}(X)=\alpha(\pi_* X)$$ for a tangent vector $X$ based at $\alpha\in T^*M$. This is the unique 1-form such that for all $\alpha\in \Omega^1(M)$,
	\begin{equation}\label{eq:liouvilleprop}
	\alpha^*\lambda_{can} = \alpha,
	\end{equation}
	where we regard $\alpha$ as a map $M\to T^*M$ when taking the pullback (see e.g. Lemma~1.4.1 of~\cite{Gei08}). This induces an exact symplectic structure $\omega_{can}=-d\lambda_{can}$ on the cotangent bundle. Note that the zero section and fibers are exact Lagrangian submanifolds. Moreover, every deformation of the zero section by a Hamiltonian diffeomorphism is also an exact Lagrangian by \ref{prop:hamiso} and \ref{exact-lagrangian}. The converse of this statement is known as the Nearby Lagrangian Conjecture.
\end{remark}

\section{Generating Functions of Subsets of $T^*M$}

Generally speaking, a generating function on some smooth manifold $M$ is a function that characterizes a subset of the cotangent bundle $T^*M$. Under sufficiently nice circumstances, this subset is automatically a Lagrangian submanifold. The aim is to have a correspondence of critical points of the generating function to intersections of the Lagrangian with the zero section, allowing us to study its geometry through Morse theory. Here, we will first consider \textit{simple generating functions} defined on $M$ itself, and then a more general framework due to H\"ormander~\cite{Hoe71} where the function is defined on a fiber bundle over $M$.\medskip

This and the following section are mostly based on expositions in \cite{San13, San14, The98, GKPS}. However, definitions and propositions have been significantly reformulated for a careful separation of generating functions of subsets of $T^*M$ and maps on $\R^{2n}$, and much more care was taken to make the statements applicable as stated to where they are used. This is necessary\footnote{Sandon acknowledges this, but only argues by analogy to the regular case.} since the lift of the contactomorphism to which this theory will be applied fails to be smooth at 0. We also add a number of proofs that were omitted in these sources.\medskip

The crucial observation to generate Lagrangian submanifolds from functions is that graphs of a closed differential form are Lagrangian. To be precise:

\begin{definition}[Simple generating forms and functions]\label{simple-gen-funs}
	We say a submanifold of $T^*M$ arises from the \textbf{simple generating form} $\alpha$ if it is given by the graph $\Gamma_\alpha$ of $\alpha$, and a submanifold arises from the \textbf{simple generating function} $f\in C^\infty(M)$ if it is given by the graph $\Gamma_{df}$ of $df$.

\ \\
\startsubpart{Proposition}
	Consider the submanifolds $\Gamma_\alpha$ generated by the 1-form $\alpha\in \Omega^1(M)$ and $\Gamma_{df}$ generated by $f\in C^\infty(M)$ and equip $T^*M$ with the canonical exact symplectic structure from Remark~\ref{cotangent-bundle}. Then the following holds:
\begin{enumerate}
	\item $\Gamma_\alpha$ is Lagrangian if and only if $\alpha$ is closed,
	\item $\Gamma_\alpha$ is exact Lagrangian if and only if $\alpha$ is exact,
	\item critical points $x$ of $f$ correspond precisely to intersections of $\Gamma_{df}$ with the zero section in $T^*M$. $x$ is non-degenerate if and only if the intersection is transverse.
\end{enumerate}

\end{definition}

\begin{proof}

	Regarding $\alpha$ as as an inclusion $M\hookrightarrow T^*M$ with image $\Gamma_\alpha$, we can use Eq.~\eqref{eq:liouvilleprop} to compute the pullback of the symplectic form along $\alpha:$
	$$\alpha^*\omega = \alpha^* d\lambda = d (\alpha^*\lambda) = d\alpha$$
	 Since the pullback of $\omega$ along an immersion vanishes if and only if the restriction to its image is zero and since $\dim\Gamma_\alpha=\frac12 \dim T^*M$, Lemma~\ref{lemma:submfd-types}~(iv,v) allow us to conclude part (i) of the statement. Similarly, $\alpha^*\lambda = \alpha$ immediately establishes part (ii).\medskip

	To show (iii), assume $f\in C^\infty(M)$ is a simple generating function, i.e. $\alpha=df.$ A critical point is a point $x\in M$ where $df$ does not have full rank, which is equivalent to $d_xf=0$ since $f$ is scalar. But this just means that $\alpha$ and the zero section intersect at $x$.

	Consider local coordinates $q\in\R^n$ around $x$ and induced coordinates $(q,p)\in\R^n\times(R^n)^*$ of $T^*M.$ The graph $\Gamma_{df}$ is given by $(q,p)$ with $p=\partial f/\partial q$ in these coordinates. The intersection with the zero section at $x$ is transverse if the tangent directions $\partial^2 f/\partial^2 q\in (R^n)^*\times(R^n)^*$ of the graph span all directions in $(R^n)^*$ when inserting arbitrary $v\in\R^n$. This is exactly the same as the Hessian matrix of $f$ having full rank, i.e. $x$ being non-degenerate.
\end{proof}

We want to consider a more general situation and enlarge the domain of our generating functions to a fiber bundle $p:E\to M.$ To relate this to the cotangent space of $M$, we consider a submanifold $N_E$ of $T^*E$ that philosophically adds an artificial parameter to the bundle $T^*M:$

\begin{definition} \label{fiber-normal-bundle}
Define for a fiber bundle $p:E\to M$ the \textbf{fiber conormal bundle} $$N_E:=\{(e,\alpha)\in T^*E\; |\; \alpha|_{\ker d_ep}=0\}.$$

\startsubpart{Proposition}
	$N_E$ is a regular coisotropic submanifold of the cotangent bundle $T^*E$ equipped with the canonical symplectic form. The symplectic reduction of $N_E$ is symplectomorphic to the cotangent bundle $T^*M$ of the base space via the map $$\Psi_E:T^*M \to (N_E)_\omega $$ that is defined by $\Psi_E(\beta)=[\beta(dp|_{TE}\cdot)]$ for all $\beta\in T^*M.$\smallskip

\startsubpart{Definition} Define the symplectic map$$
	\pi_E: N_E\to T^*M
	$$
	by setting $\pi_E:=\Psi_E^{-1}\circ\pi_0$,
where $\pi_0:N_E\to (N_E)_\omega$ is the quotient map of the symplectic reduction.
\end{definition}

\begin{proof}
It is straightforward to check in canonical coordinates of $T^*E$ that $N_E$ is a coisotropic submanifold. In particular, the distribution $N_E^\omega$ is integrable. Note that two forms $\alpha_e\in (N_E)_e$ and $\alpha'_{e'}\in (N_E)_{e'}$ lie in the same leaf if and only if $p(e)=p(e')$ and for all $X\in T_eE, X'\in T_{e'}E$ with $dpX= dpX'$ we have that $\alpha_e(X)=\alpha'_{e'}(X')$. This means leaves contain forms that differ only in their position along the artificial fiber but otherwise match each other on lifts of vectors from $TM.$

For regularity of $N_E$, we need to check the two conditions of Definition~\ref{def:regularity}. Given any $\alpha_e\in (N_E)_e$, we can take for $S$ a section $\alpha$ of $N_E$ that extends $\alpha_e.$ Transversality and the Hausdorff property of the quotient space follow in canonical coordinates by our characterization of the leaves.

By the Proposition in~\ref{def:regularity}, we can conclude that $(N_E)_\omega$ carries a smooth manifold structure and symplectic form $\bar\omega$ that is induced by the smooth quotient map. This also makes the map $\Psi_E$ smooth and symplectic. To see that it also is a diffeomorphism, one can check explicitly that it is a bijection and, in canonical coordinates, that it is a local diffeomorphism.
\end{proof}

\begin{definition}\label{def:gen-func-submnfd}
	Consider the fiber conormal bundle $N_E$ of $p:E\to M$ and let $F:E\to\R$ be differentiable. We define the set of \textbf{fiber critical points} $$\Sigma_F:= \big\{e\in E \;\big|\; e\text{ critical point of } F|_{p^{-1}(p(e))}\big\}$$ and the map\footnotemark\ 
	\begin{align*}
	i_F:\Sigma_F&\to T^*M\\
	e&\mapsto\pi_E(d_eF).
	\end{align*}
	We now say that $F$ is a \textbf{(H\"ormander) generating function} of the image $i_F(\Sigma_F).$

\end{definition}
\footnotetext{This is a well-definition since $d_eF\in N_E$ by construction of $\Sigma_F.$}

\begin{remark}[Relation to simple generating functions]
	Note that simple generating functions of Definition~\ref{simple-gen-funs} are a special case of H\"ormander generating functions: By setting $E=M$ and $p=\Id_M$, we have that $N_E=T^*M, \Sigma_F=M,$ $\pi_E=\operatorname{Id}_{T^*M}$ such that $i_F(\Sigma_F)$ is the graph of $dF.$
\end{remark}

Note that $\big\{d_eF \;\big|\; e\in\Sigma_F \big\} = dF\cap N_E.$ If this intersection is transverse, this generalized setup still yields immersed Lagrangian submanifolds:

\begin{lemma}
\label{prop:func-generates-lagrangian}
Let $F:E\to M$ be a generating function on the fiber bundle $p:E\to M.$ If $F$ is smooth and $dF$ is transverse to $N_E$ around a given point $e\in \Sigma_F$, then $i_F$ is a Lagrangian immersion around $e.$
\end{lemma}

\begin{proof}
By Proposition~\ref{simple-gen-funs} part (i), $dF$ is a Lagrangian submanifold of $T^*E$. By assumption, a neighbourhood of $d_eF$ in $dF$ intersects $N_E$ transversally. By Lemma~\ref{lem:lag-symp-red}, the image of that neighbourhood intersected with $N_E$ under the quotient map $\pi_0:N_E\to (N_E)_\omega$ is an immersed Lagrangian. Since $i_F$ is, up to identification of $e$ with $d_eF$, given by the composition of $\pi_0|_{\Sigma_F}$ with the symplectomorphism $\Psi_E^{-1},$ we are done.
\end{proof}

\begin{proposition}[Critical points correspond to intersections with the zero section]\label{prop:crit-points-correspond-to-zero-intersections}
Let $F:E\to M$ be a generating function on the fiber bundle $p:E\to M$.

\begin{enumerate}

\item $e\in E$ is a critical point of $F$ if and only if $e\in\Sigma_F$ and $i_F(e)=0,$ i.e. if $i_F(\Sigma_F)$ intersects the $0$-section of $T^*M$ at $p(e)$. If $i_F$ is injective, every intersection at $x\in M$ corresponds to exactly one critical point of $F$ in $p^{-1}(x).$

\item If $F$ is smooth around a critical point $e\in E$, $dF$ intersects $N_E$ transversally around $e$ and $U\subseteq \Sigma_F$ is a neighbourhood of $e$ small enough such that $i_F(U)\subseteq T^*M$ is a submanifold, then the following holds: The corresponding intersection of $i_F(U)$ with the zero section in $T^*_{p(e)}M$ is transverse if and only if the critical point $e$ is non-degenerate. This equivalence also holds when considering transversality and non-degeneracy along any subspace of $T^*_{p(e)}M.$
\end{enumerate}
\end{proposition}

\begin{proof} \ 

Regarding (i): Since we have $d_eF=0$ by definition for a critical point $e\in E,$ we immediately get $e\in\Sigma_F$ and $i_F(e)=0.$ If conversely $e\in\Sigma_F$ and $i_F(e)=0,$ the first condition implies that $d_eF$ vanishes in the fiber direction, while the second guarantees that $\pi_0(d_eF)=0$ for the quotient map $\pi_0:N_E\to (N_E)_\omega$. This means $d_eF$ coincides with a zero form on lifts from the base space $M$, i.e. it vanishes horizontally too and $d_eF=0$.

If $i_F$ is injective, then there can just be one $e\in\Sigma_F$ with $i_F(e)=0$ for every intersection at $p(e)=x.$\smallskip

Regarding (ii): Consider the case where $E=\R^a\times\R^b$ and write $(\zeta,\nu)\in \R^a\times\R^b.$ This case implies (ii) by imitating the following argument in local coordinates induced by coordinates on $M$:

In this situation, we have
\begin{align*}
dF&=\{(\zeta,\nu,\partial_\zeta F(\zeta,\nu), \partial_\nu F(\zeta,\nu))\in E\times E\;|\; (\zeta,\nu)\in E\},\\
N_E&=\R^a\times\R^b\times\R^a\times\{0\}\subseteq E\times E,\\
\Sigma_F &= dF\cap N_E,\\
i_F(\Sigma_F) &= \{(\zeta,\partial_\zeta F(\zeta,\nu))\in\R^a\times\R^a \;|\; (\zeta,\nu)\in\Sigma_F\}.
\end{align*}
At a critical point $(\zeta,\nu)$ of $F$, we have $\partial F/\partial(\zeta,\nu)=0$ such that $i_F(\zeta,\nu,0,0)=(\zeta,0).$ Note that this is consistent with part (i) of this proposition. The critical point is non-degenerate precisely if the maps $$
\frac{\partial F}{\partial\zeta}:\R^{a}\times\R^b\to (\R^a)^*
\qquad\text{and}\qquad
\frac{\partial F}{\partial\nu}:\R^{a}\times\R^b\to (\R^b)^*
$$
are submersions at that point. In our concrete situation, the second of these being a submersion is exactly transversality of the intersection of $dF$ and $N_E$ (compare Lemma~\ref{euclidean-formulae}), which holds by assumption. The first being a submersion is precisely transversality of the intersection of $i_F(\Sigma_F)$ and the zero section. This argument also works when considering transversality and non-degeneracy along any subspace of $T^*_{p(e)}M$.
\end{proof}

\section{Generating Functions of Maps on $\R^{2n}$}

In this section, we expand our discussion of generating functions to maps on Euclidean space. Under sufficiently nice circumstances, this map is automatically a symplectomorphism with regard to the standard symplectic structure. We will in particular see that we can relate critical points of the generating function to fixed points of the symplectomorphism and discuss a composition formula approach to proving the existence of generating functions for a given symplectomorphism.\medskip

Like the previous section, this is mostly based on expositions in \cite{San13, San14, The98, GKPS}, but definitions and propositions have been significantly reformulated and omitted proofs were added.\medskip

The central observation to extend our theory to symplectomorphisms is that some Lagrangian submanifolds of signed symplectic product spaces $\overline{M}_1\times M_2$ correspond to symplectomorphisms $\phi:M_1\to M_2$:

\begin{lemma}[Symplectomorphisms as Lagrangian submanifolds]\label{symplecto-to-lagrangian}
	For two symplectic manifolds $(M_i,\omega_i)$, $i\in\{1,2\}$, let $\overline{M}_1\times M_2$ be the product space $M_1\times M_2$ equipped with the 2-form $$\tilde{\omega}=-p_1^*\omega_1+p_2^*\omega_2,$$ where $p_i:M_1\times M_2\to M_i$ are the projection maps.

	\begin{enumerate}
	\item $\tilde{\omega}$ is a symplectic form.
	\item The graph $\Gamma_\phi$ of a smooth map $\phi:M_1\to M_2$ is an embedded submanifold of $\overline{M}_1\times M_2.$ It is a Lagrangian submanifold if and only if $\phi$ is a symplectomorphism.

	\item If $\phi$ is a symplectomorphism on $(M_1,\omega_1)=(M_2,\omega_2)$, then the diagonal $\Delta$ in $\overline{M}_1\times M_2$ is another Lagrangian submanifold. Fixpoints of $\phi$ then correspond precisely to points of the Lagrangian intersection $\Gamma_\phi\cap\Delta$. This intersection is transverse exactly when the fixed point is non-degenerate.
	\end{enumerate}
\end{lemma}

\begin{proof}
We recall standard arguments, see e.g. Proposition 3.8 of~\cite{Can03}.

Regarding (i): $\tilde{\omega}$ is a smooth 2-form by construction, closed by compatibility of pullback and exterior derivative, and non-degenerate by a straightforward computation.

Regarding (ii): A smooth submanifold embedding is given by the map $\iota: M_1\to \overline{M}_1\times M_2$ defined as $p\mapsto (p,\phi(p)).$ Note that since $p_1\iota=\Id_{M_1}$ and $p_2\iota=\phi,$ we have
$$\iota^*\tilde{\omega} = -\omega_1 + \phi^* \omega_2.$$

Now $\Gamma_\phi$ is Lagrangian if and only if $\iota^*\tilde{\omega}=0$, which by the above is exactly the case when $\omega_1=\phi^*\omega_2.$

Regarding (iii): The diagonal is Lagrangian: For the smooth embedding $\iota':M_1\to\overline{M}_1\times M_1$ defined by $p\mapsto(p,p)$ we have $p_1\iota'=p_2\iota'=\Id_{M_1},$ so in particular $$\iota'^*\tilde{\omega} = -\omega_1 + \omega_1 = 0.$$ The rest of the statement follows immediately by definition of (non-degenerate) fixed points.

\end{proof}

Say we have a symplectic identification of a neighbourhood of the graph of $\phi$ with an open set in the cotangent bundle in a way that maps the diagonal onto the zero section. If the image of the graph is a submanifold generated by some generating function $F$, then we can say that $F$ also generates $\phi$. In the compact case, the Weinstein Lagrangian neighbourhood Theorem~\ref{thm:weinstein-lagrangian} can provide such an identification for $\phi$ $C^1$-small. We instead follow~\cite{GKPS} and choose an explicit global identification specifically for Euclidean space:

\begin{remark}[Symplectic identification for Euclidean space]\label{rem:tau}
We can define a symplectomorphism $\tau:\overline{\R^{2n}}\times\R^{2n}\to T^*\R^{2n}$ from the signed symplectic product in the sense of Lemma~\ref{symplecto-to-lagrangian} to the cotangent space of $\R^{2n}$ with the canonical symplectic form by setting
$$\tau(x,y,X,Y)=\left( \frac{x+X}{2}, \frac{y+Y}{2}, Y-y, X-x \right).$$
In complex notation, this reads
$$ \tau(z,Z)=\left( \frac{z+Z}{2}, i(z-Z)  \right).$$
The diagonal in $\overline{\R^{2n}}\times\R^{2n}$ is mapped to the zero section in $T^*\R^{2n}.$
\end{remark}

\begin{definition}
Let $\phi$ be a map on $\R^{2n}.$ We say a differentiable function $F:\R^{2n}\times\R^k\to\R$ defined on the trivial vector bundle $p:\R^{2n}\times\R^k\to\R^{2n}$ is a \textbf{(H\"ormander) generating function} of $\phi$ if $$i_F(\Sigma_F) = \tau(\Gamma_\phi) \subseteq T^*\R^{2n},$$ i.e. it generates the graph $\Gamma_\phi\subseteq\overline{\R^{2n}}\times\R^{2n}$ of $\phi$ up to identification via $\tau$. If $k=0,$ we say $F$ is a \textbf{simple} generating function of $\phi.$
\end{definition}

\begin{remark}[Relation to other definitions in the literature]\label{remark-def-gen-euclidean}

Note that our \textit{simple} generating functions are essentially a coordinate-free version of the \textit{generating functions of type V} in~\cite{MS17}. They were used in Chaperon's proof of the Conley-Zehnder theorem, and using a Weinstein neighbourhood instead of our identification $\tau$ allows one to use analogous generating functions for the $C^1-$small versions of the Arnol'd conjecture and its contact version (see e.g. \cite{San13}).\smallskip

Our definition of \textit{H\"ormander} generating functions is somewhat more general than those typically found in the literature (e.g. \cite{The98,San13,GKPS}). These also assume for generating functions of subsets of $T^*M$ that $F$ is smooth and $dF$ intersects $N_E$ transversally. We have seen in Lemma~\ref{prop:func-generates-lagrangian} that this turns $i_F$ into a Lagrangian immersion. For generating functions of maps on Euclidean space, the literature commonly requires that $i_F$ is even an embedding\footnote{Strictly speaking, \cite{San13} only requires the image of $i_F$ to be an embedded submanifold.}. The reason for our more general definition is that we need to apply results to generating functions of lifted contactomorphisms, which do not satisfy any of these strong assumptions. We can therefore only speak of generated \textit{subsets} and \textit{maps} instead of \textit{Lagrangian submanifolds} and \textit{symplectomorphisms}, respectively. While \cite{The98} and \cite{San13} argue mostly by analogy that necessary results essentially continue to hold, we separate granularly which additional assumptions are needed for any given part of the argument.\smallskip

In particular, we will see that injectivity of $i_F$ is needed for a one-to-one correspondence of critical points of $F$ to fixed points of $\phi$ (Prop.~\ref{prop:crits-are-fixed}(i)). The transversality of $dF$ and $N_E$ is required both for notions of non-degeneracy to match under this correspondence (Prop.~\ref{prop:crits-are-fixed}(iii)) and for the generated map to preserve the symplectic form (Lemma~\ref{prop:preserve-symplecto}).

\end{remark}

\begin{lemma}\label{prop:preserve-symplecto}
Let $F:\R^{2n}\times\R^k\to\R$ be a generating function of a map $\phi:\R^{2n}\to\R^{2n}.$ If $F$ is smooth and $dF$ transverse to $N_{\R^{2n}\times\R^k}$ around some point $(\zeta,\nu)\in\R^{2n}\times\R^k,$ then $\phi$ is smooth and preserves the standard symplectic form around $\zeta$.
\end{lemma}

\begin{proof}
By applying Lemma~\ref{symplecto-to-lagrangian}~(ii) locally, we only need to show that the graph $\zeta\mapsto(\zeta,\phi(\zeta))$ is a Lagrangian immersion. Since $\tau$ is a symplectomorphism and $i_F(\Sigma_F)=\tau(\Gamma_\phi),$ this immediately follows from Lemma~\ref{prop:func-generates-lagrangian}.
\end{proof}

Now that we restrict ourselves to the Euclidean setting, we can write down concrete expressions for the objects used to define generating functions:

\begin{lemma}[Formulae for the Euclidean setting]\label{euclidean-formulae}
Let $F:\R^{2n}\times\R^k\to\R$ be a generating function of a map $\phi$ on $\R^{2n}$ with respect to the trivial fiber bundle  $p:\R^{2n}\times\R^k\to\R^{2n}.$ We then have
\begin{align*}
\Sigma_F &= \left\{(\zeta,\nu)\in\R^{2n}\times\R^k \;|\; {\textstyle\frac{\partial F}{\partial\nu}}(\zeta,\nu)=0\right\},\\
i_F(\zeta,\nu) &= \left(\zeta,\; {\textstyle\frac{\partial F}{\partial\zeta}}(\zeta,\nu)\right),
\end{align*}
and the intersection of $dF$ and $N_{\R^{2n}\times\R^k}$ at $(\zeta,\nu)\in\Sigma_F$ is transverse if and only if $\frac{\partial F}{\partial\nu}:\R^{2n}\times\R^k\to (\R^k)^*$ is a submersion at $(\zeta,\nu).$
\end{lemma}

\begin{proof}
The first equation follows by the definition of $\Sigma_F$ since
the total derivative of $F|_{p^{-1}(p(\zeta,\nu))}$ can be canonically identified with $\frac{\partial F}{\partial\nu}.$

By definition, $i_F(e):=\Psi_E^{-1}\pi_0(d_eF),$ where $\pi_0:N_E\to (N_E)_\omega$ is the quotient map of the symplectic reduction. For $E=\R^{2n}\times\R^k$ and $e=(\zeta,\nu)\in \Sigma_F,$ it follows that $d_e F = (\zeta,\nu,\frac{\partial F}{\partial\zeta}, 0)\in T^*E=\R^{2n}\times\R^k\times\R^{2n}\times\R^k$.  The second equation of the lemma follows since $\pi_0$ projects out the last factor and $\Psi_E^{-1}$ the second.

The intersection of $dF$ and $N_E=\R^{2n}\times\R^k\times\R^{2n}\times\{0\}$ at $e\in\Sigma_F$ is transverse iff these spaces together span all directions of $T^*E.$ $N_E$ clearly spans all but the last factor. So we have transversality iff the tangent spaces of $dF$ restricted to that factor, i.e. the image of $d(\frac{\partial F}{\partial\nu})$, span all of $\R^k.$ This is exactly regularity of $\frac{\partial F}{\partial\nu}.$
\end{proof}

\begin{proposition}[Critical points correspond to fixed points]\label{prop:crits-are-fixed}
Let $F:\R^{2n}\times\R^k\to\R$ be a generating function of a map $\phi$ on $\R^{2n}$.
\begin{enumerate}
\item $(\zeta,\nu)$ is a critical point of $F$ if and only if $\zeta$ is a fixed point of $\phi$. If $i_F$ is injective, then this is a one-to-one correspondence.
\item If $F$ is smooth around a critical point $(\zeta,\nu)\in\R^{2n}\times\R^k$ of $F$, and $dF$ intersects $N_{\R^{2n}\times\R^k}$ transversally at $d_{(\zeta,\nu)} F$, then $(\zeta,\nu)$ is a non-degenerate critical point if and only if $\zeta$ is a non-degenerate fixed point of $\phi$. This equivalence also holds along any subspace of $T^*_{p(e)}M$.
\end{enumerate}
\end{proposition}

\begin{proof}

Regarding (i): Fixed points $\zeta\in\R^{2n}$ of $\phi$ correspond one-to-one to points $(\zeta,\zeta)$ where the graph $\Gamma_\phi$ intersects the diagonal. Under the identification map $\tau,$ this corresponds to points $(\zeta,0)$ in $T^*\R^{2n},$ i.e. intersections with the zero section. The statement now follows by Proposition~\ref{prop:crit-points-correspond-to-zero-intersections}~(i), which tells us that these intersections correspond to critical points of $F$ and that this correspondence is one-to-one if $i_F$ is injective.

Regarding (ii): Under the identifications via $\tau$ and $i_F$, transversality of an intersection is preserved. The statement follows by Proposition~\ref{prop:crit-points-correspond-to-zero-intersections}~(iii).
\end{proof}

A central tool to construct generating functions is the following:

\begin{proposition}[Composition formula]\label{theret-comp}
Let $F_1:\R^{2n}\times\R^{k_1}\to\R$ and $F_2:\R^{2n}\times\R^{k_2}\to\R$ be generating functions of the maps $\phi_1$ and $\phi_2$ on $(\R^{2n},\omega_\text{std})$ and define
\begin{align*}
F_1\# F_2:\R^{2n+2n+2n+k_1+k_2}&\to\R \\
(q;\zeta_1,\zeta_2,\nu_1, \nu_2) &\mapsto F_1(\zeta_1,\nu_1) + F_2(\zeta_2,\nu_2)-2\;\omega_\text{std}(\zeta_1-q, \zeta_2-q).
\end{align*}

\begin{enumerate}
\item $F_1\# F_2$ is a generating function for the composition $\phi:=\phi_2\phi_1.$
\item If $i_{F_1}$ and $i_{F_2}$ are injective, then $i_{F_1\#F_2}$ is as well.
\item Let $(q,\zeta_1,\zeta_2,\nu_1, \nu_2)\in\Sigma_{F_1\#F_2}$ and assume for $j=1,2$ that around $(\zeta_j,\nu_j),$ $F_j$ is smooth, $dF_j$ intersects $N_{\R^{2n}\times\R^{N_j}}$ transversally. Then around $(q,\zeta_1,\zeta_2,\nu_1, \nu_2)\in\Sigma_{F_1\#F_2},$ $F_1\#F_2$ is smooth and $d(F_1\#F_2)$ intersects $N_{\R^{2n}\times\R^{2n}\times\R^{2n}\times\R^{k_1}\times\R^{k_2}}$ transversally.
\end{enumerate}
\end{proposition}

\begin{proof}
We roughly follow the proof of Proposition~2.2 in~\cite{GKPS}. For brevity, we write $F:=F_1\# F_2.$ Note that $F$ is differentiable since $F_1$ and $F_2$ are.\medskip

Before proving parts (i)-(iii), we first find more convenient characterizations of the elements in $\Sigma_F$ and the map $i_F$.
Compute the components of the vertical derivative of $F$ at $(q,\zeta_1,\zeta_2,\nu_1, \nu_2)$:
\begin{align*}
\frac{\partial F}{\partial\zeta_1} &= \frac{\partial F_1}{\partial\zeta_1}+ 2i(\zeta_2-q), & \frac{\partial F}{\partial \nu_1} &= \frac{\partial F_1}{\partial \nu_1}, \\
\frac{\partial F}{\partial\zeta_2} &= \frac{\partial F_2}{\partial\zeta_2}- 2i(\zeta_1-q), & \frac{\partial F}{\partial \nu_2} &= \frac{\partial F_2}{\partial \nu_2}.
\end{align*}
By Lemma~\ref{euclidean-formulae}, we therefore have
\begin{equation}\label{eq:fiber-critical-criterion1}
(q,\zeta_1,\zeta_2,\nu_1, \nu_2)\in \Sigma_F \quad\iff\quad
\begin{cases}
(\zeta_1,\nu_1)\in\Sigma_{F_1} \text{ and } (\zeta_2,\nu_2)\in\Sigma_{F_2},\\
\frac{\partial F_1}{\partial\zeta_1}= -2i(\zeta_2-q),\\
\frac{\partial F_2}{\partial\zeta_2}= 2i(\zeta_1-q).
\end{cases}
\end{equation}
Since the $F_j$ are generating functions of $\phi_j$, $(\zeta_j,\nu_j)\in\Sigma_F$ is equivalent to the existence of $z_j\in \R^{2n}$ such that $(\zeta_j,\nu_j)\in i_{F_j}^{-1}(\tau(z_j,\phi_j(z_j))).$ By Lemma~\ref{euclidean-formulae} and definition of $\tau$, the latter is equivalent to
\begin{equation}\label{eq:helperle}
\left(\zeta_j,{\textstyle\frac{\partial F_j}{\zeta_j}(\zeta_j,\nu_j)}\right)
=
\left(\frac{z_j+\phi_j(z_j)}{2}, i(z_j-\phi_j(z_j))  \right).
\end{equation}
Using this equation, we can reformulate the remaining conditions on the right hand side of Eq.~\eqref{eq:fiber-critical-criterion1} in a straightforward computation to get\footnote{Note that in the original proof from \cite{GKPS}, $(\zeta_j,\nu_j)$ and $z_j$ are related via a diffeomorphism. We have this weaker relationship because we do not require that $i_F$ has to be an embedding in the definition of generating functions.}
\begin{equation}\label{eq:fiber-critical-criterion2}
(q,\zeta_1,\zeta_2,\nu_1, \nu_2)\in \Sigma_F \;\iff\;
 \exists z_1,z_2\in \R^{2n}:
\begin{cases}
\forall j\in\{1,2\}:\; (\zeta_j,\nu_j)\in i_{F_j}^{-1}\big(\tau(z_j,\phi_j(z_j))\big),\\
q= (z_1+\phi(z_1))/2,\\
z_2= \phi_1(z_1).
\end{cases}
\end{equation}
We can also use Eq.~\eqref{eq:helperle} and eliminate $z_2$ to restate this in a more convenient but less suggestive way:
\begin{equation}\label{eq:fiber-critical-criterion3}
(q,\zeta_1,\zeta_2,\nu_1, \nu_2)\in \Sigma_F \quad\iff\quad
 \exists z_1\in \R^{2n}:
\begin{cases}
q= (z_1+\phi(z_1))/2,\\
\zeta_1 = (z_1+\phi_1(z_1))/2,\\
\zeta_2 = (\phi_1(z_1)+\phi(z_1))/2,\\
{\textstyle\frac{\partial F_1}{\partial\zeta_1}(\zeta_1,\nu_1)} = i(z_1-\phi_1(z_1)),\\
{\textstyle\frac{\partial F_2}{\partial\zeta_2}(\zeta_2,\nu_2)} = i(\phi_1(z_1)-\phi(z_1)).
\end{cases}
\end{equation}

To express $i_F$ more conveniently, we first compute the horizontal derivative of $F$:
$$
\frac{\partial F}{\partial q}(q,\zeta_1,\zeta_2,\nu_1, \nu_2) = 2i(\zeta_1-\zeta_2).
$$
For $(q,\zeta_1,\zeta_2,\nu_1,\nu_2)\in\Sigma_F$, Eq.~\eqref{eq:fiber-critical-criterion3} with $z_1:= i_{F_1}(\zeta_1,\nu_1)$ and Lemma~\ref{euclidean-formulae} then give
\begin{equation}\label{eq:iF-nicer}
i_F(q,\zeta_1,\zeta_2,\nu_1,\nu_2) = \left(\frac{z_1+\phi(z_1)}{2},\; i(z_1-\phi(z_1)) \right) = \tau(z_1,\phi(z_1)).
\end{equation}

Regarding (i): We need to show that $i_F(\Sigma_F) = \tau(\Gamma_\phi).$ Eq.~\eqref{eq:iF-nicer} gives $i_F(\Sigma_F) \subseteq \tau(\Gamma_\phi)$ immediately. In the opposite direction, we can write any element of $\tau(\Gamma_\phi)$ as $\tau(z_1,\phi(z_1))$ for some $z_1$ and set $z_2:=\phi_1(z_1)$. Since the $F_j$ are generating functions, we can then find $(\zeta_j,\nu_j)$ and $q$ that satisfy the right hand side of Eq.~\eqref{eq:fiber-critical-criterion2}. Again by Eq.~\eqref{eq:iF-nicer}, this yields a preimage $(q,\zeta_1,\zeta_2,\nu_1,\nu_2)$ of $\tau(z_1,\phi(z_1))$ under $i_F.$\medskip

Regarding (ii):
Assume $i_F(q,\zeta_1,\zeta_2,\nu_1,\nu_2)=i_F(q',\zeta'_1,\zeta'_2,\nu'_1,\nu'_2).$ $i_F$ preserves fibers by construction, so we have $q=q'.$ Considering Eq.~\eqref{eq:iF-nicer}, injectivity of $i_{F_1}$ gives $(\zeta_1,\nu_1)=(\zeta'_1,\nu'_1)$. By Eq.~\eqref{eq:fiber-critical-criterion2}, $$i_{F_2}(\zeta_2,\nu_2)=\tau(z_2,\phi_2(z_2))=i_{F_2}(\zeta'_2,\nu'_2)$$ must then hold for $z_2=\phi_1(z_1).$ Injectivity of $i_{F_2}$ thus yields $(\zeta_2,\nu_2)=(\zeta'_2,\nu'_2)$.\medskip

Regarding (iii): Smoothness at $(q,\zeta_1,\zeta_2,\nu_1, \nu_2)$ follows immediately. To check the remaining properties, consider for $j=1,2$ the map $$
H_j:
\Sigma_{F_j} \xrightarrow{i_{F_j}}
i_{F_j}(\Sigma_{F_j}) \xrightarrow{\tau^{-1}|_{i_{F_j}(\Sigma_{F_j})}}
\Gamma_{\phi_j} \xrightarrow{(z,z')\mapsto z}
\R^{2n}.
$$
Intuitively, it gives the point $z_j$ whose image under $\phi_j$ is determined by the fiber critical point $(\zeta_j,\nu_j).$ By Lemma~\ref{euclidean-formulae} and the definition of $\tau$, this map is given by
$$(\zeta_j,\nu_j)\mapsto \zeta_j + \frac{1}{2i}\frac{F_j}{\partial \zeta_j}(\zeta_j,\nu_j).$$

The map $H_j$ is a submersion for dimensional reasons as each of the maps we used to define it is an immersion (e.g. by Lemma~\ref{prop:func-generates-lagrangian} for $i_{F_j}$).
It follows that the derivative of this map, defined on $T_{((\zeta_j,\nu_j))}\Sigma_{F_j}=\ker(d_{(\zeta_j,\nu_j)}(\frac{\partial F_j}{\partial \nu_j}))$, is surjective. By transversality and Lemma~\ref{euclidean-formulae}, we furthermore have that the derivative of $\frac{\partial F_j}{\partial \nu_j}$ is surjective.

Together, this implies that the matrices $$
M_j :=
\begin{pmatrix}
\frac{1}{2i} \frac{\partial^2 F_j}{\partial\zeta_j^2} + \Id_{\R^{2n}}
& \frac{1}{2i} \frac{\partial^2 F_j}{\partial\nu_j\partial\zeta_j} \\
\frac{\partial^2 F_j}{\partial\zeta_j\partial\nu_j}
& \frac{\partial^2 F_j}{\partial\nu_j^2} \\
\end{pmatrix}
$$
are surjective: The second row is surjective as the derivative of $\frac{\partial F_j}{\partial \nu_j},$ and the first row is the derivative of $H_j$ extended to all of $\R^{2n+k_1}$, which we know to be surjective on the kernel of the second row.

To show that we have a transverse intersection of $d(F_1\#F_2)$ and $N_{\R^{2n}\times\R^{2n}\times\R^{2n}\times\R^{k_1}\times\R^{k_2}}$ at $(q,\zeta_1,\zeta_2,\nu_1, \nu_2),$ by Lemma~\ref{euclidean-formulae} we need to show that the derivative
$$
\frac{\partial (F_1\# F_2)}{\partial(q,\zeta_1,\zeta_2,\nu_1,\nu_2)}
=
\left(
\frac{\partial F_1}{\partial\zeta_1} + 2i(\zeta_2-q),
\frac{\partial F_2}{\partial\zeta_2} - 2i(\zeta_1-q),
\frac{\partial F_1}{\partial\nu_1},
\frac{\partial F_2}{\partial\nu_2},
\right)
$$
is a submersion, i.e.
$$
\left(
\begin{array}{c|cc|cc}
-2\Id_{\R^{2n}} & \frac{\partial^2F_1}{\partial\zeta_1^2} & 2i\Id_{\R^{2n}} & \frac{\partial^2 F_1}{\partial\nu_1\partial\zeta_1} & 0
\\
2\Id_{\R^{2n}} & -2i\Id_{\R^{2n}} & \frac{\partial^2F_2}{\partial\zeta_2^2} & 0 & \frac{\partial^2 F_2}{\partial\nu_2\partial\zeta_2}
\\
\hline
0 & \frac{\partial^2 F_1}{\partial\nu_1\partial\zeta_1} & 0 & \frac{\partial^2 F_1}{\partial\nu_1^2} & 0
\\
0 & 0 & \frac{\partial^2 F_2}{\partial\nu_2\partial\zeta_2} & 0 & \frac{\partial^2 F_2}{\partial\nu_2^2}
\end{array}
\right)
$$
is surjective. This now follows since we can use elementary row and column operations to bring it into the form
$$
\left(
\begin{array}{c|cc}
* & M_1 & 0 \\
* & * & M_2
\end{array}
\right).
$$

\end{proof}

\begin{remark}[Other composition formulas]
Note that other composition formulas can be obtained for different choices of the identification $\tau$ and with different dimensions of the domain of $F_1\# F_2$, compare e.g.~\cite{The98, San13, San14}. The formula above, which was first introduced in~\cite{GKPS}, has the advantage of being almost symmetric in $F_1$ and $F_2$ and not requiring one of the two to be a simple generating function in order for Proposition~\ref{theret-comp}(iii) to hold. However, it does exhibit a faster growth of the dimension of the domain than other choices.\\[-5mm]
\end{remark}

\section{Lifting Contactomorphisms from $S^{2n-1}$ to $\R^{2n}$}\label{sec:lifting}

We want to pull the notion of generating function on Euclidean space back to the sphere in order to detect translated points. To achieve this, we follow~\cite{San13} and consider a \textit{lift} of contactomorphisms on the sphere. This essentially glues in one additional point into the symplectization of the sphere at the cost of smoothness in the origin:\\[-5mm]

\begin{definition}[Lifting contactomorphisms and contact isotopies from $S^{2n-1}$ to $\R^{2n}$]\label{lift-sphere}
	Let $\phi$ be a contactomorphism on $(S^{2n-1},\alpha)$ and $\phi_t$ a contact isotopy generated by a Hamiltonian function $H:S^{2n-1}\times\R\to\R$. We define the \textbf{lifts} of $\phi$ and $H$, respectively, as
\begin{align*}
	\lift \phi:\R^{2n}&\to\R^{2n}, & \lift H: \R^{2n}\times\R&\to\R. \\
	z&\mapsto \begin{cases}
		|z| e^{-\frac12 g(z/|z|)} \phi\left(\frac{z}{|z|}\right) & z \neq 0 \\
		0 & z = 0
	\end{cases}
		& (z,t) &\mapsto |z|^2 H_t(z/|z|)
\end{align*}

\startsubpart{Proposition}
\begin{enumerate}
	\item $\lift\phi$ is continuous and $\R_+$-equivariant. Outside of 0, it is a smooth symplectomorphism with respect to the standard symplectic structure $\omega_\text{std}$ on $	\R^{2n}$.

	\item $\lift H_t(z)$ is homogeneous of degree 2 in $z$, i.e. $\lift H_t(\lambda z)=\lambda^2\, \lift H_t(z)$ for $\lambda\geq0$, and continuous everywhere. Outside of $z=0$ it is smooth and generates the lifts of $\phi_t$, i.e. $$\omega_\text{std}\left( \frac{d}{dt}\lift \phi_t ,\cdot\right) = -d(\lift H_t).$$

	\item Discriminant points $q$ of $\phi$ correspond to radial lines $\R_+ q$ of fixed points of the lift $\lift\phi$. $q$ is non-degenerate precisely when $p$ is non-degenerate in the directions tangent to the sphere of radius $|p|$ for any (equivalently all) points $p\in\R_+q.$

\end{enumerate}
\end{definition}

\begin{proof}

	We will show that outside of zero, these maps are the symplectization from~\ref{def:symplectization} and~\ref{def:symplectization-ham}
 under the identification along
	\begin{align*}
		\psi: \symp(S^{2n-1},\alpha) &\to (\R^{2n}\setminus \{0\},\omega_\text{std}) \\
		(q,\theta) &\mapsto \sqrt{2} e^{\theta/2}q.
	\end{align*}
	This is clearly a diffeomorphism. To see that it is a symplectomorphism, we need to verify that $$\psi^* \omega_\text{std}=d (e^\theta \alpha).$$ Since $\alpha$ is the restriction $\iota^*\lambda$ of the Liouville form $\lambda$ along $\iota:S^{2n-1}\hookrightarrow \C^n$ and $\omega_\text{std} = d\lambda,$ it suffices to verify $$e^\theta \iota^*\lambda = \psi^*\lambda$$ by inserting $$\lambda = -i|z|\left( \bar{z}d\frac{z}{|z|} - zd\frac{\bar{z}}{|z|} \right)$$
	in complex coordinates of $\R^{2n}\simeq\C^n.$ We indeed have
	\begin{align*}
		\psi\big(\symp \phi(\psi^{-1}(z))\big) &= \psi\big(\symp\phi(z/|z|,\, 2\ln |z|)\big) \\
		&= \psi\big(\phi(z/|z|),\, \ln |z| \,- g(z/|z|)/2\big) \\
		&= |z|\; e^{-g(z/|z|)/2}\,\phi\big(z/|z|\big) \;=\; \lift\phi(z)
	\end{align*}
	and $$\lift H_t(\psi(q,\theta))=e^\theta H_t(q) = \symp H_t(q,\theta).$$

	Regarding (i): $\R_+$-invariance is immediate by the definition. We only need to check continuity in 0: Since $\phi$ and $g$ are defined on a compact set, they have a minimal and maximal value, so that $\lift\phi(z)$ tends to zero as $|z|$ does.
	Outside zero, it is a smooth symplectomorphism as it is a composition $\psi\circ \symp\phi\circ \psi^{-1}$ of symplectomorphisms.

	Regarding (ii): Homogeneity is immediate, and continuity in zero again follows since $H_t$ is defined on a compact set. To conclude that $\lift H_t$ generates $\lift \phi_t$ outside zero, apply Def.~\& Prop.~\ref{def:symplectization-ham} after identifying the symplectization along $\psi$ as above.

	Regarding (iii): Immediate by Proposition~\ref{char-pts-via-symp}~(i) and the identification along~$\psi.$
\end{proof}

\section{Generating Functions of Contactomorphisms on $S^{2n-1}$}

We extend our notion of generating functions to contactomorphisms on spheres via the lifts from Section~\ref{sec:lifting}. The key observation is the following: As we will see in more detail in this section, applying a general approach to proving existence using the composition formula to lifts $\mathcal{\tilde S}\phi$ of contactomorphisms $\phi$ always yields generating functions $G:\R^{2n+k}\to\R$ that are homogeneous of degree two, i.e. $G(\lambda z)=\lambda^2\, G(z)$ for $\lambda\geq0.$ Such functions are in particular determined by their restriction to the unit sphere, so it stands to reason to call this restriction $F=G|_{S^{2n+k-1}}$ the generating function of the contactomorphism $\phi$ on the sphere.\medskip

\begin{definition}[Generating functions on $S^{2n-1}$]\label{def-gen-on-sphere}
Let $\phi$ be a contactomorphism on $(S^{2n-1},\alpha_\text{std}).$ For every function $F:S^{2n+k-1}\to \R,$ we define the \textbf{extension} $\hat{F}:\R^{2n}\times\R^k\to\R$ by $\hat{F}(\lambda x)=\lambda^2 F(x)$ for $\lambda>0, x\in S^{2n+k-1}$ and $\hat{F}(0)=0$.

We say $F$ is a \textbf{generating function} of $\phi$ if
\begin{enumerate}
\item $\hat F$ is differentiable and a generating function of the lift $\lift\phi:\R^{2n}\to\R^{2n}$ of $\phi$,
\item $i_{\hat F}$ is injective and $i_{\hat F}(0)=0$.
\item Around any point $(\zeta,\nu)\in\Sigma_{\hat F}\setminus\{0\},$ $\hat F$ is smooth and the intersection of $d\hat F$ and $N_{\R^{2n}\times\R^k}$ is transverse.
\item $\hat F$ has a Lipschitz differential everywhere.
\end{enumerate}
$F$ is a \textbf{simple} generating function if additionally $k=0.$

\end{definition}

\begin{remark}[Remarks on Definition~\ref{def-gen-on-sphere}]\ 
\begin{enumerate}
\item The typical approach in the literature is to show that the lift of a contactomorphism on a sphere or lens space can be generated by an appropriate notion of \textit{conical} generating function~\cite{The98,San13,GKPS}. We instead put the focus on the restriction of this function to the sphere. We do not lose any information by this as it still determines the conical generating function of the lifted symplectomorphism by homogeneous extension.
I prefer this approach as it allows for a simpler and more lucid formulation of the results leading up to the main theorem.

\item We took a bare-minimum approach in defining generating functions of maps on \textit{Euclidean} space for the reasons outlined in Remark~\ref{remark-def-gen-euclidean}. Definition~\ref{def-gen-on-sphere} on the other hand is fitted to our specific situation by incorporating the stronger assumptions \ref{def-gen-on-sphere}.(ii)-(iv). This allows for a simpler formulation of the upcoming propositions.

\end{enumerate}
\end{remark}

In order to show existence, we first translate the composition formula to the spherical setting and subsequently consider the $C^2-$small case:

\begin{proposition}[Composition formula]\label{comp-form-sphere}
Let $F_1:S^{2n+k_1-1}\to\R$ and $F_2:S^{2n+k_2-1}\to\R$ be generating functions of the contactomorphisms $\phi_1$ and $\phi_2$ on $(S^{2n-1},\alpha_\text{std}).$ Then the function $F_1\#F_2:S^{6n+k_1+k_2-1}\to\R$ defined by $$F_1\#F_2(x):=(\hat F_1 \# \hat F_2)|_{S^{6n+k_1+k_2-1}}(x),$$ i.e. the restriction of the Euclidean composition formula from Proposition~\ref{theret-comp} applied to the extensions $\hat F_1$ and $\hat F_2$, is a generating function for the composition $\phi:=\phi_2\phi_1.$
\end{proposition}

\begin{proof}
For brevity, we write $F:=F_1\# F_2$. First note that $\hat F=\hat F_1\# \hat F_2.$ This follows since both sides coincide on the unit sphere and are homogeneous functions of degree 2 for a positive real factor: The former by construction, the latter since $\hat F_1$ and $\hat F_2$ are and the Euclidean composition formula preserves this property.\medskip

We now need to check each of the four conditions in Definition~\ref{def-gen-on-sphere} for $F$:\medskip

Regarding (i): $\hat F = \hat F_1\# \hat F_2$ is a generating function of $\lift \phi_2 \lift \phi_1$ by part (i) of Proposition~\ref{theret-comp}. By functoriality of the lift this is just $\lift (\phi_2\phi_1) = \lift \phi.$\medskip

Regarding (ii): Injectivity of $i_{\hat F} = i_{\hat F_1\# \hat F_2}$ follows by part (ii) of Proposition~\ref{theret-comp}.
Adapting Eq.~\eqref{eq:iF-nicer} from the proof of Proposition~\ref{theret-comp} to our situation gives
$$
i_{\hat F}(q,\zeta_1,\zeta_2,\nu_1,\nu_2) = \tau(z_1,\lift \phi(z_1)) \qquad \text{ where }
z_1=i_{\hat F_1}(\zeta_1,\nu_1).
$$ As $i_{\hat F_1}(0)=0,$ $\lift\phi(0)=0$ and $\tau(0,0)=0,$ this indeed yields $i_{\hat F}(0)=0.$\medskip

Regarding (iii):
Let $p=(q,\zeta_1,\zeta_2,\nu_1,\nu_2)\in\Sigma_{\hat F}\setminus\{0\}.$ By Eq.~\eqref{eq:fiber-critical-criterion1}, $(\zeta_j,\nu_j)\in\Sigma_{\hat F_j}$ for $j=1,2.$ If $(\zeta_j,\nu_j)\neq(0,0)$ for both $j$ then the desired properties at $p$ follow immediately by part (iii) of Proposition~\ref{theret-comp} since $F_j$ are generating functions on the sphere. 

Assume instead towards contradiction that $\exists j_0\in\{1,2\}: (\zeta_{j_0},\nu_{j_0})=(0,0)$. Then by Eq.~\eqref{eq:fiber-critical-criterion2} and the fact that $F_{j_0}$ is a generating function on the sphere, there exist $z_1,z_2\in\R^{2n}$ such that $$0=i_{\hat F_{j_0}}(0,0)=\tau(z_{j_0},\lift\phi_{j_0}(z_{j_0}))\qquad\text{and}\qquad z_2=\lift\phi_1(z_1).$$ In particular $z_j=0,$ and since $\lift\phi_1$ is a bijection that preserves zero, this implies that both $z_1$ and $z_2$ are zero. Again by Eq.~\eqref{eq:iF-nicer}, this yields $i_{\hat F}(p)=0.$ The contradiction $p=0$ follows by part (ii) of this proof.\medskip

Regarding (iv): The differential of $\hat F$ is made up of the differentials of $\hat F_1$ and $\hat F_2,$ which are Lipschitz by assumption, and that of the bilinear map $\omega_\text{std}(\zeta_1-q,\zeta_2-q),$ so that it is Lipschitz itself.

\end{proof}

\begin{lemma}\label{lemma-c2-small-existence}
There is a neighbourhood $U$ of the identity in the $C^2$-topology on smooth maps on $S^{2n-1}$ such that any time-1 map $\phi$ of a contact isotopy starting at the identity on $(S^{2n-1},\alpha_\text{std})$ and remaining within $U$ has a simple generating function $F:S^{2n-1}\to \R$.
\end{lemma}

\begin{proof}
Denote the contact isotopy $\phi_t.$ The graph $\Gamma_{\lift\phi}\subseteq \R^{2n}\times\R^{2n}$ of the lift $\lift\phi$ and the diagonal in $\R^{2n}\times\R^{2n}$ can both be canonically identified with $\R^{2n}.$ The projection of the graph onto the diagonal is then the map on $\R^{2n}$ defined by $\Phi(0)=0$ and
\begin{equation}\label{eq:defPhi}
\Phi(z):= \frac{z+\lift\phi(z)}{2} = \frac{|z|}{2}\left(
\frac{z}{|z|}
+e^{-\frac12 g(z/|z|)}\phi\left(\frac{z}{|z|}\right)
\right) \qquad\text{for }z\neq0.
\end{equation}

We first construct the neighbourhood $U$ such that $\Phi$ is a homeomorphism and outside of zero even a diffeomorphism.
To do this, choose $U$ small enough that the determinant of the differential of $\Phi$ is positive everywhere and the term in parentheses on the right-hand side of Eq.~\eqref{eq:defPhi} is bounded away from zero for all $\frac{z}{|z|}\in S^{2n-1}.$ Note that we need to choose a neighbourhood in the $C^2$ topology instead of the $C^1$ topology for this as the derivative of $g$ incorporates second derivatives of $\phi$.

The restriction $\Phi|_M$ on $M:={\R^{2n}\setminus\{0\}}$ is a proper map: Using e.g. the Bolzano-Weierstrass Theorem, it is straightforward to see that the compact sets of $M$ are the closed sets bounded away from both zero and infinity. This boundedness is preserved by $\Phi$ in both directions by our choice of $U$, so that the preimage of a compact set is compact. This argument also works for the homotopy $(z,t)\mapsto(z+\lift \phi_t(z))/2$ connecting $\Phi|_M$ to the identity. It follows that the mapping degree of $\Phi|_M$ is one (see Chapter III of~\cite{ruiz} and in particular  Proposition~III.2.5).

Since $\Phi_M$ is regular everywhere by our choice of $U$, we can now apply Theorem~III.2.3 of~\cite{ruiz} to conclude that $\Phi$ indeed is a bijection on $M,$ and by $\Phi(0)=0$ also on all of $\R^{2n}.$ It follows by invariance of domain and the inverse function theorem that the graph of $\lift\phi$ projects homeomorphically onto the diagonal, and diffeomorphically outside of zero.\medskip

Applying the identification $\tau:\overline{\R^{2n}}\times\R^{2n}\to T^*\R^{2n}$ from Remark~\ref{rem:tau} yields that $\tau(\Gamma_{\lift\phi})\subseteq T^*\R^{2n}$ projects homeomorphically onto the zero section and thereby is the graph of some (not necessarily smooth at zero) 1-form $\alpha.$\medskip

Restricting $\lift\phi$ and $\alpha$ to $M=\R^{2n}\setminus\{0\},$ we see that $\tau(\Gamma_{\lift\phi|_M})$ is the time-1 image of the zero section of $T^*M$ under the isotopy $\Psi_t:=(\tau\circ(\Id, \lift\phi_t)\circ\tau^{-1})|_{T^*M}$.
This isotopy is Hamiltonian since every contact isotopy is generated by a contact Hamiltonian that lifts to a symplectic Hamiltonian by Proposition \ref{lift-sphere}~(ii). It follows by Corollary~\ref{exactcoroll} that the isotopy consists of exact symplectomorphisms. By Proposition~\ref{exact-lagrangian}, this means that $\tau(\Gamma_{\lift\phi|_M})$ is an exact Lagrangian submanifold. $\alpha|_M$ is then an exact form $d S_1$ for the smooth $S_1:M\to\R$ of Proposition~\ref{simple-gen-funs}~(ii).\medskip

The lift is homogeneous of degree 2 by construction. As $S_1$ arises from the formula for $S_t$ in Proposition~\ref{prop:hamiso}, it is homogeneous of degree 2 as well. This means its first derivatives are homogeneous of degree one and thereby extend continuously from $M$ to $\R^{2n}$. The second derivatives are $\R_{>0}-$invariant, so that this extension is Lipschitz. By continuity of the first derivatives and $\alpha,$ we have $\alpha=d\tilde S_1$ on all of $\R^{2n}$ for this extension $\tilde S_1$ of $S_1.$ \medskip

A (simple) generating function of $\phi$ on the sphere is now given by $F:=\tilde S_1|_{S^{2n-1}}$. We have just seen that $\hat F=\tilde S_1$ is $C^1$ with Lipschitz differential, generates the lift $\lift \phi$ and is smooth outside of zero. The remaining conditions for generating functions on a sphere from Definition~\ref{def-gen-on-sphere} are all satisfied trivially since we have constructed a \textit{simple} generating function.
\end{proof}

We are now ready to show the existence of generating functions of all time-1 maps of contact isotopies on the sphere:

\begin{proposition}\label{prop:existence-generating-sphere}
Let $\phi$ be the time-1 map of a contact isotopy starting at the identity on $(S^{2n-1},\alpha_\text{std})$. Then there exists a generating function $F:S^{2n+k-1}\to \R$ of $\phi.$
\end{proposition}

\begin{proof}
For the contact isotopy $\phi_t$, we can always choose a sufficiently large $N$ and numbers $$0=t_0 < t_1 < ... < t_N = 1$$ such that each $\psi_{j,t}:=\phi_{t_{j-1}+(t_{j}-t_{j-1})t}\circ\phi^{-1}_{t_{j-1}}$ for $j=1...N, t\in[0,1]$ is a contact isotopy starting at the identity that satisfies the $C^2$-smallness assumption of Lemma~\ref{lemma-c2-small-existence}. Thereby each $\psi_{j,1}$ is generated by some simple generating function $F_j$. Inductively applying the composition formula then yields $$F = F_N \# ( ... \# (F_1 \# F_0) ...)$$ as a generating function of
\begin{align*}
\phi_1 &= (\phi_{t_{N}}\circ\phi^{-1}_{t_{N-1}}) \circ ... \circ (\phi_{t_{1}}\circ\phi^{-1}_{t_{0}}) \\
&= \psi_{N,1} \circ ... \circ \psi_{1,1}.
\end{align*}
\end{proof}

We have a correspondence of critical points and discriminant points:

\begin{proposition}\label{prop-discr-sphere}
Let $\phi$ be a contactomorphism on $(S^{2n-1},\alpha_\text{std})$ with a generating function $F:S^{2n+k-1}\to \R$.

For every critical point $(\zeta,\nu)\in S^{2n+k-1}\subseteq\R^{2n+k}$ of $F$ with value 0, $\zeta/|\zeta|\in S^{2n-1}$ is a discriminant point of $\phi.$ This correspondence is one-to-one and the notions of degeneracy match.
\end{proposition}

\begin{proof}

By construction of $\hat F$, each critical point $(\zeta,\nu)\in S^{2n+k-1}$ of $F$ with value zero\footnote{Any other critical points of $F$ fail to be a critical point of $\hat F$ due to homogeneity of degree 2 in the radial direction.} corresponds one-to-one to a radial ray of critical points of $\hat F$ starting at zero and passing $(\zeta,\nu)$. Since $F$ is a generating function on the sphere, $i_{\hat F}$ is injective so that according to Proposition~\ref{prop:crits-are-fixed}~(i), these rays correspond one-to-one to rays of fixed points of $\lift \phi$ starting at 0 and passing $\zeta$. By part (iii) of Proposition~\ref{lift-sphere}, each of these corresponds one-to-one to a discriminant point $\zeta/|\zeta|$ of $\phi.$

By part (iii) of Definition~\ref{def-gen-on-sphere}, we can also apply part (ii) of Proposition~\ref{prop:crits-are-fixed} along the subspace perpendicular to the radial direction. It follows that each of the steps above preserves (non-)degeneracy of critical points, fixed points and discriminant points (outside of the radial direction, where applicable).\medskip

\end{proof}

For the Reeb flow, we will require generating functions with a number of particularly nice properties:

\begin{proposition}\label{gen-reeb-props}
Consider the negative Reeb flow $a_t(z):=e^{-2\pi i t}z$ on $S^{2n-1}$. Then there is an $m\in\N$ and a smooth family of generating functions $(A_t:S^{2n+m-1}\to\R)_{t\in[0,1]}$ of $a_t$ such that:
\begin{enumerate}
\item The extensions $\hat A_t:\R^{2n+m}\to\R$ form a smooth family of quadratic forms.
\item $\partial \hat A_t/\partial t < 0$ on $\Sigma_{\hat A_t}\setminus\{0\}\subseteq \R^{2n+m}$.
\item $\ind(\hat A_1)-\ind(\hat A_0)=2n$ holds\footnotemark. 
\end{enumerate}
\end{proposition}

\footnotetext{The \textbf{index} $\ind(Q)$ of a quadratic form $Q$ is the number of negative eigenvalues of the coefficient matrix.}

\begin{proof}
This proposition and its proof are based on Lemma~4.4 of~\cite{The98}, augmented with an additional computation for (iii).\medskip

We first show that the quadratic form $$Q_t(z):= -\tan(\pi t) ||z||^2$$ on $\C^n\simeq\R^{2n}$ generates $\lift a_t$ for $t\in[0,1/2)$. The graph of $\lift a_t(z)=e^{-2\pi i t}z$ is given by $\{(z,e^{-2\pi i t}z)\;|\; z\in\C^n\}$, which under the identification $\tau$ gets mapped to$$
\tau(\Gamma_{\lift a_t}) = \left\{ \left((1+e^{-2\pi i t})\frac{z}{2}, iz(1-e^{-2\pi i t})\right) \;\huge|\; z\in\C^n\right\}.
$$
Reparametrizing this in terms of the first variable via $z\mapsto 2iz(1-e^{-2\pi i t})^{-1}$ (this is a diffeomorphism since $t<1/2$) yields
$$
\tau(\Gamma_{\lift a_t}) = \{ (z, -2\tan(\pi t)z) \;|\; z\in\C^n\},
$$
i.e. the graph of the 1-form $-2\tan(\pi t)z = d_z Q_t.$ Thus $Q_t$ is even a \textit{simple} generating function of $\lift a_t$ for $t<1/2$. As it also is smooth everywhere, all conditions for the restriction to the sphere being a generating function of $a_t$ are met automatically.\medskip

Regarding construction of $A_t$ and (i): We now set $m=8n$ and
$$
\hat A_t:=Q_{t/3}\#(Q_{t/3}\# Q_{t/3}).
$$
This is a smooth family of quadratic forms because each $Q_t$ is. For the same reason, it satisfies $\hat A_t(\lambda x)=\lambda^2 A_t(x)$ for $\lambda\in\R_{\geq 0}$ and thereby is the extension of the restriction $A_t:=\hat A_t|_{S^{2n+m-1}},$ as our notation suggests. Per Proposition~\ref{comp-form-sphere}, $A_t$ is a generating function of $a_t$.\smallskip

Regarding (ii): By definition of $\hat A_t$ and the composition formula (Proposition~\ref{theret-comp}),
\begin{align}
\hat A_t(q,\zeta_1,\zeta_2,\zeta_a,\zeta_b) &= Q_{t/3}(\zeta_1) + Q_{t/3}(\zeta_a) + Q_{t/3}(\zeta_b)\notag\\
&\quad - 2\omega_\text{std}(\zeta_a-\zeta_2,\zeta_b-\zeta_2) - 2\omega_\text{std}(\zeta_1-q,\zeta_2-q).\label{eq:ausgeschrieben}
\end{align}

The time derivative of $\hat A_t$ fails to be strictly negative exactly at points where $\zeta_1=\zeta_a=\zeta_b=0.$ So we are done if $(q,\zeta_1,\zeta_2,\zeta_a,\zeta_b)\in\Sigma_{\hat A_t}$ and $\zeta_1=\zeta_a=\zeta_b=0$ already imply that $\zeta_2$ and $q$ vanish, too. By Lemma~\ref{euclidean-formulae},
$$\Sigma_{\hat A_t}=\big\{(q,\zeta_1,\zeta_2,\zeta_a,\zeta_b)\in\R^{10n} \;\big|\; \partial \hat A_t/\partial(\zeta_1,\zeta_2,\zeta_a,\zeta_b)=0\big\}.$$
After evaluating the derivative, $\partial \hat A_t/\partial\zeta_a=0$ and $\zeta_a=\zeta_b=0$ indeed yield $\zeta_2=0,$ and $\partial \hat A_t/\partial\zeta_1=0$ and $\zeta_1=\zeta_2=0$ yield $q=0.$
\medskip

Regarding (iii): In a lengthy but mindless computation, we can evaluate Eq.~\eqref{eq:ausgeschrieben} at $t=0,1$ and calculate the eigenvalues of the coefficient matrices to verify that
$$
\ind(\hat A_0) = 5n \qquad\text{and}\qquad \ind(\hat A_1)= 7n.
$$
\end{proof}

%% file: sandon.tex
\chapter{Proof of the Main Result}\label{chap:sandon}

In this chapter, we follow Sandon's proof~\cite{San13} of the special case of the contact Arnol'd conjecture~\ref{carnold} where the underlying manifold is a sphere: \textit{Every generic contactomorphism $\phi$ of $S^{2n-1}$ which is contact isotopic to the identity has at least two translated points.} However, due to a gap in Sandon's argument, we need to replace the condition that $\phi$ is contact isotopic. Instead, we assume that $\phi$ has a generating function $F:S^{2n+k-1}\to\R$ such that the sublevel set $\{F\#0\leq0\}$ is either empty or an embedded submanifold with non-trivial homology.\medskip

In the first section of this chapter, we will introduce cell attachments and a parametric Morse theory that will relate discriminant points of contactomorphisms to the homology of a sublevel set of their generating function.  In the second section, we investigate the homology of the sublevel sets $\{A_t\#F\leq 0\}$, where $A_t$ generates the Reeb flow. We can then combine these findings to prove the main result in the last section. Here, we also discuss the impact of the gap in Sandon's argument.

\section{Cell Attachments and Parametric Morse Theory}\label{sec:morse}

Our goal is to detect the presence of translated points, which the generating function approach relates to critical points. In turn, \textit{Morse theory} provides a connection to topology, allowing us to detect translated points using topological invariants. This essentially builds on the insight that the sublevel sets $f^{-1}((-\infty,a])$ of a function $f:M\to\R$ change by attaching \textit{cells} when passing non-degenerate critical points.\medskip

In this chapter, we will first discuss cell attachments and how we can detect them by their effect on the \textit{Betti numbers} of the space. We then fill out the details of a \textit{parametric} Morse theory that Sandon outlines in~\cite{San13, GKPS}. In this part, we will crucially draw from Milnor's \cite{milnor-morse}.

\begin{definition}[Cell attachment and CW complexes]\ 

\begin{enumerate}
	\item An \textbf{$n$-cell} is the topological space $D^n:=\{x\in\R^n\;|\; ||x||\leq 1\}$.

	\item A topological space arises from a topological space $Y$ by \textbf{attachment of an $n$-cell along an attachment map $\phi:\partial D^n\to Y$} if it is given by the quotient space $$Y\cup_\phi D^n:=(Y\cup D^n)/\sim,$$ where $\sim$ is the equivalence relation that identifies every $x\in \partial D^n$ with $\phi(x)\in Y.$ Note that for $n=0$, this is a disjoint union with a point since $\partial D^n=\emptyset$.

	\item We say a space $X$ is a \textbf{CW complex} if it arises by the following process: We start with a discrete set $X^0,$ construct $X^n$ from every $X^{n-1}$ by attaching an arbitrary number of $n$-cells and finally take the union $X=\bigcup_n X^n$ equipped with the weak topology.
	We regard the decomposition of $X$ into cells as part of the CW structure. For more details, see e.g. Chapter~0 of~\cite{Hat02}.
\end{enumerate}
\end{definition}

It is a well-established fact that every manifold is a CW complex:

\begin{theorem}[Triangulation]\label{triang}
	Every smooth manifold $M$ admits a \textit{triangulation,} i.e. a homeomorphism to a \textit{simplicial complex}. In particular, this induces a CW complex structure on the manifold $M$.
\end{theorem}

\begin{proof}
	See e.g.~\cite{whitehead-triangulation}.
\end{proof}

We will use the following proposition in the proof of the main theorem to detect cell attachments through homology:

\begin{proposition}\label{prop:betti}
Let $X$ be a CW complex and $\phi:\partial D^n\to X$ an attachment map of an $n$-cell into $X$. Then exactly one Betti number of $X$ changes when attaching the cell.
\end{proposition}

\begin{proof}
This argument is adapted from~\cite{RV06}, where it is given in the context of simplicial homology.\medskip

Write $Y:=X\cup_\phi D^n.$ Recall that the $j$-th Betti number is the rank of the $j$-th integral singular homology group. As such, they are invariant under homotopy equivalences. Following the arguments in the second paragraph of the proof of Theorem~3.5 in~\cite{milnor-morse}, we can find a CW complex $Y'$ that is homotopy equivalent to $Y$ and carries the same CW structure as $X$, except for one additional $n$-cell\footnote{It is not true that we can just take the CW structure of $X$ and add in the additional $n$-cell to obtain a CW structure of $Y$: It is not clear that the cell gets attached only to cells of lower dimension, which the definition of a CW structure requires. However, Milnor uses \textit{cellular approximation} of the attachment map to find a homotopic map $\phi'$ that attaches to sufficiently low dimensional cells. His Lemma~3.6, a compatibility result about cell attachments, homotopies and homotopy equivalences, then gives the desired homotopy equivalence between $Y$ and $Y':=X\cup_{\phi'} D^n.$}. We will denote this additional cell by $\sigma.$\medskip

We can now calculate the Betti numbers of the CW complexes $X$ and $Y'$ with \textit{cellular homology}. For a precise definition and a proof that it coincides with singular homology, see e.g. Chapter~2.2 of~\cite{Hat02}. The $j$-th element $C_j$ in the cellular chain complex can be seen as the group of integral linear combinations of $j$-cells, linked together by the natural boundary homomorphism $d_j:C_j\to C_{j-1}.$ By rank-nullity we have 
\begin{equation}\label{eq:sum_ker_im}
\operatorname{rk} C_j = \operatorname{rk} \ker d_j + \operatorname{rk} \operatorname{im} d_j.
\end{equation}
As usual, the homology groups and the Betti numbers can be calculated as
\begin{equation}\label{eq:bettidef}
H_j:=\frac{\ker d_j}{\operatorname{im} d_{j+1}}, \qquad\text{and}\qquad
\beta_j :=\operatorname{rk} H_j = \operatorname{rk} \ker d_j - \operatorname{rk} \operatorname{im} d_{j+1}.
\end{equation}\smallskip

Turning back to our specific cellular complexes $Y'$ and $X,$ it follows by their construction that
\begin{equation}\label{eq:bettichain}
C_j(Y') = \begin{cases}
C_j(X)\oplus \mathbb{Z}[\sigma] & \text{for } j=n,\\
C_j(X) & \text{otherwise.}
\end{cases}
\end{equation}
Since only $C_n(X)$ and $C_n(Y')$ differ, $H_j(X)$ and $H_j(Y')$ as well as the $j$-th Betti numbers must coincide for all $j\not\in\{n-1,n\}.$  We now distinguish two cases based on whether $d_n \sigma\in \operatorname{im} d^{Y'}_n$ also lies in $\operatorname{im} d^{X}_n$:

If $d_n \sigma \not\in \operatorname{im} d^{X}_n$, we intuitively have that the attachment of $\sigma$ closes up an $n$-dimensional 'hole' in $X$. Clearly$$\operatorname{rk} \operatorname{im} d^{Y'}_j = \begin{cases}
\operatorname{rk} \operatorname{im} d^{X}_j+1 & \text{for } j=n,\\
\operatorname{rk} \operatorname{im} d^{X}_j & \text{otherwise,}
\end{cases}
$$
such that$$
\operatorname{rk} \ker d^{Y'}_j \stackrel{\eqref{eq:sum_ker_im},\eqref{eq:bettichain}}{=}\left.
\begin{cases}
C_j(X) +1 - \operatorname{rk} \operatorname{im} d^{X}_j-1 & \text{for } j=n,\\
C_j(X) - \operatorname{rk} \operatorname{im} d^{X}_j & \text{otherwise,}
\end{cases}
\right\}
\stackrel{\eqref{eq:sum_ker_im}}{=}
\operatorname{rk} \ker d^{X}_j.$$
It follows with Eq.~\eqref{eq:bettidef} that$$\beta_j(Y') = \begin{cases}
\beta_j(X)-1 & \text{for } j=n-1,\\
\beta_j(X) & \text{otherwise.}
\end{cases}
$$

If $d_n \sigma \in d_n(C_n(X))$, the attachment of $\sigma$ instead creates a new $(n+1)$-dimensional 'hole'. In this situation, we immediately have $\operatorname{rk} \operatorname{im} d^{Y'}_j =
\operatorname{rk} \operatorname{im} d^{X}_j$
such that$$
\operatorname{rk} \ker d^{Y'}_j \stackrel{\eqref{eq:sum_ker_im},\eqref{eq:bettichain}}{=}\left.
\begin{cases}
C_j(X) +1 - \operatorname{rk} \operatorname{im} d^{X}_j & \text{for } j=n,\\
C_j(X) - \operatorname{rk} \operatorname{im} d^{X}_j & \text{otherwise,}
\end{cases}
\right\}
\stackrel{\eqref{eq:sum_ker_im}}{=}
\begin{cases}
\operatorname{rk} \ker d^{X}_j+1 & \text{for } j=n,\\
\operatorname{rk} \ker d^{X}_j & \text{otherwise.}
\end{cases}
$$

It follows with Eq.~\eqref{eq:bettidef} that$$\beta_j(Y') = \begin{cases}
\beta_j(X)+1 & \text{for } j=n,\\
\beta_j(X) & \text{otherwise.}
\end{cases}
$$

In both cases we see that exactly one Betti number changed.
\end{proof}

We recap some basic definitions of Morse theory:

\begin{definition}[Critical points]
Let $M$ be a manifold and $f:M\to\R$ differentiable.

\begin{enumerate}
	\item A \textbf{critical point} of $f$ is a point $p\in M$ such that $d_pf = 0.$ A \textbf{critical value} is a number $c\in\R$ such that $f^{-1}(c)$  contains at least one critical point.

	\item A \textbf{non-degenerate} critical point of \textbf{index} $\lambda$ is a critical point $p\in M$ such that
	\begin{itemize}
		\item $f$ is smooth around $p$,
		\item there is a coordinate chart around $p$ in which the Hessian matrix $H_pf$ of $f$ is non-singular and
		\item the dimension of the largest subspace on which $H_pf$ is negative definite equals $\lambda.$
	\end{itemize}
	
\end{enumerate}
\end{definition}

Note that the notion of non-degeneracy of critical points and their index is independent of the choice of coordinate system by Sylvester's law of inertia. While it would be sufficient for the results of this section to require $f$ to be only $C^2$ around critical points, we assume smoothness here and in later statements for simplicity.\medskip

The following normal form Lemma is the corner stone of Morse theory:

\begin{lemma}[Morse]\label{morse-lemma}
Let $M$ be a manifold and $f:M\to\R$ differentiable. Around any non-degenerate critical point $p$ of index $\lambda$, there exists a chart $(u_i)_{i=1...n}$ in which the function $f:M\to\R$ can be written as $$f(u)=f(p) -u_1^2-...-u_\lambda^2+u_{\lambda+1}^2+...+u_n^2.$$ In particular, non-degenerate critical points are isolated.
\end{lemma}

\begin{proof}
See e.g. Lemma~2.2 in~\cite{milnor-morse}.
\end{proof}

The fundamental theorems of Morse theory establish that sublevel sets of a function on a compact manifold with only non-degenerate critical points change up to homotopy equivalence only by attaching cells when passing critical values. We will mimic this approach for \textit{parametric} sublevel sets: For some fixed value $a$ and family $f$ of functions, we investigate how the sublevel sets $\{f_t\leq a\}$ change when varying $t$.
\medskip

If $a$ is a regular value of $f_t$ for all times $t$ we pass, nothing happens up to homotopy:

\begin{lemma}[Relation of parametric sublevel sets without passing a critical point]\label{lem-morse-simple}
Let $M$ be a compact manifold and $f:M\times[0,1]\to\R, (x,t)\mapsto f_t(x)$ a $C^1$-map such that each $f_t$ has a Lipschitz differential. Assume that $a\in\R$ is a regular value of $f_t$ for every $t\in[0,1].$ Then there exists an isotopy $\theta_t$ of $M$ such that $\theta_t(\{f_0\leq a\})=\{f_t\leq a\}.$
\end{lemma}

\begin{proof}
We mostly reproduce the proof of Lemma 4.14 from~\cite{GKPS} for the reader's convenience.\medskip

Pick a Riemannian metric on $M$. Note that the open subset $$
\mathcal{U}:=\{(x,t)\;|\; df_t|_x\neq 0\}\subseteq M\times[0,1]
$$
is exactly the set of $(x,t)$ such that the gradient $\nabla f_t|_x$ is non-zero. Define a vector field $$
u_t:= \nabla f_t/||\nabla f_t||^2 \qquad\text{on}\qquad \mathcal{U}_t:=\{x\in M\;|\; df_t|_x\neq 0\}
$$
for each $t\in[0,1]$ and note that $df_t(u_t)=1$ by construction.

Pick $\epsilon>0$ small enough such that the closed neighbourhood$$
\mathcal{W}:=\{(x,t)\in M\times [0,1] \;|\; |f_t(x)-a|\leq\epsilon\}
$$
of $\{(x,t)\in M\times[0,1]\;|\; f_t(x)=a\}$ is contained in $\mathcal{U}.$ Pick a smooth function$$
\rho:M\times [0,1]\to \R
$$
with $\operatorname{supp}\rho\subseteq\mathcal{U}$ and $\rho(x)=1$ for $x\in \mathcal{W}.$ Define a time-dependent vector field $X_{t\in[0,1]}$ by $$
X_t(x)=-\rho(x,t)\;(\partial_t f_t(x))\; u_t(x)\qquad\text{for }(x,t)\in\mathcal{U}
$$
and $X_t(x)=0$ otherwise. It is Lipschitz by assumption so that its flow $\theta_t$ is well-defined on the compact $M$.\medskip

We now claim that $\theta_t$ is the isotopy that we are looking for. We can check$$
\frac{d}{dt}f_t(\theta_t(x)) = (\partial_t f_t)(\theta_t(x)) + df_t(X_t)(\theta_t(x)) = (1-\rho(\theta_t(x),t))\;(\partial_t f_t)(\theta_t(x)).
$$
In particular, $\frac{d}{dt}f_t(\theta_t(x)) =0$ for $(\theta_t(x),t)\in\mathcal{W},$ so $\theta_t(\{f_0=a\})=\{f_t=a\}.$ As $\theta_0=\Id_M,$ $$
\theta_t(\{f_0\leq a\})=\{f_t\leq a\}
$$
follows by continuity of the isotopy.
\end{proof}

If we do pass a time $t$ where $a$ is a critical value, sublevel sets essentially change by cell attachment:

\begin{lemma}[Relation of parametric sublevel sets passing a single critical point]\label{lem-morse-hard}
Let $M$ be a compact manifold and $f:M\times[0,1]\to\R, (x,t)\mapsto f_t(x)$ a $C^1$-map such that each $f_t$ has a Lipschitz differential and $\partial_t f\leq0$. Assume there are $(x_0,t_0)\in M\times(0,1)$ and $\epsilon>0$ such that $x_0$ is a non-degenerate critical point of $f_{t_0}$ with value $a$ and index $\lambda,$ that $\partial_t f<0$ and $f$ smooth in a neighourhood of $(x_0,t_0)$ and that for all $t\in[t_0-\epsilon,t_0+\epsilon]$ there are no other critical points of $f_t$ with value $a$. Then $\{f_{t_0+\epsilon}\leq a\}$ is homotopy equivalent to $\{f_{t_0-\epsilon}\leq a\}$ with a $\lambda-$cell attached.
\end{lemma}

\begin{figure}
\centering
\centerline{\
\includegraphics[width=0.38\linewidth]{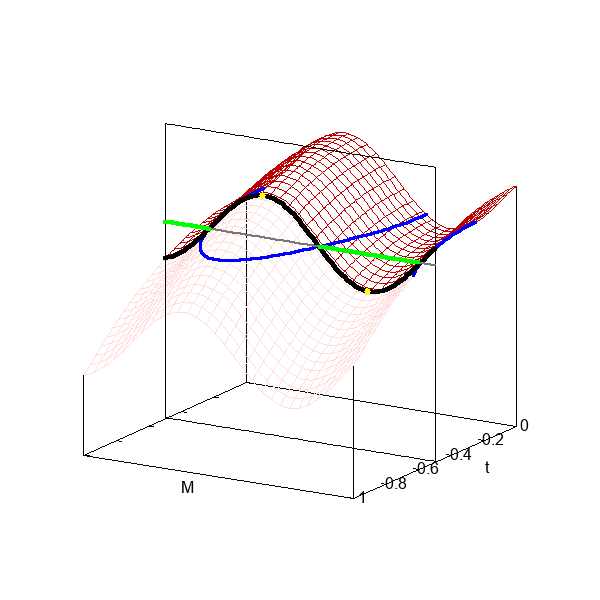}\
\includegraphics[width=0.38\linewidth]{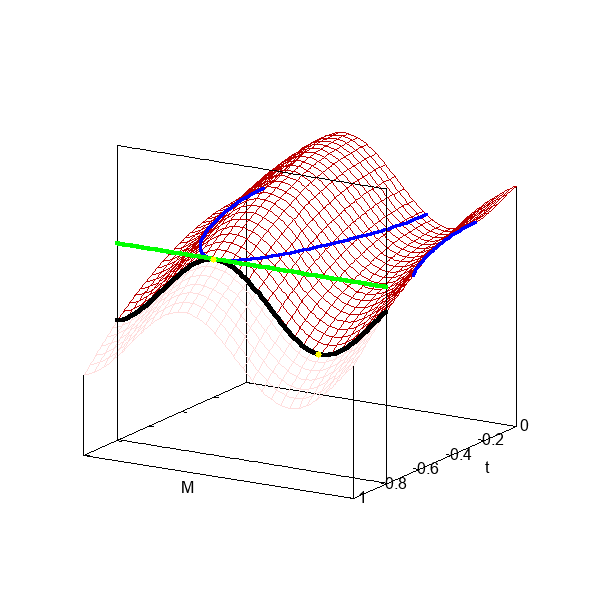}\
\includegraphics[width=0.38\linewidth]{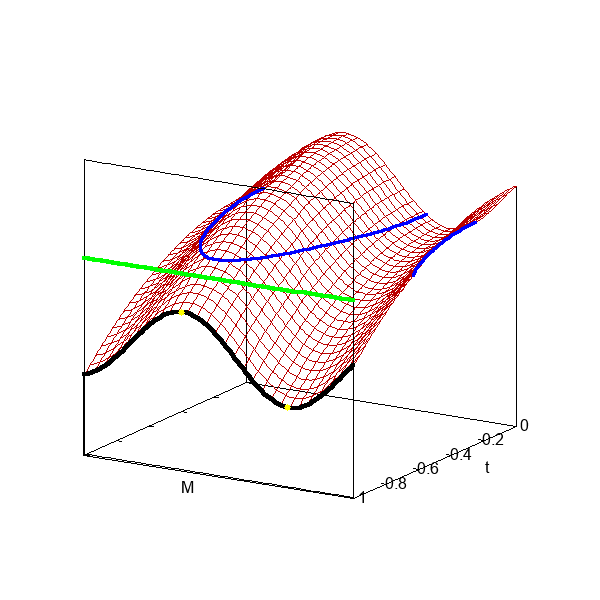}}
\caption{Example of parametric Morse theory. The red mesh is the graph of a family of functions $f$, plotted between time 0 and three different $t$. The thick black lines are the slices $f_t$ at those times $t$. The level set $\{f=a\}$ for a fixed value $a$ is colored blue and the sublevel sets $\{f_t\leq a\}$ are green. As the critical point of $f_t$ passes the value $a$, a 1-cell is attached to the green sublevel set.}
\end{figure}

\begin{proof}
This proof closely mirrors the proof of Theorem~3.2 in~\cite{milnor-morse}. The crucial difference is that we use the implicit function theorem to locally construct a map $h$ whose sublevel sets $\{h\leq t\}$ relate to the parametric sublevel sets $\{f_t\leq a\}$. The proof then proceeds by bringing the problem into a normal form with the Morse Lemma and constructing perturbed functions $\tilde f_t$ and $\tilde h$ around the critical point.
\medskip

Choose a neighbourhood $U$ of $x_0$ and $\delta\in(0,\epsilon)$ small enough such that $\partial_t f$ is bounded away from zero on the closure of $U\times(t_0-\delta,t_0+\delta).$ It follows that $f_{t_0-\delta}(x_0)>a$ and $f_{t_0+\delta}(x_0)<a.$ By shrinking $U$ we may assume that $f_{t_0-\delta}>a$ and $f_{t_0+\delta}<a$ holds on all of $U$. By monotonicity and the intermediate value theorem, we get that for all $x\in U$ there is exactly one $t\in(t_0-\delta,t_0+\delta)$ such that $f_t(x) = a.$ In other words, we can define a unique map$$
h:U\to\R \qquad\text{such that}\qquad f_{h(x)}(x)=a \quad\forall x\in U.
$$

Further shrink $U$ and $\delta$ until $f$ is smooth on the closure of $U\times(t_0-\delta,t_0+\delta).$ Since $\partial_t f$ is bounded away from zero, the implicit function theorem implies that $h$ is smooth with
\begin{equation}\label{eq:diff-relation-h}
D_x h = -(\partial_t f_{h(x)}(x))^{-1} \partial_x f_{h(x)}(x),
\end{equation}
so that critical points of $h$ with value $t$ correspond to critical points of $f_t$ with value $a$. Taking the second derivative at such a critical point shows that
\begin{equation}
D^2_x h = -(\partial_t f_{h(x)}(x))^{-1} \partial^2_x f_{h(x)}(x).
\end{equation}
This means for any choice of metric that the Hessian matrices of $h$ and $f_{t_0}$ in $x_0$ are a positive multiple of each other,
so degeneracy and index are preserved under the correspondence of critical points. In particular, $x_0$ is a non-degenerate critical point of $h$ with index $\lambda.$\medskip

By the Morse Lemma~\ref{morse-lemma}, we can further restrict $U$ to find coordinates $u=(u_1, ..., u_n)$ on it such that $$
h = t_0 -u_1^2 - ... - u_\lambda^2 + u_{\lambda+1}^2 + ... + u_n^2
$$
and the critical point has coordinates $u(x_0)=0.$ For later convenience, define
\begin{align*}
\xi(x)  &= u_1^2 - ... + u_\lambda^2\\
\eta(x) &= u_{\lambda+1}^2 + ... + u_n^2
\end{align*}
on $U$ such that $h = t_0-\xi + \eta.$\medskip

Further shrink $\delta$ until the closed ball with radius $\sqrt{2\delta}$ around the origin is contained in the image of the chart $u:U\to\R^n$ and define the $\lambda$-cell
\begin{align*}
D:&=\{p\in U \;|\; u_1^2(p)+...+u_\lambda^2(p)\leq \delta\text{ and } u_{\lambda+1}(p)+...+u_n(p)=0\}\\
&=\{p\in U \;|\; \xi(p)\leq \delta\text{ and } \eta(p)=0\}.
\end{align*}

We find ourselves in the situation sketched by the following figure if we collapse each of $(u_1,...,u_\lambda)$ and $(u_{\lambda+1},...,u_n)$ into a line:

\centerline{\includegraphics[width=0.5\linewidth]{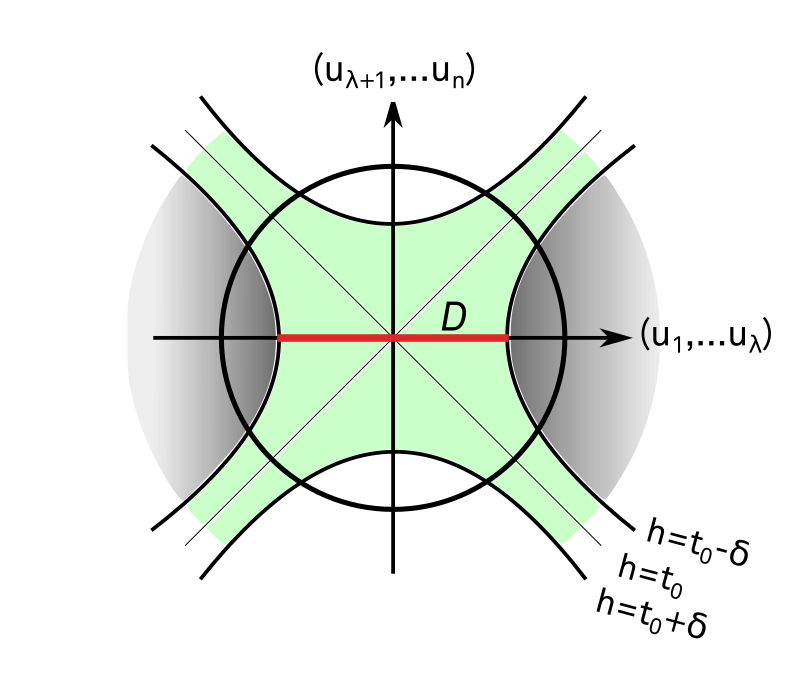}}
The circle represents the boundary of the ball with radius $\sqrt{2\delta}.$ The heavily shaded region is the region where $h\leq t_0-\delta,$ or equivalently $f_{t_0-\delta}\leq a.$ In the green region we have $t_0-\delta< h<t_0+\delta.$ The red line is the $\lambda$-cell $D$ and the level sets of $h$ are labeled.\medskip

Pick a smooth function $\mu:\R\to\R$ such that
\begin{align*}
\mu(0)&>\delta,\\
\mu(r) &= 0\qquad\qquad \forall r\geq 2\delta,\\
\mu'(r)&\in(-1,0]\qquad \forall r\in\R.
\end{align*}
Now define the map $\tilde f:M\times[0,1]\to\R$ to coincide with $f$ outside of $U\times [0,1],$ and set $$
\tilde f_t(x) = f_{t+\mu(\xi(x)+2\eta(x))}(x)
$$ for $(x,t)\in U\times [0,1].$

Just like $f$ induced $h$, the map $\tilde f$ now corresponds to a smooth $\tilde h:U\to\R$ given by $$
\tilde h := t_0-\xi + \eta -\mu(\xi+2\eta)
$$
such that $\tilde f_{\tilde h(x)}(x)=a \quad\forall x\in U.$\medskip

By definition of $\mu,$ the functions $f_{t+\delta}$ and $\tilde f_{t+\delta}$ coincide outside of the ellipsoid $\{\xi + 2\eta\leq 2\delta\}\subseteq U$. For any $x\in U$ contained in this ellipsoid,$$
\tilde h(x)\leq h(x) = t_0 -\xi(x)+\eta(x)\leq t_0 +\frac12(\xi(x)+2\eta(x)) = t_0 + \delta,
$$
so that $\tilde f_{t_0+\delta}(x)\leq a$ and $f_{t_0+\delta}(x)\leq a$ both hold. It follows that
\begin{equation}\label{eq:htpy-above}
\{f_{t_0+\delta}\leq a\} = \{\tilde f_{t_0+\delta}\leq a\}.
\end{equation}

We have constructed $\mu$ in a way such that $h'$ has the same unique critical point within $U:$ We have$$
\frac{\partial \tilde h}{\partial\xi} = -1-\mu'(\xi+2\eta)<0
\quad\text{ and }\qquad
\frac{\partial \tilde h}{\partial\eta} = 1-2\mu'(\xi+2\eta)\geq1,
$$
so that$$
d\tilde h = \frac{\partial \tilde h}{\partial\xi} d\xi + \frac{\partial \tilde h}{\partial\eta} d\eta
$$
vanishes only where $d\xi$ and $d\eta$ do, i.e. at $x_0=u(0).$ Evaluating $\tilde h$ at the critical point gives
$\tilde h(x_0) = t_0 - \mu(0)<t_0-\delta.$ Due to the correspondence of critical points of $\tilde h$ with those of $\tilde f_{\tilde h},$ we can apply Lemma~\ref{lem-morse-simple} to conclude the homotopy equivalence
\begin{equation}\label{eq:tilde-along-time}
\{\tilde f_{t_0+\delta}\leq a\} \simeq \{\tilde f_{t_0-\delta}\leq a\}.
\end{equation}

Define $H$ to be the closure of $\{\tilde f_{t_0-\delta}\leq a\} \setminus \{f_{t_0-\delta}\leq a\}$ such that
\begin{equation}\label{eq:handle-topy}
\{\tilde f_{t_0-\delta}\leq a\} = \{f_{t_0-\delta}\leq a\} \cup H.
\end{equation}
The $\lambda$-cell $D$ is contained in $H$: For $x\in D,$ $\partial \tilde h/\partial \xi<0$ implies $$
\tilde h(x)\leq \tilde h(x_0) < t_0-\delta \qquad\Rightarrow\qquad
\tilde f_{t_0-\delta}(x)< a,
$$
while also $f_{t_0-\delta}(x)\geq a.$

The set $H$ is now given by the region covered with arrows in the following sketch, while the region with $t_0-\delta<\tilde h< t_0+\delta$ is colored green:

\centerline{\includegraphics[width=0.5\linewidth]{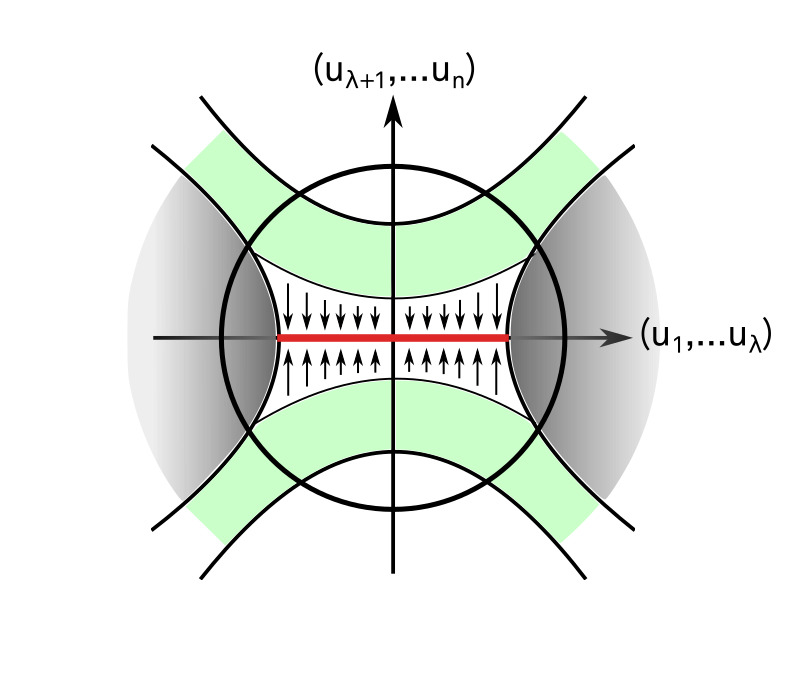}}

Note that the space $\{f_{t_0-\delta}\leq a\} \cup D$ is a cell attachment of $D$ onto $\{f_{t_0-\delta}\leq a\}.$ As the figure suggests, we can find a homotopy equivalence that is trivial outside of $U$ such that
\begin{equation}\label{eq:attach-topy}
\{f_{t_0-\delta}\leq a\} \cup H \simeq \{f_{t_0-\delta}\leq a\} \cup D.
\end{equation}
An explicit construction of such a homotopy equivalence is given in Assertion~4 in the proof of Theorem~3.2 of~\cite{milnor-morse}.\medskip

We are now ready to put everything together:
\begin{align*}
\{f_{t_0+\epsilon}\leq a\}
&\simeq
\{f_{t_0+\delta}\leq a\}
\\ &\stackrel{\eqref{eq:htpy-above}}{=}
\{\tilde f_{t_0+\delta}\leq a\}
\\ &\stackrel{\eqref{eq:tilde-along-time}}{\simeq}
\{\tilde f_{t_0-\delta}\leq a\}
\\ &\stackrel{\eqref{eq:handle-topy}}{=}
\{f_{t_0-\delta}\leq a\} \cup H
\\ &\stackrel{\eqref{eq:attach-topy}}{\simeq}
\{f_{t_0-\delta}\leq a\} \cup D
\\ &\simeq
\{f_{t_0-\epsilon}\leq a\} \cup_\psi D
\end{align*}
Apart from the previously established equivalences, we have used Lemma~\ref{lem-morse-simple} in the first and last step to cover the remaining distance from $t_0-\epsilon$ to $t_0-\delta$ and $t_0+\delta$ to $t_0+\epsilon.$ In the last step, we additionally used that a homotopy equivalence between topological spaces induces a homotopy equivalence between the spaces with an attached cell for a suitable attachment map $\psi$ (see Lemma~3.7 of~\cite{milnor-morse} for a proof due to P.~Hilton). This completes our proof.
\end{proof}

\begin{remark}[Generalization of Lemma~\ref{lem-morse-hard} to multiple critical points]\label{rem-gen-morse-hard}
Note that the proof of Lemma~\ref{lem-morse-hard} can easily be modified to allow for multiple non-degenerate critical points with value $a$: Since the Morse Lemma guarantees that these lie isolated, we can construct a perturbed $\tilde f$ and $\tilde h$ with minimal modifications. This results in a cell attachment for each of the critical points.
\end{remark}

Putting the last two lemmata together, we get:

\begin{proposition}[Parametric Morse Theorem]\label{morse-main-stuff}
Let $M$ be a compact manifold and $f:M\times[0,1]\to\R, (x,t)\mapsto f_t(x)$ a $C^1$-map such that each $f_t$ has a Lipschitz differential and $\partial_t f\leq0$.

Let $a\in\R$ and define
$$\operatorname{Crit}_{f,a} := \{(x,t)\in M\times [0,1] \;|\; x \text{ is a critical point of $f_t$ with value $a$}\}.$$ Assume that for each $(x,t)\in \operatorname{Crit}_{f,a}$, $x$ is a non-degenerate critical point of $f_t$ with index $\lambda_{x,t}$ and  that around $(x,t),$ $f$ is smooth with $\partial_t f<0$.

Then $\{f_{1}\leq a\}$ is homotopy equivalent to $\{f_{0}\leq a\}$ with a cell of dimension $\lambda_{x,t}$ attached for each $(x,t)\in\operatorname{Crit}_{f,a}$.
\end{proposition}

\begin{proof}
By the Morse Lemma and the fact that $\partial_t f_t<0$ around each $(x,t)\in\operatorname{Crit}_{f,a}$, we have that the points in $\operatorname{Crit}_{f,a}$ are isolated. It follows that there is a finite number $N$ of times $t_1<...<t_N$ such that $a$ is not a regular value of $f_{t_j}.$ Pick $\epsilon>0$ smaller than any $|t_j-t_k|$ for $j\neq k.$ Lemma~\ref{lem-morse-hard} and Remark~\ref{rem-gen-morse-hard} then give homotopy equivalences between each $\{f_{t_j+\epsilon}\leq a\}$ and $\{f_{t_j-\epsilon}\leq a\}$ with cells attached. Lemma~\ref{lem-morse-simple} provides homotopy equivalences between each $\{f_{t_j+\epsilon}\leq a\}$ and $\{f_{t_{j+1}-\epsilon}\leq a\}$. These also induce homotopy equivalences between the spaces with attached cells by Lemma~3.7 of~\cite{milnor-morse}.
\end{proof}

\section{Sublevel Sets of Composed Generating Functions}

This section is largely based on \cite{San13}, but fixes various omissions and inaccuracies. We will provide a number of lemmata concerning sublevel sets of generating functions of contactomorphisms on the sphere. All of these culminate in Proposition~\ref{homology-sublevelsets}, which will essentially allow us to detect differences in the reduced homology groups $$\tilde H_k(\{A_t\#F\leq0\})$$ of sublevel sets for $t=0$ and $t=1.$
This will be based on an argument that, for sufficiently nice $G$, the sublevel sets of composed generating functions at least philosophically are given by the \textit{join}
$$
\{G\#F\leq 0\}\simeq \{G\leq 0\} * \{F\#0\leq 0\}
$$
of their individual sublevel sets.\medskip

The join can be viewed as the union of all line segments connecting $X$ and $Y$ when these are placed in general position relative to each other:

\begin{definition}
Let $X$ and $Y$ be topological spaces. We define the \textbf{join} $X*Y$ as the quotient space obtained from $X\cup  X\times Y\times[0,1] \cup Y$ by identifying for all $x\in X, y\in Y$ the points $(x,y,0)$ with $x$ as well as the points $(x,y,1)$ with $y.$
\end{definition}

\begin{remark}[Alternative definition of the join]\label{join-remark}
We follow the definition of the join from \cite{Whi56}. Note that it is frequently defined as an analogous quotient space of only $X\times Y\times[0,1]$ instead. These definitions only differ if exactly one of $X$ and $Y$ is the empty set: For us, $M*\emptyset = M$ holds, which would otherwise be the empty set as well.

Sandon does not use our definition in Section 4 of~\cite{San13}. Strictly speaking, this leads to an error since she does not specifically exclude the case that $F$ is positive in the last part of her argument.
\end{remark}

As a step towards sublevel sets of composed generating functions, we can now consider those of the direct sum:

\begin{lemma}[Sublevel set of direct sum]\label{lemma-oplus-join}
Let $F_1:S^{n_1-1}\to\R$ and $F_2:S^{n_2-1}\to\R$ be smooth functions on the sphere such that $\{F_j\leq0\}$ are deformation retractions of neighbourhoods in $S^{n_j-1}$. Write $\hat F_1:\R^{n_1}\to\R, \hat F_2:\R^{n_2}\to\R$ for the extensions defined as in Definition~\ref{def-gen-on-sphere} and $F_1\oplus F_2$ for the restriction of the direct sum $\hat F_1\oplus \hat F_2$ to the unit sphere $S^{n_1+n_2-1}$. Then there is a homotopy equivalence $$\{F_1\leq 0\} * \{F_2\leq 0\}  \simeq \{F_1 \oplus F_2\leq 0\} .$$
\end{lemma}

\begin{proof}
We follow, in part, the proofs of Proposition~3.12 in~\cite{GKPS} and Proposition B.1 in~\cite{Giv90} and incorporate part of an argument provided by A. Givental in correspondence regarding the latter.\medskip

Note first that there is an inclusion $$
\iota:\{F_1\leq 0\} * \{F_2\leq 0\}  \hookrightarrow \{F_1 \oplus F_2\leq 0\} \subseteq S^{n_1+n_2-1}
$$
that can be defined on the join by sending equivalence classes of $x_1\in\{F_1\leq 0\}$ to $(x_1,0)$, $x_2\in\{F_2\leq 0\}$ to $(0,x_2)$, and $(x_1,x_2,s)\in \{F_1\leq 0\} \times \{F_2\leq 0\}\times[0,1]$ to $(\sqrt{s}x_1,\sqrt{1-s}x_2)$. For brevity, we will write
\begin{align*}
X &:= \{F_1 \oplus F_2\leq 0\} \subseteq S^{n_1+n_2-1},\\
A &:= \iota\left(\{F_1\leq 0\} * \{F_2\leq 0\} \right) \\
&\ = \{(x_1,x_2)\in X\;|\; F_1(x_1)\leq 0 \text{ and } F_2(x_2)\leq 0 \} \subseteq X.
\end{align*}
Our strategy will be to find a deformation retraction from $X$ to $A$.\medskip

First combine the deformation retractions of neighbourhoods $U_j\subseteq S^{n_j-1}$ to $\{F_j\leq 0\}$ to obtain a deformation retraction $$r:B\times [0,1]\to B$$ from a neighbourhood $B$ of $A$ in $X.$\medskip

We now want to find a homotopy $$r':X\times [0,1]\to X,$$ such that $r'_0=\Id_X$, $r'_t|_B = \Id_B$ and $r'_1(X)\subseteq B.$ To do this, we need to continuously interpolate between the identity on $B$ and homotopies that move points outside $B$ into it. For a $\delta>0$ that we will fix later, define
$$\hspace{-1.5cm}
r'_t(\sqrt{s}\,x_1,\sqrt{1-s}\,x_2) = 
\begin{cases}
\left(  \sqrt{(1-t)s}\,x_1,\;\sqrt{(1-t)(1-s)+t}\,x_2  \right)
&\text{for } F_1(x_1)\geq \delta, \\

\left(   \sqrt{\left(1-t q_j(x_j)\right)s}\,x_1,\;\sqrt{\left(1-t q_j(x_j)\right)(1-s)+t q_j(x_j)}\,x_2   \right)
&\text{for } 0\leq F_1(x_1)\leq \delta, \\

\left( \sqrt{s}\,x_1,\sqrt{1-s}\,x_2 \right)
&\text{for } F_1(x_1) \leq0 \text{ and } F_2(x_2)\leq0, \\

\left(   \sqrt{\left(1-t q_j(x_j)\right)s+t q_j(x_j)}\,x_1,\;\sqrt{\left(1-t q_j(x_j)\right)(1-s)}\,x_2    \right)
&\text{for } 0\leq F_2(x_2)\leq \delta, \\

\left(  \sqrt{(1-t)s+t}\,x_1,\;\sqrt{(1-t)(1-s)}\,x_2   \right)
&\text{for } F_2(x_2)\geq \delta, \\
\end{cases}
$$
where $q_j(x_j)= F_j(x_j)/\delta$, $s\in[0,1], x_1\in S^{n_1-1}$ and $x_2\in S^{n_2-1}$. We can check that the wanted properties hold if we choose any $\delta$ small enough so that the domains of the second and fourth case of the definition are contained within the neighbourhood $B$ of $A$. \medskip

We can combine $r$ and $r'$ to define the homotopy $$R_t:=
\begin{cases}
r'_{2t} &\quad \text{for } t\in[0,1/2], \\
r_{2t-1}\circ r'_{1} &\quad \text{for } t\in(1/2,1]
\end{cases}
$$
on $X.$ By the properties of $r$ and $r',$ this is continuous, $R_0=\Id_X$, $R_1(X)=A$ and $R_1|_A=\Id_A$. In other words, this is a deformation retraction of $X$ onto $A$.
\end{proof}

We will be considering sublevel sets of compositions with quadratic generating forms in particular, which the next three lemmata will allow us to simplify:

\begin{lemma}\label{lemma-subsets-quad}
Consider a quadratic form $A:\R^{n}\to\R$ and the sublevel set $$X:=\{x\in S^{n-1}\;|\; A(x) \leq 0\}$$ of its restriction to the unit sphere.

\begin{enumerate}
\item The space $S^{\ind(A)-1}$ is a strong deformation retract of $X$.

\item $X$ is a deformation retract of a neighbourhood of itself in $S^{n-1}.$ 
\end{enumerate}
\end{lemma}

\begin{proof}
Set $\lambda := \ind(A)$ for the (negative) index, i.e. the number of negative eigenvalues of the coefficient matrix of $A$. By Sylvester's law of inertia for quadratic forms, we can assume without loss of generality that there exists a $k\leq n$ such that $A$ takes the form $$A(x) = -x_1^2 - ... - x_\lambda^2 + x_{\lambda+1}^2+...+x_k^2.$$

Regarding (i): For convenience, we write $x_-:=(x_1,...,x_\lambda)$, $x_+:=(x_{\lambda+1}, ..., x_k)$ and $x_0:=(x_{k+1}, ..., x_n)$ such that $A(x)=-x_-^2+x_+^2.$ It follows that we can write the sublevel set as $$X=\{x\in S^{n-1}\;|\; x_+^2\leq x_-^2\}.$$ Consider the subset $$S:=\{x\in S^{n-1}\;|\; x_+=0=x_0,  \text{ and } |x_-|=1\} \simeq S^{\lambda-1}.$$
We define a strong deformation retract $F:X\times[0,1]\to X$ by $$(x_-,x_+,x_0,t)\mapsto \left( t\frac{x_-}{|x_-|}+(1-t)x_-, (1-t)x_+, (1-t)x_0 \right).$$
This map is indeed well-defined and continuous, equals the identity for $t=0$ and on $S$ for all $t$, and $F(x,1)\in S$ for all $x\in X.$\medskip

Regarding (ii): Consider the restriction $B$ of $A$ to its non-null directions. Zero is a regular value of this restriction so that $\{B\leq0\}$ is a submanifold and thereby a deformation retract of a neighbourhood. Extending the deformation retraction by acting trivially on the null directions then also gives that $\{A\leq0\}$ is a deformation retract of a neighbourhood.

\end{proof}

\begin{lemma}\label{prop:fiber-only-form}
Let $Q:\R^{2n}\times\R^k\to\R$ be a quadratic form that generates the identity on $\R^{2n}$.
Then there is an isotopy $(\Psi_s)_{s\in[0,1]}$ of fibre preserving\footnotemark\ linear diffeomorphisms of $\R^{2n}\times\R^k$ such that $\Psi_0 = \Id_{\R^{2n}\times\R^k}$ and $Q\circ\Psi_1$ is a quadratic form that only depends on the fiber variable.
\end{lemma}

\footnotetext{I.e., $\Psi_s(\zeta,\nu)\in\{\zeta\}\times\R^k$ for all $(\zeta,\nu)\in\R^{2n}\times\R^k.$
}

\begin{proof}
We give additional details of the proof of Lemma 4.10 from~\cite{GKPS}. Write the quadratic form at $(\zeta,\nu)\in\R^{2n}\times\R^k$ as
$$
Q(\zeta,\nu) = \frac12 \begin{pmatrix}\zeta^T & \nu^T\end{pmatrix} \begin{pmatrix} A & B \\ B^T & C \end{pmatrix} \begin{pmatrix}\zeta \\ \nu \end{pmatrix},
$$
where $A$ and $C$ are symmetric matrices. By Lemma~\ref{euclidean-formulae}, $\Sigma_Q$ consists of those $(\zeta,\nu)$ such that $$\frac{\partial Q}{\partial\nu} = C\nu+B^T\zeta \stackrel{!}{=}0,$$
and we can write
\begin{equation}\label{eq:papperlapapp}
i_Q(\zeta,\nu)=(\zeta,A\zeta+B\nu).
\end{equation}
Since $Q$ generates the identity, the image $i_Q(\Sigma_Q)$ must equal the zero section of $T^*\R^{2n},$ so by the first component of Eq.~\eqref{eq:papperlapapp}, for every $\zeta$ there must be at least one $\nu$ such that $(\zeta,\nu)\in\Sigma_Q,$ i.e. $C\nu+B^T\zeta = 0.$ Making such a choice for every vector of a basis of $\R^{2n}$ thus gives us a matrix $D$ such that
\begin{equation}\label{mat-cond-1}
CD+B^T=0.
\end{equation}
Again since $i_Q(\Sigma_Q)$ is the zero section, we must have $$i_Q(\zeta,D\zeta)=(\zeta,0),$$
which by the second component of Eq.~\eqref{eq:papperlapapp} implies
\begin{equation}\label{mat-cond-2}
A+BD=0.
\end{equation}

Now define the smooth isotopy $(\Psi_s)_{s\in[0,1]}$ of fibre preserving linear diffeomorphisms via $$\Psi_s(\zeta,\nu) = (\zeta, \nu+sD\zeta).$$ The coefficient matrix of $Q\circ\Psi_1$ is given by
$$
\begin{pmatrix} 1 & D^T \\ 0 & 1 \end{pmatrix}
\begin{pmatrix} A & B \\ B^T & C \end{pmatrix}
\begin{pmatrix} 1 & 0 \\ D & 1 \end{pmatrix}
=
\begin{pmatrix} A+BD+CD+B^T & CD+B^T \\ B+D^TC & C \end{pmatrix},
$$
which evaluates to $(\begin{smallmatrix}0&0\\ 0&C\end{smallmatrix})$ using Eqs.~\eqref{mat-cond-1} and~\eqref{mat-cond-2}. This means that $Q\circ\Psi_1$ only depends on the fiber variables.
\end{proof}

\begin{lemma}\label{diffeo-sublevel-composition}
Let $F,F':S^{2n+N-1}\to\R$ and $G:S^{2n+M-1}$ be three generating functions on the sphere and assume there is a diffeomorphism $\Psi:\R^{2n}\times\R^N\to\R^{2n}\times\R^N$ that preserves the fibers over $\R^{2n}$, is homogeneous of degree one and satisfies $\hat F'=\hat F\Psi$. Then there exists a canonical diffeomorphism $$\tilde \Psi:\{F\#G\leq0\}\to\{F'\# G\leq 0\}.$$
\end{lemma}

\begin{proof}
Since $\Psi$ preserves fibers, we can write $\Psi(\zeta_1,\nu_1)=(\zeta_1,\Psi_2(\zeta_1,\nu_1))$ for the component $\Psi_2$ of $\Psi$ on $\R^N.$ Define $\Phi:\R^{2n+2n+2n+N+M}\to\R^{2n+2n+2n+N+M}$ by setting $$
\Phi(q,\zeta_1,\zeta_2,\nu_1,\nu_2):= (q,\zeta_1,\zeta_2, \Psi_2(\zeta_1,\nu_1),\nu_2).
$$

This is a diffeomorphism such that $\hat F'\#\hat G=(\hat F\# \hat G) \Phi.$ It follows that $\Phi$ restricts to a diffeomorphism of the sublevel sets of the functions $\hat F\#\hat G$ and $\hat F'\#\hat G$ on $\R^{6n+N+M}$. Like $\Psi,$ $\Phi$ is homogeneous of degree one.\medskip

Both $\hat F\#\hat G$ and $\hat F'\#\hat G$ are homogeneous of degree 2, so every open ray starting at zero lies either entirely within their respective 0-sublevel sets or outside of it. Such rays can be identified with points on the unit sphere and are preserved under $\Phi$ since it is homogeneous of degree one. Under this identification, $\Phi$ gives a diffeomorphism $\tilde\Psi$ of sublevel sets on the sphere.
\end{proof}

Write $\tilde H_j(X)$ for the $j$-th reduced integer homology group of a compact space $X$. We have a Kuenneth formula for the join of a CW complex with a sphere:

\begin{lemma}\label{kuenneth-lemma}
Let $X\neq \emptyset$ be a CW complex and $d\in\mathbb{N}.$ It follows that
$$
\tilde H_k(S^d*X)=\tilde H_{k-d-1}(X),
$$
where we use the convention that all negative-dimensional reduced homology groups are trivial.
\end{lemma}

\begin{proof}
We follow an argument from Section~4 of~\cite{San13}. The \textit{smash product} $A\wedge B$ of two pointed topological spaces $(A,a_0)$ and $(B,b_0)$ arises from $A\times B$ by identifying all $(a,b_0)$ and $(a_0,b)$ for all $a\in A, b\in B.$
The \textit{reduced suspension} $\Sigma Y$ of a pointed topological space $(Y,y_0)$ is the space $$
\Sigma Y := \frac{Y\times[0,1]}{ Y\times\{0\}\cup Y\times\{1\}\cup \{y_0\}\times[0,1]  }.
$$
It is generally true that the join of two non-empty spaces is homotopy equivalent to the suspension of their smash product for any choice of basepoints, see e.g. 0.24 in \cite{Hat02}. Moreover, it is straightforward to see that the reduced suspension $\Sigma Y$ of a pointed topological space $Y$ is the same as $S^1\wedge Y.$\medskip

We have so far established that $$
S^d * X \simeq S^1\wedge (S^d \wedge X).$$

By the Kuenneth formula, a smash product of a CW complex with a sphere shifts the homology groups of the resulting space by the dimension of the sphere (see e.g. page 276 of \cite{Hat02}). The statement of the Lemma follows by invariance of homology under homotopy equivalences.
\end{proof}

We are now ready to combine the previous lemmata into the main result of this section.

\begin{proposition}[Homology of $\{A_t\#F\leq0\}$]\label{homology-sublevelsets}
Let $F:S^{2n+k-1}\to\R$ be a generating function of a contactomorphism $\phi$ on $S^{2n+k-1}$ and $A_t$ the family of generating functions of the Reeb flow from Proposition~\ref{gen-reeb-props}. Assume that $\{F\#0\leq 0\}$ is an embedded submanifold of $S^{6n+k-1}$ and let $t\in\{0,1\}$.

\begin{enumerate}
\item If $\{F\#0\leq 0\}$ is empty, then
$$
\tilde H_j\big(\{A_t\#F\leq 0\}\big) = \begin{cases} \mathbb{Z} & \text{if }j=\ind(\hat A_t), \\
0 &\text{else.}
\end{cases}
$$
\item If $\{F\#0\leq 0\}$ is non-empty, then
$$
\tilde H_j\big(\{A_t\#F\leq 0\}\big) = \tilde H_{j-\ind(\hat A_t)}\big(\{F\#0\leq 0\}\big),
$$
where we use the convention that all negative-dimensional reduced homology groups are trivial.
\item If $\{F\#0\leq 0\}$ is either empty or has non-trivial homology, at least two Betti numbers of $\{A_0\#F\leq 0\}$ and $\{A_1\#F\leq 0\}$ differ. Otherwise, they both have trivial homology.
\end{enumerate}
\end{proposition}

\begin{proof}\ 

Regarding (i) and (ii): By Lemma~\ref{prop:fiber-only-form}, for each $t\in\{1,2\}$ there exists a quadratic form $\hat{A}'_t:\R^{2n+m} \to\R$ that only depends on the fiber variables and equals the quadratic forms $\hat A_t:\R^{2n+m}\to\R$ up to a fiber preserving linear diffeomorphism $\Psi_t$ on the domain, i.e. $\hat A'_t=\hat A_t\Psi_t$. Write $A'_t:=\hat A'_t|_{S^{2n+m-1}}$ for the restrictions to the sphere, in line with our notation $\hat \cdot$ for extensions of maps on the sphere.\medskip

We will soon show for $t\in\{1,2\}$ that the following series of homotopy equivalences holds:
\begin{align*}
\{A_t\#F\leq 0\} &\simeq \{ A'_t \# F \leq 0 \} \\
&\simeq \{ A'_t \oplus (F\#0) \leq 0 \}  \\
&\simeq \{A'_t \leq0\} * \{ F\#0 \leq0\}  \\
&\simeq S^{\ind(\hat A_t)-1} * \{ F\#0 \leq0\}
\end{align*}

Assume for now that the equivalences hold. If $\{ F\#0 \leq0\}$ is empty, the join operation yields a sphere again and (i) follows. If $\{ F\#0 \leq0\}$ is non-empty, (ii) follows immediately with Lemma~\ref{kuenneth-lemma}. We can apply this Lemma since every manifold is a CW complex by Theorem~\ref{triang}.\medskip

Regarding the first homotopy equivalence: By Proposition~\ref{diffeo-sublevel-composition}, the linear and fiber preserving diffeomorphisms $\Psi_t$ even induce a canonical diffeomorphism between the sublevel sets.

Regarding the second homotopy equivalence: This follows immediately from $\hat A'_t\# \hat F = \hat A'_t\oplus(\hat F\#0 ),$ which holds by the structure of the composition formula and the fact that $\hat A'_t$ only depends on fiber variables.

Regarding the third homotopy equivalence: This follows from Lemma~\ref{lemma-oplus-join} if we can show that $\{A'_t\leq0\}$ and $\{F\#0\leq0\}$ are deformation retracts of neighbourhoods in $S^{2n+m-1}$ and $S^{6n+k-1}$, respectively. For the former, we have seen that this holds in Lemma~\ref{lemma-subsets-quad}~(i).
The latter is an embedded submanifold by assumption and thereby has a tubular neighbourhood that deformation retracts onto it.

The fourth homotopy equivalence follows from Lemma~\ref{lemma-subsets-quad}~(ii) since the indices of $\hat A'_t$ and $\hat A_t$ are equal.\medskip

Regarding (iii): Note that by Proposition~\ref{gen-reeb-props}~(iii), we have $$
\ind(\hat A_1) - \ind(\hat A_0) = 2n > 0.
$$

If $\{ F\#0 \leq0\}$ is empty, then part (i) implies that the Betti numbers differ in the dimensions $\ind(\hat A_1)$ and $\ind(\hat A_0)$.

If $\{ F\#0 \leq0\}$ is non-empty with non-trivial homology, then parts (ii) yields $$
\tilde H_{k+\ind(\hat A_0)}(\{A_0\#F\leq 0\}) = \tilde H_k(\{ F\#0 \leq0\}) = \tilde H_{k+\ind(\hat A_1)}(\{A_1\#F\leq 0\}).
$$
In particular, the homology groups of $\{A_0\#F\leq 0\}$ and $\{A_1\#F\leq 0\}$ are not all trivial and shifted by a positive number. Since this means that the minimal and maximal non-trivial groups are shifted, it implies that two Betti numbers differ.

If $\{ F\#0 \leq0\}$ is non-empty with trivial homology, then applying Lemma~\ref{kuenneth-lemma} to the series of homotopy equivalences from the proof of part (i) and (ii) implies that $\ind(\hat A_t)$ has trivial homology for $t=0,1$ as well.

\end{proof}

One way to satisfy the requirement that $\{F\#0\leq 0\}$ is an embedded submanifold is assuming that 0 is a regular value of $F$:

\begin{lemma}\label{deform-f-zaun}
Let $F:S^{2n+k-1}\to\R$ be a generating function on the sphere such that 0 is a regular value of $F$. Then 0 is also a regular value of $F\#0:S^{6n+k-1}\to\R,$ the composition of $F$ with the zero function on $S^{2n-1}.$
\end{lemma}

\begin{proof}
Note first that 0 is a regular value of a function on the sphere exactly if it is a regular value of its homogeneous extension as defined in Definition~\ref{def-gen-on-sphere}. It follows that we need to show that 0 is a regular value of $\hat F\#0$ assuming it is one of $\hat F$.\medskip

Per Definitions~\ref{theret-comp} and~\ref{comp-form-sphere}, $F\#0$ is the restriction to the sphere of
\begin{align*}
\hat F\# 0: \R^{2n+2n+2n+k}&\to\R\\
(q,\zeta_1,\zeta_2,\nu_1) &\mapsto \hat F(\zeta_1,\nu_1) - 2\omega_\text{std}(\zeta_1-q,\zeta_2-1).
\end{align*}
Now pick any $x=(q,\zeta_1,\zeta_2,\nu_1)\in\R^{2n+2n+2n+k}$ such that $(\hat F\#0)(x)=0,$ i.e.
\begin{equation}
\label{eq:are-at-0}
\hat F(\zeta_1,\nu_1) = 2\omega_\text{std}(\zeta_1-q,\zeta_2-q).
\end{equation}
We need to find a tangent vector $\Delta x = (\Delta q,\Delta \zeta_1,\Delta \zeta_2,\Delta \nu_1)\in T_x\R^{6n+k}$ such that $d(\hat F\# 0)\;\Delta x\neq 0.$ Compute
\begin{equation}\label{eq:diff-sharpened}
d(\hat F\# 0)\;\Delta x = d \hat F (\Delta \zeta_1,\Delta \nu_1)
- 2\omega_\text{std}(\Delta\zeta_1-\Delta q,\zeta_2-q)
- 2\omega_\text{std}(\zeta_1-q,\Delta\zeta_2-\Delta q).
\end{equation}

We consider two cases: First assume $F_1(\zeta_1,\nu_1)\neq 0,$ set $\Delta \zeta_2=\zeta_2-q$ and all other components of $\Delta x$ to zero. It follows by Eqs.~\eqref{eq:are-at-0} and \eqref{eq:diff-sharpened} that $$
d(\hat F\# 0)\;\Delta x = -2\omega_\text{std}(\zeta_1-q,\zeta_2-1) = -\hat F(\zeta_1,\nu_1)\neq 0.
$$

Now assume instead that $F_1(\zeta_1,\nu_1)= 0.$ Since 0 is a regular value of $F_1$, there exist $(\Delta\zeta_1,\Delta\nu_1)$ such that $h:=d \hat F (\Delta \zeta_1,\Delta \nu_1)\neq 0.$ We again consider two cases: Assume first that $\omega_\text{std}(\Delta \zeta_1,\zeta_2-q)\neq h/2.$ Then we immediately get $d(\hat F\# 0)\;\Delta x\neq 0$ from Eq.~\eqref{eq:diff-sharpened} by setting $\Delta x= (0,\Delta \zeta_1,0,\Delta \nu_1).$ Assume instead that $\omega_\text{std}(\Delta \zeta_1,\zeta_2-q)=h/2.$ In particular, $\zeta_2-q\neq 0$ and we can pick a $\Delta q$ such that $\omega_\text{std}(\Delta q,\zeta_2-q)\neq 0.$ Now setting $\Delta\zeta_2=\Delta q$ in Eq.~\eqref{eq:diff-sharpened} yields$$
d(\hat F\# 0)\;\Delta x = h -h
- 2\omega_\text{std}(-\Delta q,\zeta_2-q)
- 2\omega_\text{std}(\zeta_1-q,0) = \omega_\text{std}(\Delta q,\zeta_2-q)\neq0,
$$
and we are done.

\end{proof}

\section{Existence of Translated Points on $S^{2n-1}$}

We are now ready to prove the main theorem and discuss the difference to Sandon's original statement:

\begin{theorem}\label{main-result-proven}
Let $\phi$ be a contactomorphism on $S^{2n-1}$ without degenerate translated points. Assume that $\phi$ has a generating function $F:S^{2n+k-1}\to\R$ such that the sublevel set $\{F\#0\leq0\}$ is either empty or an embedded submanifold with non-trivial homology.

Then $\phi$ has at least two translated points.
\end{theorem}

\begin{proof}

Let $A_t:S^{2n+m-1}\to\R$ be the smooth family of generating functions of the Reeb negative flow $a_t$ from Proposition~\ref{gen-reeb-props} and define the function
$$
G_t:= A_t \# F,
$$
which generates $a_t\circ\phi$ by Proposition~\ref{comp-form-sphere}.\medskip

By Definition~\ref{def:translated}, each translated point of $\phi$ corresponds to a discriminant point of $a_t\circ\phi$ for some $t$. Due to periodicity of $a_t,$ we can make this correspondence one-to-one by restricting $t$ to $[0,1).$ Per Proposition~\ref{prop-discr-sphere}, each of those corresponds one-to-one to a critical point of $G_t$ with value zero and by assumption all these critical points are non-degenerate.\medskip

To detect critical points, we consider the homology of the sublevel sets $$
N_t := \{G_t(x) \leq 0\} \qquad\text{for }t\in[0,1].
$$
If the points in time $t\in[0,1]$ where $G_t$ has a critical point with value zero were not isolated, we would clearly have more than two translated points and we would be done. Assuming instead that they are isolated, we want to show that $\partial_t G_t<0$ around every critical point of $G_t$ with value zero in order to apply Proposition~\ref{morse-main-stuff}. By homogeneity, this is equivalent to showing that $\partial_t \hat G_t<0$ around every non-zero critical point of $\hat G_t.$ Recall that$$
\hat G_t(q,\zeta_1,\zeta_2,\nu_1,\nu_2) = \hat A_t(\zeta_1,\nu_1)+\hat F(\zeta_2,\nu_2) - 2\omega_\text{std}(\zeta_1-q,\zeta_2-q).
$$
If $(q,\zeta_1,\zeta_2,\nu_1,\nu_2)$ is a non-zero critical point then it lies in $\Sigma_{\hat A_t \# \hat F}\setminus\{0\}.$ By the same argument as in part (iii) of the proof of Proposition~\ref{comp-form-sphere}, we have $(\zeta_1,\nu_1)\in\Sigma_{\hat A_t}\setminus\{0\}.$ With Proposition~\ref{gen-reeb-props}~(ii) and continuity, it follows that $\partial_t \hat G_t = \partial_t \hat A_t <0$ in a neighbourhood of each such critical point.
\medskip

Proposition~\ref{morse-main-stuff} now yields that $N_1$ arises from $N_0$ by attaching a cell whenever $t\in[0,1]$ passes a critical point of $G_t$ with value zero. Note that the sublevel sets $\{G_t\leq 0\}$ away from these $t$ with critical points are manifolds and by Theorem~\ref{triang} in particular CW complexes. It follows by Proposition~\ref{prop:betti} that exactly one Betti number changes whenever attaching a cell. By Proposition~\ref{homology-sublevelsets}~(iii), we know that at least two Betti numbers of $N_0$ and $N_1$ differ. It follows that we must attach at least two cells in this process, corresponding to at least two critical points of $G_t$ with value zero, i.e. translated points of $\phi$.
\end{proof}

The following proposition covers some of the distance between Theorem~\ref{main-result-proven} and Sandon's statement:

\begin{proposition}\label{assumption-fix}
Let $\phi$ be a contactomorphism on $S^{2n-1}$ that is contact isotopic to the identity. Then either $\phi$ has an infinite number of translated points, or there is some $t\in[0,1]$ such that $a_t\phi$ has a generating function $F:S^{2n+k-1}\to\R$ whose sublevel set $\{F\#0\leq 0\}$ is an embedded submanifold.
\end{proposition}

\begin{proof}
Pick some strictly increasing series $(t_i)_{i\in\N}$ of points in $(0,1).$ By Proposition~\ref{prop:existence-generating-sphere}, each of the contactomorphisms $a_{t_i} \phi$ has a generating function $F_i:S^{2n+k-1}\to\R.$ Assume first that 0 is a regular value of one of these, then Lemma~\ref{deform-f-zaun} gives that it is also a regular value of $F_i\#0.$ It follows that $\{F_i\#0\leq0\}$ is an embedded submanifold, and we would be done. Now assume instead that 0 is a singular value of all $F_i.$ By the same arguments as in the first two paragraphs of the proof of Theorem~\ref{main-result-proven}, this implies the existence of a translated point of $\phi$ for each $i\in\N.$

\end{proof}

\begin{remark}[Comparison to Sandon's original statement of the theorem]\label{remark-missing-assumption}
In light of Proposition~\ref{assumption-fix}, the critical assumption of Theorem~\ref{main-result-proven} that we cannot yet guarantee is the homological condition on $\{F\#0\leq 0\}$. The problem is exemplified by the following situation: Say for the sake of argument that $\{F\#0\leq 0\}$ was a contractible space. The proof of Theorem~\ref{main-result-proven} then fails at the step where we try to apply Proposition~\ref{homology-sublevelsets}~(iii). It still follows by the proof of this proposition that$$
N_t=\{A_t\#F\leq 0\}\simeq S^{\ind(\hat A_t)-1}*\{F\#0\leq 0\}.
$$
Since $\{F\#0\leq 0\}$ is homotopy equivalent to a point, it follows that $N_t$ is homotopy equivalent to a point as well, for both $t\in\{0,1\}.$ The proof of Theorem~\ref{main-result-proven} is built around the idea of finding translated points by detecting differences in the homology of $N_0$ and $N_1$, so there is a considerable problem in this situation.

Based on a suggestion by Sandon, one approach to fixing this is the following:
By using the Hamilton-Jacobi equation and composing a general generating function with the Reeb flow sufficiently many times, one can replace it with a generating function $F$ that is positive on its set of fiber-critical points. Since only that set determines what contactomorphism is generated, one can then attempt modifying $F$ on the complement to construct a positive generating function $F'=F+h$. However, in order to have $\Sigma_{F'}=\Sigma_F,$ the function $h$ must satisfy $\partial h/\partial\nu \neq -\partial F/\partial\nu$ whereever it doesn't vanish. So far, any attempt to construct such an $h$ large enough to make $F'$ positive seems to force passing through this derivative.

Following this, I have also attempted a simpler approach of constructing an $F'=F+h$ yielding a non-trivial homology of the sublevel set at zero through a local modification. This showed the same problem, however: To create a non-triviality in the sublevel set, one has to create a critical point of $F'$ restricted to the sphere. This should essentially force the derivative of $h$ to be that of $-F,$ breaking the inequality above. (Strictly speaking, it is more complicated since the inequality is about all of $F$ and $h$ instead of the restrictions to the sphere, but this should only shift the problematic point due to homogeneity.)

Another way to address this problem could be to modify the generating function $F$ we get from Proposition~\ref{assumption-fix} by composing it with some appropriate generating function $G$ of the identity such that $\{(F\# G)\#0\leq0\}$ satisfies the assumptions. Failing that, one could make a genericity argument by considering a carefully chosen series of functions $G_i$ whose generated contactomorphisms tend toward the identity. Assuming one can show that the assumptions are met for them, one can argue that this also yields at least one translated point for the original contactomorphism.

\end{remark}

%% file: appendix.tex
\chapter{Appendix}

\section{Notes on \cite{San13}}

This appendix gives an itemized summary of all the differences between this thesis and Sheila Sandon's \textit{A Morse estimate for translated points of contactomorphisms of spheres and projective spaces}~\cite{San13}. We consider in particular Sandon's section two (\textit{Preliminaries}), three (\textit{Homogeneous generating functions}) and four (\textit{Translated points for contactomorphisms of $S^{2n-1}$}). I am deeply grateful to Sandon for her help in understanding these points.

\subsection{Preliminaries}
\begin{enumerate}

\item Our Definition~\ref{def:translated} of non-degenerate translated points is more restrictive that Sandon's. While they coincide in the context of the main result, ours seems more natural and fits in better to with non-degeneracy of leafwise fixed-points (compare Remark~\ref{comparison-sandon-def}). 

\end{enumerate}

\subsection{Homogeneous generating functions}
\begin{enumerate}

	\item Sandon acknowledges that her definition of generating functions does not literally apply to lifts of contactomorphisms in her Remark~3.4. This is essentially a consequence of the lift failing to be smooth at the origin. In this thesis, we took great care to make the propositions applicable to the situation we use them in. This forced us to state and prove slightly more general statements throughout chapter three (compare Remark~\ref{remark-def-gen-euclidean}) and in particular led to the complicated statement of Definition~\ref{def-gen-on-sphere}.

	\item In her definition of a generating function $F$ for a symplectomorphism $\Phi$, Sandon only demands that the set $i_F(\Sigma_F)$ generated by it gets identified with the graph of $\Phi$, a Lagrangian submanifold. 
	However, we must demand at least that $i_F$ must be injective for the correspondence of critical points to be one-to-one. Indeed, if $F(x)$ is a generating function of $\Phi$, then $F'(x,\theta):=F(x)+\sin(\theta)$ can be seen to generate $\Phi$ as well since $i_F'(\Sigma_F')=i_F(\Sigma_F)$ by Lemma~\ref{euclidean-formulae}. But $F'$ has an infinite number of critical points for each critical point of $F.$  
	
	\item Sandon's Proposition 3.1 was likely meant to be formulated for symplectomorphisms that were lifted from a contactomorphism, not general Hamiltonian symplectomorphisms $\Phi$ on $\R^{2n}$: The Hamiltonian isotopy $\Phi_t$ associated with the latter can not generally be subdivided into $C^1$-small pieces on its non-compact domain. Instead we require the arguments of our Lemma~\ref{lemma-c2-small-existence}. Note also that the isotopy needs to be split into $C^2$-small pieces instead of $C^1$-small pieces as the derivative of the function $g$ in the lift incorporates second derivatives of the contactomorphism $\phi.$

	\item The inequality in Sandon's Lemma 3.6 should not be strict: At the coordinate origin, homogeneity of the family $F_t$ forces a vanishing time derivative. It does not suffice to only exclude this origin: While her argument in the $C^1-$small case would hold, the composition formula carries this problem away from the origin. This assessment seems consistent with the fact that Sandon and her coauthors formulate a weaker statement in Proposition 2.22 of the later~\cite{GKPS}. In the larger context of the main result, this weaker statement required us to state the assumptions of the parametric Morse theory approach more carefully and derive that the inequality holds strictly for the Reeb flow at least on $\Sigma_{\hat A_t}\setminus\{0\}$ (see Proposition~\ref{gen-reeb-props}~(iii)). This approach is adapted from~\cite{The98}.
	
	\item Instead of Sandon's Proposition~3.8, we give an explicit proof of the difference of indices based on our concrete choice of $A_t$ in Proposition~\ref{gen-reeb-props}~(iv).
\end{enumerate}

\subsection{Translated points for contactomorphisms of $S^{2n-1}$}

\begin{enumerate}

	\item All sublevel sets of functions on Euclidean space should refer to their restrictions to the unit sphere instead. Similarly, the distinction between cases where $F$ is positive everywhere or not should refer to the restriction of $F$ to the unit sphere. Otherwise, $F$ is never strictly positive as it must vanish in the origin by homogeneity.

	\item $\bar{A_0}$ should be defined on $\R^{2N'}$ instead of $\R^{2M}$, since $A_t$ is defined on $\R^{2n}\times\R^{2N'}.$

	\item Strictly speaking, the equalities $A_0\circ \Psi_1^0=\bar A_0$ and $A_1\circ \Psi_1^1=\bar A_1$ are missing projection maps.

	\item Sandon suggests that equivalent generating functions have homotopy equivalent sublevel sets. Care should be taken here since she does not define equivalence of generating functions. Common definitions from the literature allow for the addition of constants, which would make this statement false. Here, she likely only allows for the application of a diffeomorphism in the domain and stabilization. Lemma~4.8 from~\cite{The98} provides a stronger statement than she quotes, however, and immediately gives homotopy equivalence of sublevel sets in this situation.

	\item Sandon does not give any justification for the formula $\{\bar A_0\# F\leq 0\}\simeq\{\bar A_0\leq 0\}$ for the case $F$ positive, and $\{\bar A_0\# F\leq 0\}\simeq\{\bar A_0\leq 0\} * \{F\leq 0\}$ for $F$ general. The later preprint~\cite{GKPS} gives some context to these formulae with their Proposition~3.14 and the arguments at the end of page~31. These establish only an equality of their Maslov index on lens spaces of the involved sublevel sets. We derived the stronger statements applicable to our situation mostly in Lemma~\ref{lemma-oplus-join}, which crucially builds on communication with Alexander Givental on his Proposition~B.1 in~\cite{Giv90}.
	
	\item Sandon's definition of the join strictly speaking leads to an error: She does not exclude the case $F$ strictly positive for the general case, which would imply $\{F\leq 0\}=\emptyset$ and $\{\bar A_0\# F\leq 0\}\simeq\{\bar A_0\leq 0\} * \{F\leq 0\} = \emptyset.$ This leads to problems as the rest of the argument assumes non-empty spaces in order to use constructions of pointed spaces. This is easy to fix, however: Either one can explicitly exclude the case that $F$ is strictly positive as it was already considered separately, or one can choose a more natural definition of the join as we do (compare our Remark~\ref{join-remark}). 
	
	\item Sandon hints at the required parametric Morse theory statements, but gives very little details. In the later preprint~\cite{GKPS}, she and her coauthors only give an argument for the case that $a$ is a regular value for all times with her Lemma~4.14. We provide the singular case in our Lemma~\ref{lem-morse-hard} by adapting a proof due to Milnor~\cite{milnor-morse}.

	\item Similarly, Sandon only provides the idea of how to use a Kuenneth formula to detect cell attachments through homology. We filled out the details by considering specifically the change in Betti numbers by our Propositions~\ref{prop:betti}   and~\ref{homology-sublevelsets}. This gap is what allowed the assumption on non-trivial homology of $\{F\#0\leq0\}$ to go unnoticed (see Remark~\ref{remark-missing-assumption} for details).

\end{enumerate}

\section{Lifting Contactomorphisms from $\R P^{2n-1}$ to $\R^{2n}$}\label{sec:proj-lifting}

Sandon also gives a proof of the contact Arnol'd conjecture for real projective space in~\cite{San13} with a very similar approach, where $\R P^{2n-1}$ is equipped with the standard contact form $\tilde\alpha$ from Remark~\ref{standard-struct-proj}. Analogously to Section~\ref{sec:lifting} for the sphere, this appendix provides the foundation for Sandon's proof with a way to lift contactomorphisms from real projective space to Euclidean space.\medskip

It is based on the discussion of lifting in~\cite{San13}, but spells out the details of the proof.

\begin{proposition}[Lifting contactomorphisms and contact isotopies from $\R P^{2n-1}$ to $S^{2n-1}$]\label{lift-proj}
	Let $\phi$ be a contactomorphism and $\phi_t$ a contact isotopy on $(\R P^{2n-1},\tilde{\alpha}).$

\begin{enumerate}
	\item We can lift $\phi$ to a contactomorphism $\hat{\phi}$ on $(S^{2n-1},\alpha)$ such that $\pi\hat{\phi}=\phi\pi$, where $\pi$ is the canonical projection from the sphere to projective space.

	\item We can lift $\phi_t$ to a unique contact isotopy $\hat{\phi}_t$ on $(S^{2n-1},\alpha)$ by requiring in addition to $\pi\hat\phi_t=\phi_t\pi$ that the lift starts at the identity. If $\phi_t$ is generated by the Hamiltonian function $H_t$, then $\hat{\phi}_t$ is generated by $\hat{H}_t:=H_t\circ\pi$.

\item Each (non-degenerate) translated point of $\phi$ corresponds precisely to two (non-degenerate) translated points of $\hat{\phi},$ namely those which project to it under $\pi$.

\end{enumerate}
\end{proposition}

Note that the lift $\hat\phi$ from part (i) is not unique since its composition with the antipodal map descends to the same $\phi$. Part (iii) does not hold for discriminant points, i.e. their preimages are not necessarily discriminant points themselves: The lift of $\phi$ may permute them so that they fail to be fixed points.

\begin{proof}
	Regarding (i): Note that $S^{2n-1}$ is a double cover of $\R P^{2n-1}$. We can then apply a smooth version of the lifting Theorem~3.5.2 in \cite{dieck} to the composition $f=\phi\pi$ since $\phi$ is a diffeomorphism and the push-forwards of the fundamental groups are therefore the same. Since the contact form of $\R P^{2n-1}$ is defined as the pushforward along the covering map, this lift preserves it. Similarly lifting the inverse of $\phi$ must, up to applying the antipodal map, yield the inverse of $\hat\phi$ such that it is a contactomorphism.

	Regarding (ii): Since $\pi:S^{2n-1}\to \R P^{2n-1}$ is a smooth fibration, it satisfies a homotopy lifting property (see e.g. Section 5.5 of~\cite{dieck}). The lifted homotopy is smooth and unique since $\pi$ is even a smooth covering space. We can apply this to the homotopy $\phi_t\pi$ to get $\hat\phi_t$, requiring the lifted homotopy to start at the identity. By the same arguments of part (i), each $\phi_t$ is a contactomorphism.

	We still need to show that the isotopy $\hat\phi_t$ is generated by the contact Hamiltonian $\hat H_t$. To see this, note that differentiating $\pi\hat{\phi}_t=\phi\pi_t$ shows that $\hat\phi_t$ is generated by the vector field $$\hat X_t:= d\pi^{-1} X_t\circ\pi,$$ where $X_t$ is the generating field of $\phi_t$ and we used that $\pi$ is a local diffeomorphism. Using the equations from Def. \& Prop.~\ref{contact-hamiltonian} that characterize $X_t$ as generated by $H_t$ and inserting the definitions of $\hat X_t, \hat H_t$ and $\alpha=\pi^*\tilde\alpha$ yields that $\hat H_t$ indeed generates $\hat X_t.$

	Regarding (iii): By part (ii), the Reeb flow $\R^\alpha$ on $S^{2n-1}$ is a lift of the Reeb flow $R^{\tilde\alpha}$ on $\R P^{2n-1},$ i.e. $$\pi R^\alpha_t = R^{\tilde\alpha}_t \pi.$$

	If $p$ is a translated point of $\hat\phi$, we can apply $\pi$ to $R^\alpha_t\hat\phi(p)=p$ to get
	$$\pi(p) = \pi R^\alpha_t\hat\phi(p) = R^{\tilde\alpha}_t\pi\hat\phi(p) = R^{\tilde\alpha}_t\phi(\pi(p)).$$

	If conversely we have a translated point $\pi(p),$ we can run the same argument backwards to see that $\hat\phi(p)$ is in the same Reeb orbit as either $p$ or $-p$. But since the Reeb flow on $S^{2n-1}\subseteq\C^n$ is given by $z\mapsto \exp(it)z,$ we can conclude that there exists some $t$ such that $R^\alpha_t\hat\phi(p)=p.$ The same argument applies to $-p$, giving two candidates for translated points corresponding to $\pi(p).$

	The remaining criteria for translated points and their degeneracy are all local and thus the proof is completed in both directions by the fact that $\pi$ is a local diffeomorphism.
\end{proof}

Subsequently applying Proposition~\ref{lift-sphere} to the contactomorphism on the sphere yields symplectomorphisms on $\R^{2n-1}$.

%% file: bibliography.bib
@article{Hir61,
	author = {Hirsch, Morris W.},
	doi = {10.2307/1970318},
	issn = {0003-486X},
	journal = {Annals of Mathematics. Second Series},
	mrclass = {57.20},
	mrnumber = {124915},
	mrreviewer = {M. L. Curtis},
	pages = {566--571},
	title = {On imbedding differentiable manifolds in euclidean space},
	volume = {73},
	x-fetchedfrom = {MR Lookup},
	year = {1961}
}

@article{San12,
	author = {Sandon, Sheila},
	journal = {International Journal of Mathematics},
	number = {02},
	pages = {1250042},
	publisher = {World Scientific},
	title = {On iterated translated points for contactomorphisms of 2n+ 1 and 2n$\times$ S1},
	volume = {23},
	x-fetchedfrom = {Google Scholar},
	year = {2012}
}

@article{San13,
	author = {Sandon, Sheila},
	journal = {Geometriae Dedicata},
	number = {1},
	pages = {95--110},
	publisher = {Springer},
	title = {A Morse estimate for translated points of contactomorphisms of spheres and projective spaces},
	volume = {165},
	x-fetchedfrom = {Google Scholar},
	year = {2013}
}

@misc{San14,
	author = {Sandon, Sheila},
	note = {Lecture notes for the CIMPA research school on geometric methods in classical dynamical systems, Santiago 2014},
	title = {Generating functions in Symplectic Topology},
	x-fetchedfrom = {Google Scholar},
	year = {2014}
}

@book{MS17,
	author = {McDuff, Dusa and Salamon, Dietmar},
	doi = {10.1093/oso/9780198794899.001.0001},
	edition = {Third Edition},
	isbn = {978-0-19-879490-5; 978-0-19-879489-9},
	mrclass = {53D35 (53D40 57R17 57R57 57R58)},
	mrnumber = {3674984},
	mrreviewer = {Hansj{\"o}rg Geiges},
	pages = {xi+623},
	publisher = {Oxford University Press, Oxford},
	series = {Oxford Graduate Texts in Mathematics},
	title = {Introduction to symplectic topology},
	x-fetchedfrom = {MathSciNet},
	year = {2017}
}

@article{Zil17,
	author = {Ziltener, Fabian},
	doi = {10.1093/imrn/rnx182},
	issn = {1073-7928},
	journal = {International Mathematics Research Notices. IMRN},
	mrclass = {37J05 (37C25)},
	mrnumber = {3942166},
	number = {8},
	pages = {2411--2452},
	title = {Leafwise fixed points for {$C^0$}-small {H}amiltonian flows},
	volume = {Vol. 2019},
	x-fetchedfrom = {MR Lookup},
	year = {2017}
}

@article{Zil10,
	author = {Ziltener, Fabian},
	issn = {1527-5256},
	journal = {The Journal of Symplectic Geometry},
	mrclass = {53D35 (53D12)},
	mrnumber = {2609631},
	mrreviewer = {Chris M. Wendl},
	number = {1},
	pages = {95--118},
	title = {Coisotropic submanifolds, leaf-wise fixed points, and presymplectic embeddings},
	url = {http://projecteuclid.org/euclid.jsg/1271166377},
	volume = {8},
	x-fetchedfrom = {MathSciNet},
	year = {2010}
}

@article{The98,
	author = {Th{\'e}ret, David},
	doi = {10.1215/S0012-7094-98-09402-9},
	issn = {0012-7094},
	journal = {Duke Mathematical Journal},
	mrclass = {58F05 (57S25)},
	mrnumber = {1635892},
	mrreviewer = {Karl Friedrich Siburg},
	number = {1},
	pages = {13--27},
	title = {Rotation numbers of {H}amiltonian isotopies in complex projective spaces},
	volume = {94},
	x-fetchedfrom = {MR Lookup},
	year = {1998}
}

@article{Hoe71,
	author = {H{\"o}rmander, Lars},
	doi = {10.1007/BF02392052},
	issn = {0001-5962},
	journal = {Acta Mathematica},
	mrclass = {58G15 (35S05 47G05)},
	mrnumber = {388463},
	mrreviewer = {Yu. V. Egorov},
	number = {1-2},
	pages = {79--183},
	title = {Fourier integral operators. {I}},
	volume = {127},
	x-fetchedfrom = {MathSciNet},
	year = {1971}
}

@article{Vit92,
	author = {Viterbo, Claude},
	doi = {10.1007/BF01444643},
	issn = {0025-5831},
	journal = {Mathematische Annalen},
	mrclass = {58F05 (53C15 53C23 57R50 58E05)},
	mrnumber = {1157321},
	mrreviewer = {Nikolai K. Smolentsev},
	number = {4},
	pages = {685--710},
	title = {Symplectic topology as the geometry of generating functions},
	volume = {292},
	x-fetchedfrom = {MR Lookup},
	year = {1992}
}

@incollection{Giv90,
	author = {Givental, A. B.},
	booktitle = {Theory of singularities and its applications},
	mrclass = {58F05},
	mrnumber = {1089671},
	mrreviewer = {Yong-Geun Oh},
	pages = {71--103},
	publisher = {Amer. Math. Soc., Providence, RI},
	series = {Adv. Soviet Math.},
	title = {Nonlinear generalization of the {M}aslov index},
	volume = {1},
	x-fetchedfrom = {MR Lookup},
	year = {1990}
}

@article{Arn65,
	author = {Arnold, Vladimir},
	journal = {CR. Acad. Sci. Paris},
	pages = {3719--3722},
	title = {Sur une propri{\'e}t{\'e}s topologique des applications globalment canonique de la m{\'e}chanique classique},
	volume = {261},
	year = {1965}
}

@phdthesis{The96,
	author = {Th{\'e}ret, David},
	note = {Th{\`e}se de doctorat dirig{\'e}e par Chaperon, Marc Math{\'e}matiques Paris 7 1996},
	note2 = {1996PA077140},
	pages = {1 vol. (117 P.)},
	title = {Utilisation des fonctions generatrices en geometrie symplectique globale},
	url = {http://www.theses.fr/1996PA077140},
	year = {1996}
}

@article{GKPS,
	author = {Granja, Gustavo and Karshon, Yael and Pabiniak, Milena and Sandon, Sheila},
	journal = {arXiv preprint arXiv:1704.05827v3},
	note = {Version 3},
	title = {Givental's non-linear Maslov index on lens spaces},
	x-fetchedfrom = {Google Scholar},
	year = {2017}
}

@article{Wei81,
	author = {Weinstein, Alan},
	doi = {10.1090/S0273-0979-1981-14911-9},
	issn = {0273-0979},
	journal = {American Mathematical Society. Bulletin. New Series},
	mrclass = {58F05 (70G15)},
	mrnumber = {614310},
	mrreviewer = {Gunther A. Uhlmann},
	number = {1},
	pages = {1--13},
	title = {Symplectic geometry},
	volume = {5},
	x-fetchedfrom = {MR Lookup},
	year = {1981}
}

@book{Gei08,
	author = {Geiges, Hansj{\"o}rg},
	doi = {10.1017/CBO9780511611438},
	isbn = {978-0-521-86585-2},
	mrclass = {57R17 (53D35)},
	mrnumber = {2397738},
	mrreviewer = {John B. Etnyre},
	pages = {xvi+440},
	publisher = {Cambridge University Press, Cambridge},
	series = {Cambridge Studies in Advanced Mathematics},
	title = {An introduction to contact topology},
	volume = {109},
	x-fetchedfrom = {MR Lookup},
	year = {2008}
}

@book{Can03,
	author = {{Cannas da Silva}, Ana},
	doi = {10.1007/978-3-0348-8071-8},
	isbn = {85-244-0195-8},
	mrclass = {53Dxx (37J05 37J35 70G45)},
	mrnumber = {2115646},
	mrreviewer = {Catalin Zara},
	pages = {x+130},
	publisher = {Instituto de Matem{\'a}tica Pura e Aplicada (IMPA), Rio de Janeiro},
	series = {Publica{\c c}{\~o}es Matem{\'a}ticas do IMPA. [IMPA Mathematical Publications]},
	title = {Introduction to symplectic and {H}amiltonian geometry},
	x-fetchedfrom = {MR Lookup},
	year = {2003}
}

@misc{Ter18,
	archiveprefix = {arXiv},
	author = {Tervil, Brian},
	comment = {published = 2018-11-25T10:29:36Z, updated = 2019-03-19T15:02:31Z, Corrected typos; added discussion in the introduction, explanations, and references},
	eprint = {1811.09984v2},
	month = mar,
	primaryclass = {math.SG},
	title = {Translated points for prequantization spaces over monotone toric manifolds},
	url = {http://arxiv.org/abs/1811.09984v2; http://arxiv.org/pdf/1811.09984v2},
	x-fetchedfrom = {arXiv.org},
	year = {2018}
}

@article{She17,
	author = {Shelukhin, Egor},
	doi = {10.4310/JSG.2017.v15.n4.a7},
	issn = {1527-5256},
	journal = {The Journal of Symplectic Geometry},
	mrclass = {53D10 (37J05 57R17 57S05)},
	mrnumber = {3734612},
	mrreviewer = {Philippe Rukimbira},
	number = {4},
	pages = {1173--1208},
	title = {The {H}ofer norm of a contactomorphism},
	volume = {15},
	x-fetchedfrom = {MR Lookup},
	year = {2017}
}

@article{MN18,
	author = {Meiwes, Matthias and Naef, Kathrin},
	journal = {Journal of Topology and Analysis},
	number = {02},
	pages = {289--322},
	publisher = {World Scientific},
	title = {Translated points on hypertight contact manifolds},
	volume = {10},
	x-fetchedfrom = {Google Scholar},
	year = {2018}
}

@article{AM13,
	author = {Albers, Peter and Merry, Will J},
	journal = {Journal of Fixed Point Theory and Applications},
	number = {1},
	pages = {201--214},
	publisher = {Springer},
	title = {Translated points and Rabinowitz Floer homology},
	volume = {13},
	year = {2013}
}

@article{AFM15,
	author = {Albers, Peter and Fuchs, Urs and Merry, Will J},
	journal = {Compositio Mathematica},
	number = {12},
	pages = {2251--2272},
	publisher = {London Mathematical Society},
	title = {Orderability and the Weinstein conjecture},
	volume = {151},
	year = {2015}
}

@article{EH89,
	author = {Ekeland, Ivar and Hofer, Helmut},
	journal = {Journal de math{\'e}matiques pures et appliqu{\'e}es},
	number = {4},
	pages = {467--489},
	publisher = {Elsevier},
	title = {Two symplectic fixed-point theorems with applications to Hamiltonian dynamics},
	volume = {68},
	year = {1989}
}

@article{Gin07,
	author = {Ginzburg, Viktor L and others},
	journal = {Duke Mathematical Journal},
	number = {1},
	pages = {111--163},
	publisher = {Duke University Press},
	title = {Coisotropic intersections},
	volume = {140},
	year = {2007}
}

@article{Mos78,
	author = {Moser, J{\"u}rgen},
	journal = {Acta Mathematica},
	number = {1},
	pages = {17--34},
	publisher = {Springer},
	title = {A fixed point theorem in symplectic geometry},
	volume = {141},
	year = {1978}
}

@article{Dra08,
	author = {Dragnev, Dragomir L.},
	doi = {10.1002/cpa.20203},
	issn = {0010-3640},
	journal = {Communications on Pure and Applied Mathematics},
	mrclass = {53D40 (37J10)},
	mrnumber = {2376845},
	mrreviewer = {Stefan Haller},
	number = {3},
	pages = {346--370},
	title = {Symplectic rigidity, symplectic fixed points, and global perturbations of {H}amiltonian systems},
	volume = {61},
	x-fetchedfrom = {MR Lookup},
	year = {2008}
}

@article{Ban80,
	author = {Banyaga, Augustin},
	doi = {10.1007/BF01390045},
	issn = {0020-9910},
	journal = {Inventiones Mathematicae},
	mrclass = {58F05},
	mrnumber = {561971},
	mrreviewer = {A. Crumeyrolle},
	number = {3},
	pages = {215--229},
	title = {On fixed points of symplectic maps},
	volume = {56},
	x-fetchedfrom = {MR Lookup},
	year = {1980}
}

@article{Gue10,
	author = {G{\"u}rel, Ba{\c s}ak Zehra},
	doi = {10.1093/imrn/rnp164},
	issn = {1073-7928},
	journal = {International Mathematics Research Notices. IMRN},
	mrclass = {53D12 (53D35)},
	mrnumber = {2595016},
	mrreviewer = {Hai-Long Her},
	number = {5},
	pages = {914--931},
	title = {Leafwise coisotropic intersections},
	x-fetchedfrom = {MR Lookup},
	year = {2010}
}

@article{Kan12,
	author = {Kang, Jungsoo},
	doi = {10.1007/s11856-011-0184-4},
	issn = {0021-2172},
	journal = {Israel Journal of Mathematics},
	mrclass = {53D10 (53C12 53D05 53D40)},
	mrnumber = {2956235},
	mrreviewer = {Jelena Kati{\'c}},
	pages = {111--134},
	title = {Existence of leafwise intersection points in the unrestricted case},
	volume = {190},
	x-fetchedfrom = {MR Lookup},
	year = {2012}
}

@article{Ker08,
	author = {Kerman, Ely},
	doi = {10.3934/jmd.2008.2.471},
	issn = {1930-5311},
	journal = {Journal of Modern Dynamics},
	mrclass = {53D40 (37J05 53D35)},
	mrnumber = {2417482},
	mrreviewer = {Tobias Ekholm},
	number = {3},
	pages = {471--497},
	title = {Displacement energy of coisotropic submanifolds and {H}ofer's geometry},
	volume = {2},
	x-fetchedfrom = {MR Lookup},
	year = {2008}
}

@article{Mer11,
	author = {Merry, Will J.},
	doi = {10.1007/s00526-011-0391-1},
	issn = {0944-2669},
	journal = {Calculus of Variations and Partial Differential Equations},
	mrclass = {53D40},
	mrnumber = {2846260},
	mrreviewer = {Vincent Humili{\`e}re},
	number = {3-4},
	pages = {355--404},
	title = {On the {R}abinowitz {F}loer homology of twisted cotangent bundles},
	volume = {42},
	x-fetchedfrom = {MR Lookup},
	year = {2011}
}

@article{AF12,
	author = {Albers, Peter and Frauenfelder, Urs},
	doi = {10.1016/j.exmath.2012.01.005},
	journal = {Expositiones Mathematicae},
	month = dec,
	number = {2},
	pages = {168--181},
	pii = {S0723086912000060},
	title = {Infinitely many leaf-wise intersections on cotangent bundles},
	url = {https://www.sciencedirect.com/science/article/pii/S0723086912000060},
	volume = {30},
	x-fetchedfrom = {ScienceDirect},
	year = {2012}
}

@article{AF10,
	author = {Albers, Peter and Frauenfelder, Urs},
	journal = {Journal of Topology and Analysis},
	number = {01},
	pages = {77--98},
	publisher = {World Scientific},
	title = {Leaf-wise intersections and Rabinowitz Floer homology},
	volume = {2},
	year = {2010}
}

@article{AF10b,
	archiveprefix = {arXiv},
	author = {Albers, Peter and Frauenfelder, Urs},
	comment = {published = 2010-01-17T19:43:19Z, updated = 2010-01-17T19:43:19Z, 29 pages},
	doi = {10.3934/jmd.2010.4.329},
	eprint = {1001.2920v1},
	month = jan,
	pages = {329--357},
	primaryclass = {math.SG},
	title = {Spectral Invariants in Rabinowitz Floer homology and Global Hamiltonian perturbations},
	url = {http://arxiv.org/abs/1001.2920v1; http://arxiv.org/pdf/1001.2920v1},
	x-fetchedfrom = {arXiv.org},
	year = {2010}
}

@article{AF12b,
	author = {Albers, Peter and Frauenfelder, Urs},
	journal = {Israel Journal of Mathematics},
	number = {1},
	pages = {485--491},
	publisher = {Springer},
	title = {On a theorem by Ekeland-Hofer},
	volume = {187},
	year = {2012}
}

@article{AM10,
	archiveprefix = {arXiv},
	author = {Albers, Peter and Momin, Al},
	comment = {published = 2010-02-17T15:32:49Z, updated = 2010-05-09T19:20:08Z, 13 pages, 4 figures; v2: minor modifications},
	doi = {10.1017/S0305004110000435},
	eprint = {1002.3283v2},
	month = may,
	pages = {539--551},
	primaryclass = {math.SG},
	title = {Cup-length estimates for leaf-wise intersections},
	url = {http://arxiv.org/abs/1002.3283v2; http://arxiv.org/pdf/1002.3283v2},
	x-fetchedfrom = {arXiv.org},
	year = {2010}
}

@article{AM11,
	author = {Albers, Peter and McLean, Mark and others},
	journal = {Journal of Symplectic Geometry},
	number = {3},
	pages = {271--284},
	publisher = {International Press of Boston},
	title = {Non-displaceable contact embeddings and infinitely many leaf-wise intersections},
	volume = {9},
	year = {2011}
}

@article{Ch96,
	author = {Chekanov, Yu. V.},
	doi = {10.1007/BF02509451},
	issn = {1573-8485},
	journal = {{Functional Analysis and Its Applications}},
	number = {2},
	pages = {118--128},
	publisher = {Springer},
	title = {Critical points of quasi-functions and generating families of Legendrian manifolds},
	volume = {30},
	x-fetchedfrom = {SpringerLink},
	year = {1996}
}

@article{Hof90,
	author = {Hofer, H.},
	doi = {10.1017/S0308210500024549},
	issn = {0308-2105},
	journal = {Proceedings of the Royal Society of Edinburgh. Section A. Mathematics},
	mrclass = {58F05 (70H05)},
	mrnumber = {1059642},
	mrreviewer = {Yong-Geun Oh},
	number = {1-2},
	pages = {25--38},
	title = {On the topological properties of symplectic maps},
	volume = {115},
	x-fetchedfrom = {MR Lookup},
	year = {1990}
}

@book{milnor-morse,
	author = {Milnor, John},
	mrclass = {57.50 (53.72)},
	mrnumber = {0163331},
	mrreviewer = {H. I. Levine},
	pages = {vi+153},
	publisher = {Princeton University Press, Princeton, N.J.},
	series = {Based on lecture notes by M. Spivak and R. Wells. Annals of Mathematics Studies, No. 51},
	title = {Morse theory},
	x-fetchedfrom = {MR Lookup},
	year = {1963}
}

@book{dieck,
	author = {tom Dieck, Tammo},
	publisher = {European Mathematical Society},
	title = {Algebraic Topology},
	volume = {8},
	x-fetchedfrom = {Google Scholar},
	year = {2008}
}

@article{Whi56,
	author = {Whitehead, George W.},
	doi = {10.2307/1992905},
	issn = {0002-9947},
	journal = {Transactions of the American Mathematical Society},
	mrclass = {55.0X},
	mrnumber = {80918},
	mrreviewer = {W. S. Massey},
	pages = {55--69},
	title = {Homotopy groups of joins and unions},
	volume = {83},
	x-fetchedfrom = {MR Lookup},
	year = {1956}
}

@book{ruiz,
	author = {Outerelo, Enrique and Ruiz, Jes{\'u}s M.},
	doi = {10.1090/gsm/108},
	isbn = {978-0-8218-4915-6},
	mrclass = {55M25 (01A55 01A60 57R35)},
	mrnumber = {2566906},
	mrreviewer = {Jean Mawhin},
	pages = {x+244},
	publisher = {American Mathematical Society Providence, RI; Real Sociedad Matem{\'a}tica Espa{\~n}ola Madrid},
	series = {Graduate Studies in Mathematics},
	title = {Mapping degree theory},
	volume = {108},
	x-fetchedfrom = {MR Lookup},
	year = {2009}
}

@book{Bla10,
	author = {Blair, David E.},
	doi = {10.1007/978-0-8176-4959-3},
	edition = {Second},
	isbn = {978-0-8176-4958-6},
	mrclass = {53D10 (53C25 53D05)},
	mrnumber = {2682326},
	mrreviewer = {Joeri Van der Veken},
	pages = {xvi+343},
	publisher = {Birkh{\"a}user Boston Ltd., Boston MA},
	series = {Progress in Mathematics},
	title = {Riemannian geometry of contact and symplectic manifolds},
	volume = {203},
	x-fetchedfrom = {MR Lookup},
	year = {2010}
}

@book{Hat02,
	author = {Hatcher, Allen},
	isbn = {0-521-79160-X; 0-521-79540-0},
	mrclass = {55-01 (55-00)},
	mrnumber = {1867354},
	mrreviewer = {Donald W. Kahn},
	pages = {xii+544},
	publisher = {Cambridge University Press, Cambridge},
	title = {Algebraic topology},
	x-fetchedfrom = {MR Lookup},
	year = {2002}
}

@article{whitehead-triangulation,
	author = {Whitehead, J. H. C.},
	doi = {10.2307/1968861},
	issn = {0003-486X},
	journal = {Annals of Mathematics. Second Series},
	mrclass = {56.0X},
	mrnumber = {2545},
	mrreviewer = {R. H. Fox},
	pages = {809--824},
	title = {On {$C^1$}-complexes},
	volume = {41},
	x-fetchedfrom = {MR Lookup},
	year = {1940}
}

@incollection{RV06,
	author = {Rote, G{\"u}nter and Vegter, Gert},
	booktitle = {Effective Computational Geometry for curves and surfaces},
	pages = {277--312},
	publisher = {Springer},
	title = {Computational topology: An introduction},
	x-fetchedfrom = {Google Scholar},
	year = {2006}
}

@book{AD14,
	author = {Audin, Mich{\`e}le and Damian, Mihai},
	doi = {10.1007/978-1-4471-5496-9},
	isbn = {978-1-4471-5495-2; 978-1-4471-5496-9; 978-2-7598-0704-8},
	mrclass = {53-02 (53D40 58E05)},
	mrnumber = {3155456},
	mrreviewer = {Sonja Hohloch},
	note = {Translated from the 2010 French original by Reinie Ern{\'e}},
	pages = {xiv+596},
	publisher = {Springer, London; EDP Sciences, Les Ulis},
	series = {Universitext},
	title = {Morse theory and {F}loer homology},
	x-fetchedfrom = {MR Lookup},
	year = {2014}
}

@incollection{Sal99,
	author = {Salamon, Dietmar},
	booktitle = {Symplectic geometry and topology ({P}ark {C}ity, {UT}, 1997)},
	doi = {10.1016/S0165-2427(99)00127-0},
	mrclass = {53D40 (37J45 53D45 57R17 57R58)},
	mrnumber = {1702944},
	mrreviewer = {David E. Hurtubise},
	pages = {143--229},
	publisher = {Amer. Math. Soc., Providence, RI},
	series = {IAS/Park City Math. Ser.},
	title = {Lectures on {F}loer homology},
	volume = {7},
	x-fetchedfrom = {MR Lookup},
	year = {1999}
}

@book{CLOT03,
	author = {Cornea, Octav and Lupton, Gregory and Oprea, John and Tanr{\'e}, Daniel},
	doi = {10.1090/surv/103},
	isbn = {0-8218-3403-5},
	mrclass = {55M30 (53D35 55P62 55Q25 57R17)},
	mrnumber = {1990857},
	mrreviewer = {Samuel B. Smith},
	pages = {xviii+330},
	publisher = {American Mathematical Society, Providence, RI},
	series = {Mathematical Surveys and Monographs},
	title = {Lusternik-{S}chnirelmann category},
	volume = {103},
	x-fetchedfrom = {MR Lookup},
	year = {2003}
}

@book{Mat02,
	author = {Matsumoto, Yukio},
	isbn = {0-8218-1022-7},
	mrclass = {57R70 (57-01 57M25)},
	mrnumber = {1873233},
	mrreviewer = {James F. Reineck},
	note = {Translated from the 1997 Japanese original by Kiki Hudson and Masahico Saito, Iwanami Series in Modern Mathematics},
	pages = {xiv+219},
	publisher = {American Mathematical Society, Providence, RI},
	series = {Translations of Mathematical Monographs},
	title = {An introduction to {M}orse theory},
	volume = {208},
	x-fetchedfrom = {MR Lookup},
	year = {2002}
}

@article{MU19,
	author = {Merry, Will J. and Uljarevic, Igor},
	doi = {10.1007/s11856-018-1792-z},
	issn = {0021-2172},
	journal = {Israel Journal of Mathematics},
	mrclass = {53D40 (53D35)},
	mrnumber = {3905596},
	number = {1},
	pages = {39--65},
	title = {Maximum principles in symplectic homology},
	volume = {229},
	x-fetchedfrom = {MR Lookup},
	year = {2019}
}
